\documentclass[11pt]{amsart}
\usepackage{amsmath, amsthm, amscd, amsfonts,amssymb, graphicx}
\usepackage{caption}
\usepackage{siunitx}
\usepackage{subcaption}
\usepackage{float}
\usepackage{accents}
\usepackage{algorithm} 
\usepackage{algpseudocode} 
\usepackage{xcolor}
\usepackage{hyperref}

\newlength{\dhatheight}

\newcommand{\bea}{\begin{eqnarray*}}
\newcommand{\eea}{\end{eqnarray*}}

\newcommand{\bfbeta}{\mbox{\boldmath $\beta$ \unboldmath} \hskip -0.05 true in}

\newcommand{\beq}{\begin{equation}}
\newcommand{\eeq}{\end{equation}}
\newcommand{\bfomega}{\mbox{\boldmath $\omega$ \unboldmath} \hskip -0.05 true in}
\newcommand{\bfell}{\mbox{\boldmath $\ell$ \unboldmath} \hskip -0.05 true in}

 \newcommand{\D}{\mathrm{d}}
 \newcommand{\ii}{\mathrm{i}}
\newcommand{\dd}{{\rm d}}
\newcommand{\x}{\mathrm{x}}
\newcommand{\y}{\mathrm{y}}
\newcommand{\rt}{\mathrm{r}}

\newtheorem{theorem}{Theorem}[section]
\newtheorem{lemma}[theorem]{Lemma}
\theoremstyle{definition}

\newtheorem{example}[theorem]{Example}

\newtheorem{proposition}[theorem]{Proposition}
\newtheorem{corollary}[theorem]{Corollary}

\theoremstyle{remark}
\newtheorem{remark}[theorem]{Remark}

\numberwithin{equation}{section}

\begin{document}

\title[Fast Convolutions via Radial Translational Dependence and FFT on $\mathbb{Z}^2\backslash SE(2)$]{Fast Convolutions on $\mathbb{Z}^2\backslash SE(2)$ via Radial Translational Dependence and Classical FFT }

\author[A. Ghaani Farashahi]{Arash Ghaani Farashahi$^{*,1}$}
\address{$^{1}$Department of Mechanical Engineering, University of Delaware, USA.}
\email{aghf@udel.edu}
\email{ghaanifarashahi@outlook.com}

\author[G.S. Chirikjian]{Gregory S. Chirikjian$^{2}$}
\address{$^2$Department of Mechanical Engineering, University of Delaware, USA}
\email{gchirik@udel.edu}
\address{$^2$Department of Mechanical Engineering, National University of Singapore, Singapore.}
\email{mpegre@nus.edu.sg}


\thanks{$^*$Corresponding author}
\date{\today}

\begin{abstract}
Let $\mathbb{Z}^2\backslash SE(2)$ denote the right coset space of the subgroup consisting of translational isometries of the orthogonal lattice $\mathbb{Z}^2$ in the non-Abelian group of planar motions $SE(2)$. This paper develops a fast and accurate numerical scheme for approximation of functions on $\mathbb{Z}^2\backslash SE(2)$. We address finite Fourier series of functions on the right coset space $\mathbb{Z}^2\backslash SE(2)$ using finite Fourier coefficients. The convergence/error analysis of finite Fourier coefficients are investigated. Conditions are established for the finite Fourier coefficients to converge to the Fourier coefficients. The matrix forms of the finite transforms are discussed. The implementation of the discrete method to compute numerical approximation of $SE(2)$-convolutions with functions which are radial in translations are considered. The paper is concluded by discussing capability of the numerical scheme to develop fast algorithms for approximating multiple convolutions with functions which are radial in translations.
\end{abstract}
\maketitle
\tableofcontents

\section{{\bf Introduction}}\label{Int}

Special Euclidean groups $SE(d):=\mathbb{R}^d \rtimes SO(d)$ are non-Abelian and non-compact semi-direct product groups which are classically used for modeling of the rigid body motions (handedness-preserving isometries) on Euclidean spaces. These semi-direct product Lie groups play significant roles in applied analysis,  geometry, mathematical physics, and computational mathematics \cite{Ba.Gi, Amit, Kisil.Adv, Kisil.Book, Kisil.JPA, Kisil.BJMA,  DR.SIAM.JC.2007, DR.ACHA.1995, YY.ITSF.2007, YY.IPI.2007}.

Convolution of probability density functions on Special Euclidean groups  $SE(d)$, as well as associated diffusion processes have a number
of applications in robotic manipulation, mobile robot localization, biomolecular statistical mechanics
\cite{imme, PI5, GSC.JFAA.2000, Bana.Wolf.Chir}, and even cognition (retinal and neural systems) \cite{citti+sarti, citti+sarti1, Du.Entropy,  duits1, Du, duits}. Of particular importance are the cases $d=2,3$.

For $f_1,f_2\in L^1(SE(d))$, the $SE(d)$-convolution $f_1\star f_2\in L^1(SE(d))$ is given by 
\begin{equation}\label{fstargd}
(f_1\star f_2)(\mathbf{g})=\int_{SE(d)}f_1(\mathbf{h})f_2(\mathbf{h}^{-1}\circ\mathbf{g})\D\mathbf{h},\hspace{1cm}{\rm for}\ \mathbf{g}\in SE(d),
\end{equation} 
where $\mathbf{h}^{-1}$ is the $SE(d)$-inverse of $\mathbf{h}$, $\circ$ is the $SE(d)$-group operation, and $\D\mathbf{h}$ is the Haar measure on $SE(d)$.

The non-Abelian Fourier transform
for $SE(d)$ is a natural tool for computing the associated noncommutative
convolutions as 
$$ \widehat{(f_1 \star f_2)}(p) = \widehat{f_2}(p)\widehat{f_1}(p) \,,$$
where $\widehat{f_j}(p)$ ($\{j\in\{1,2\}$) is the non-Abelian Fourier transform of $f_j$ at $p\in(0,\infty)$, see \cite{FollH}.

However two problems present themselves immediately when taking this approach:
1) The Fourier transforms $\widehat{f_j}(p)$ are infinite dimensional linear operators; 2) They are functions of a continuous parameter, $p$. These two problems lead to issues in the numerical implementation of $SE(d)$ harmonic analysis.

Since in real-world applications, the functions $f_j: SE(d) \,\rightarrow\, \mathbb{R}_{\geq 0}$ typically have compact support, an approach to circumventing these problems is to do what is often done in engineering -- revert to Fourier series on a region sufficiently large enough to encapsulate the problem, and periodize. In the Abelian setting, this allows one to convert a problem on $\mathbb{R}^d$ to one on $\mathbb{R}^d/\mathbb{Z}^d \,\cong\, \mathbb{T}^d$ (the $d$-dimensional torus). Then, sampling of this compact space results in efficient
Fast Fourier Transform (FFT) implementations. 

Here the same strategy is invoked by
quotienting $SE(d)$ by the discrete subgroup $\mathbb{L}$ consisting of translational isometries of the  orthogonal lattice $\mathbb{Z}^d$ in $\mathbb{R}^d$. 
Whereas in the commutative case left and right cosets are the same, in the present context they are different (since $\mathbb{L}$ is not normal in $SE(d)$) and only the right coset space
$$ \mathbb{L}\backslash SE(d) 
\,\cong\, \mathbb{T}^d \times SO(d), $$
is of interest for reasons explained
in \cite{xstal}.

If all that was desired was to decompose functions on this trivial principal bundle into harmonics, it could be done quite simply since the basis for $\mathbb{T}^d$ and $SO(d)$ 
each have been known respectively for 200 and 100 years as developed by Fourier himself and by the Peter-Weyl theorem \cite{FollH}. But applications originate from computing convolutions on $SE(d)$ localized to a fundamental domain containing one element of $SE(d)$ per representative of a coset 
$\mathbb{L}\mathbf{g}$. In this way, the right coset space $\mathbb{L}\backslash SE(d)$ can be viewed as $\Omega:=[-\frac{1}{2},\frac{1}{2}]^d \times SO(d)$ with the antipodal ends of each interval $[-\frac{1}{2},\frac{1}{2}]$ glued to each other, as is done in classical Fourier analysis. As long as the whole physical problem of interest is contained in such a compact domain, then there is no distinction between the original $SE(d)$ convolutions of interest, $(f_1 \star f_2)(\mathbf{g})$, and what we compute here, which are of the form 
$$ \widetilde{(f_1 \star f_2)}(\mathbb{L}\mathbf{g}) \,=\, \int_{SE(d)}
\widetilde{f_1}(\mathbb{L}\mathbf{h}) f_2(\mathbf{h}^{-1}\circ\mathbf{g})\D \mathbf{h} \,,\hspace{0.5cm}{\rm for}\ \mathbf{g}\in SE(d)\,,$$
where $\widetilde{f}\in L^1(\mathbb{L}\backslash SE(d))$ is the $\mathbb{L}$-periodization of $f\in L^1(SE(d))$. 

In most applications, the original functions on $SE(d)$ are very general without specific simplifying structures.
But on the way to developing such a general harmonic analysis for computing
$\widetilde{f_1 \star f_2}$ as a proxy for general $f_1 \star f_2$, we first examine the case when
$f_2$ has translational part that is radially symmetric around the origin.
In principle, if $\{\rho_k\}$ is a sufficiently rich (complete) set of $SE(d)$-shifts of a function $\rho\in L^1(SE(2))$ which is radial in translations, then a rather general $f_2$ can be approximated as
$ f_2 \,\approx\, \sum_k a_k \, 
\rho_k$. Then 
$$ \widetilde{(f_1 \star f_2)}(\mathbb{L}\mathbf{g})
\,\approx\,  \sum_k a_k \, \int_{SE(d)} \widetilde{f_1}(\mathbb{L}\mathbf{h})\rho_k(\mathbf{h}^{-1}\circ\mathbf{g})\D\mathbf{h} \,,\hspace{0.5cm}{\rm for}\ \mathbf{g}\in SE(d)\,.$$

This approach allows us to first consider the reduced problem of computing convolutions of the form
$$ \widetilde{(f \star \rho)}(\mathbb{L}\mathbf{g}) \,=\,
\int_{SE(d)} \widetilde{f}(\mathbb{L}\mathbf{h})
\rho(\mathbf{h}^{-1}\circ\mathbf{g})\D\mathbf{h}\,,\hspace{0.5cm}{\rm for}\ \mathbf{g}\in SE(d)\,,$$
where $f,\rho\in L^1(SE(d))$ and $\rho$ is radially symmetric (radial) on translations.
 
As will be shown, the translational part of this convolution integral can be reduced in this case to the classical Abelian case, and implemented by FFT. In the very special case when $d=2$, 
$$ \mathbb{L}\backslash SE(2) 
\,\cong\, \mathbb{T}^2 \times SO(2)
\,\cong\, \mathbb{T}^3\,, $$
which means that the whole convolution can be implemented by
classical FFT. The details of this follow in the remainder of the paper. 

Throughout, we develop a fast numerical scheme to approximate $SE(2)$ convolutions of the form $f\star \rho$ where $f$ is arbitrary and $\rho$ is radial in translations.
This paper discusses a constructive Fourier-type approximation scheme on the right coset space $\mathbb{L}\backslash SE(2)$ together with a fast and accurate numerical implementation by utilizing a unified transition structure which allows for approximation of functions on the right coset space $\mathbb{L}\backslash SE(2)$ while benefiting from both group structure of $SE(2)$ and compactness of the homogeneous space $\mathbb{L}\backslash SE(2)$.

The paper is organized as follows. Section 2 is devoted to review of general notation, a summary for classical harmonic analysis methods on the group $SE(2)$ and the right coset space $\mathbb{L}\backslash SE(2)$.  Section 3 investigates the structure of finite Fourier series on the right coset space $\mathbb{L}\backslash SE(2)$ as the theoretical construction of the approximation scheme. This includes applying several analytic, algebraic, and numerical approaches. To begin with, we present construction of Fourier series on the  fundamental domain $\Omega:=[-\frac{1}{2},\frac{1}{2}]^2\times SO(2)$ for the subgroup $\mathbb{L}$ in $SE(2)$. In addition, we discuss a transition approach using the structure of the group $SE(2)$, the orthogonal lattice $\mathbb{Z}^2$, and the right coset space $\mathbb{L}\backslash SE(2)$ to analyze structure of functions on the right coset space $\mathbb{L}\backslash SE(2)$ by functions supported in the fundamental domain $\Omega$. To sum up, by applying the transition structure to the Fourier series on the fundamental domain $\Omega$, we then study structure of Fourier series on the right coset space $\mathbb{L}\backslash SE(2)$. 
Next, a fast and accurate numerical scheme to approximate partial Fourier sums of functions on the right coset space $\mathbb{L}\backslash SE(2)$ will be investigated. The computational scheme employs a unified constructive numerical integration scheme using uniform sampling on the fundamental domain $\Omega$ for computing accurate approximations of Fourier coefficients associated to an arbitrary Fourier partial sum. The convergence/error analysis for numerical approximation of Fourier coefficients are investigated. We established constructive assumptions and boundary conditions for the numerical approximation of Fourier coefficients to converge to the Fourier coefficients.
Section 4 develops the matrix form for the numerical approximation of Fourier partial sums. It is shown that the numerical scheme can be reformulated into DFT and hence can be implemented using fast Fourier algorithms (FFT).
The section is concluded by illustration of different classes of numerical experiments in MATLAB. 

Next, section 5 discusses structure of Fourier coefficients for convolutions of the form $f\star \rho$ where $f$ is arbitrary and $\rho$ is radial in translations. Then construction of convolutional finite Fourier series as an efficient numerical approximation of convolutions of the form $f\star \rho$, where $f$ is arbitrary and $\rho$ is radial in translations, will be studied. Section 6 presents the matrix forms and numerical experiments related to convolutional finite Fourier series. It is shown that the convolutional numerical scheme can be reformulated into DFT and hence can be implemented using fast Fourier algorithms (FFT). The paper is concluded by discussing capability of the numerical scheme to develop a fast algorithms for multiple convolutions with functions which are radial in translations and the related simulations in MATLAB.

\newpage
\section{{\bf Preliminaries and Notation}}
Throughout this section we fix notations and review some preliminaries.
\subsection{General notation} Suppose $a>0$, $d,q\in\mathbb{N}$ and $p\in (0,\infty]$.   
\begin{enumerate}
\item For $\mathbf{L}:=(L_1,\cdots,L_d)^T\in\mathbb{N}^d$, $\mathbb{C}^\mathbf{L}$ denotes the linear space of all $d$-dimensional array of size $L_1\times\cdots\times L_d$. We may also denote $\mathbb{C}^\mathbf{L}$ by $\mathbb{C}^{L_1\times\cdots\times L_d}$.
\item For $\mathbf{L}\in\mathbb{N}^d$ and $\mathbf{A},\mathbf{B}\in\mathbb{C}^{\mathbf{L}}$, $\mathbf{A}\odot\mathbf{B}$ is the Hadamard product of $\mathbf{A},\mathbf{B}$.
\item For $L\in\mathbb{N}$ and $k\in\mathbb{Z}$, $\tau_L(k)$ denotes $\mathrm{mod}(k,L)$, that is the remainder after division of $k$ by $L$.
\item For $x\in\mathbb{R}$, $\lceil x\rceil$ denotes the ceiling of $x$, that is the least integer part of $x$, which is defined as the smallest integer greater than or equal to $x$.
\item For $\mathbf{x}=(x_1,\ldots,x_d)^T\in\mathbb{R}^d$, we write $|\mathbf{x}|:=(|x_1|,\ldots,|x_d|)$.
\item For $\mathbf{x}=(x_1,\ldots,x_d)^T\in\mathbb{R}^d$, we write $\min(\mathbf{x}):=\min\{x_i:1\le i\le d\}$.
\item For $\mathbf{x},\mathbf{y}\in\mathbb{R}^d$, we write $\mathbf{x}\le\mathbf{y}$ (resp. $\mathbf{x}<\mathbf{y}$) if $x_i\le y_i$ (resp. $x_i<y_i$) for $1\le i\le d$.
\item For sequences $(a_n),(b_n)\in\mathbb{C}^\mathbb{N}$ and $N\in\mathbb{N}$,  we write $a_n=\mathcal{O}(b_n)$ for $n\ge N$, if there exists $M>0$ such that $|a_n|\le M|b_n|$ for $n\ge N$.
\item If $\mathbf{x}\in\mathbb{R}^d$ and $p<\infty$,  $\|\mathbf{x}\|_p^p$ denotes $\sum_{i=1}^d|x_i|^p$.
\item If $\mathbf{x}\in\mathbb{R}^d$,  $\|\mathbf{x}\|_{\infty}$ denotes $\max\{|x_i|:1\le i\le d\}$.
\item If $\mathbf{x},\mathbf{y}\in\mathbb{R}^d$,  $\langle\mathbf{x},\mathbf{y}\rangle:=\sum_{i=1}^dx_iy_i$ denotes the dot product (standard inner product) of $\mathbb{R}^d$.
\item $\mathbb{B}_a:=\{\mathbf{x}\in\mathbb{R}^d:\|\mathbf{x}\|_2\le a\}$ denotes the disk of radius $a$ in $\mathbb{R}^d$ and $\mathbb{B}$ denotes $\mathbb{B}_1$.
\item $SO(d)$ denotes the special orthogonal group in dimension $d$.
\item $\mathbb{D}_a:=\mathbb{B}_a\times SO(d)$. 
\item $U,V,\ldots$ denote complex-valued functions on boxes in $\mathbb{R}^d$ with $d\ge 1$.
\item $f,g,\ldots$ denote arbitrary functions on $SE(d)$.
\item $\rho,\varrho,\ldots$ denote functions on $SE(d)$ which are radial in translations.
\item For a function $f:\mathbb{R}^d\to\mathbb{C}$, $\mathrm{supp}(f):=\{\mathbf{x}\in\mathbb{R}^d:f(\mathbf{x})\not=0\}$ denotes the support of $f$
\item $\mathcal{C}_c(X)$ denotes the set of all continuous $f:X\subseteq\mathbb{R}^d\to\mathbb{C}$ with compact support.
\item $\mathcal{C}^q(X)$ denotes the set of all $f:X\subseteq\mathbb{R}^d\to\mathbb{C}$ such that all partial derivatives of order $\le q$ exist and are continuous.
\item For $f\in\mathcal{C}^1(X)$,  $\nabla f:X\subseteq\mathbb{R}^d\to\mathbb{C}$ denotes the gradient of $f$.
\item For $X\subseteq\mathbb{R}^d$ and $1\le p<\infty$,  $L^p(X)$ denotes the space of measurable complex valued functions for which the $p$-th power of the absolute value is Lebesgue integrable.
\end{enumerate}

\subsection{Harmonic analysis on $SE(2)$}
The 2D special Euclidean motion group,  denoted $SE(2)$,  can be realized as the semi-direct product of the Abelian group $\mathbb{R}^2$ with the 2D special orthogonal group $SO(2)$,  that is $SE(2)=\mathbb{R}^2\rtimes SO(2)$.
The group element $\mathbf{g}\in SE(2)$ is then denoted as $\mathbf{g}=(\mathbf{x},\mathbf{R})$ where $\mathbf{x}\in\mathbb{R}^2$ and $\mathbf{R}\in SO(2)$.  For $\mathbf{g}=(\mathbf{x},\mathbf{R})$ and 
$\mathbf{g}'=(\mathbf{x}',\mathbf{R}')\in SE(2)$ the group operation is 
\begin{equation}\label{SE2law}
\mathbf{g}\circ\mathbf{g}'=(\mathbf{x}+\mathbf{R}\mathbf{x}',\mathbf{RR}'),
\end{equation}
with the inverse 
\[
\mathbf{g}^{-1}=(-\mathbf{R}^T\mathbf{x},\mathbf{R}^T)=(-\mathbf{R}^{-1}\mathbf{x},\mathbf{R}^{-1}).
\]
The group $SE(2)$ can be faithfully represented by the multiplicative group of $3\times 3$ homogeneous transformation matrices of the form 
\begin{equation}\label{gxyt}
\mathbf{g}(x_1,x_2,\theta):=\hspace{-0.1cm}\left(\hspace{-0.1cm}\begin{array}{ccc}
\cos\theta & -\sin\theta & x_1 \\ 
\sin\theta & \cos\theta & x_2 \\ 
0 & 0 & 1
\end{array}\hspace{-0.1cm}\right),
\end{equation}
where $x_1,x_2\in\mathbb{R}$ and $\theta\in[0,2\pi)$. The Lie algebra $\mathfrak{se}(2)$ of $SE(2)$ consists of $3\times 3$ real matrices of the form
\[
\mathfrak{se}(2):=\left\{\mathrm{X}(v_1,v_2,\theta):=\left(\begin{array}{ccc}
0 & -\theta & v_1 \\ 
\theta & 0 & v_2 \\ 
0 & 0 & 0
\end{array}\right):(v_1,v_2,\theta)\in\mathbb{R}^3 \right\}.
\]
The group $SE(2)$ is non-Abelian and unimodular with a Haar measure given by 
\begin{equation}\label{dg}
\dd\mathbf{g}=\dd\mathbf{x}\dd\mathbf{R}=\frac{1}{2\pi}\dd x_1\dd x_2\dd\theta.
\end{equation}

Suppose $\mathbb{L}:=\left\{(\bfell,\mathbf{I}):\bfell\in\mathbb{Z}^2\right\}$ is the group of translational isometries of the  orthogonal lattice $\mathbb{Z}^2$ in $\mathbb{R}^2$. Then $\mathbb{L}$ is a discrete subgroup of $SE(2)$ with the counting measure as the Haar measure. The right coset space 
$\mathbb{L}\backslash SE(2):=\left\{\mathbb{L}\circ\mathbf{g}:\mathbf{g}\in SE(2)\right\}$,
is a compact homogeneous space which the Lie group $SE(2)$ acts on it from the right, where $\mathbb{L}\circ\mathbf{g}:=\left\{(\bfell,\mathbf{I})\circ\mathbf{g}:\bfell\in\mathbb{Z}^2\right\}$,
for $\mathbf{g}\in SE(2)$. For simplicity in notations, the coset $\mathbb{L}\circ\mathbf{g}$ is denoted by $\mathbb{L}\mathbf{g}$, the subgroup $\mathbb{L}$ may denoted by just $\mathbb{Z}^2$ and the coset space $\mathbb{L}\backslash SE(2)$ sometimes also denoted by $\mathbb{Z}^2\backslash SE(2)$ or $\mathbb{X}$. 
A subset $\Omega\subset SE(2)$ is called as a fundamental domain for the subgroup $\mathbb{L}$ in $SE(2)$ if 
$SE(2)=\bigcup_{\gamma\in\mathbb{L}}\gamma\circ\Omega,$
and $\gamma\circ\Omega\bigcap\gamma'\circ\Omega=\emptyset,
$ if $\gamma\not=\gamma'$ and $\gamma,\gamma'\in\mathbb{L}$.

The theory of harmonic analysis on locally compact homogeneous spaces studied via different approaches in \cite{FollH, AGHF.JAuMS, AGHF.JKMS, AGHF.BBMSS, HR2, HR1, HR.JS} and the list of references therein. 
The space $\mathcal{C}(\mathbb{X})$ consists of all continuous functions on $\mathbb{X}$, can be identified as the space of all $\mathbb{L}$-periodized functions 
$\widetilde{f}:\mathbb{X}\to\mathbb{C}$, where 
$f\in\mathcal{C}_c(SE(2))$ and (Proposition 2.48 of \cite{FollH})
\begin{equation}\label{TH}
\widetilde{f}(\mathbb{L}\mathbf{g}):=\sum_{\bfell\in\mathbb{Z}^2}f(\bfell+\mathbf{x},\mathbf{R}),\hspace{1cm}{\rm for}\hspace{0.25cm}\mathbf{g}:=(\mathbf{x},\mathbf{R})\in SE(2).
\end{equation}
Let $\mu$ be a Radon measure on the right coset space $\mathbb{X}$ and $\mathbf{h}\in SE(2)$. The right translation $\mu_\mathbf{h}$ of $\mu$ is defined by $\mu_\mathbf{h}(E):=\mu(E\circ\mathbf{h})$, for all Borel subsets $E$ of $\mathbb{X}$, where 
$E\circ \mathbf{h}:=\left\{\mathbb{L}\mathbf{g}\circ\mathbf{h}:\mathbb{L}\mathbf{g}\in E\right\}$.  
The measure $\mu$ is called $SE(2)$-invariant if $\mu_\mathbf{h}=\mu$, for every $\mathbf{h}\in SE(2)$.
Since $SE(2)$ is unimodular, $\mathbb{L}$ is discrete and $\mathbb{X}$ is compact,  there exists a unique finite $SE(2)$-invariant measure $\mu$ on the right coset space $\mathbb{X}$ satisfying the following Weil's formula  
\begin{equation}\label{Weil}
\int_{\mathbb{X}}\widetilde{f}(\mathbb{L}\mathbf{g})\dd\mu(\mathbb{L}\mathbf{g})=\int_{SE(2)}f(\mathbf{g})\dd\mathbf{g},
\end{equation}
for every $f\in L^1(SE(2))$, see Theorem 2.49 of \cite{FollH}. In this case, 
$\mu$ is called as the normalized $SE(2)$-invariant measure on the right coset space $\mathbb{X}$ according to the Haar measure (\ref{dg}).

\subsection{3D Discrete Fourier Transform (DFT)}

Suppose $L,M,N\in\mathbb{N}$ and $\mathbb{C}^{L\times M\times N}$ is the linear space of all 3D array with complex entries of size $L\times M\times N$. The discrete Fourier transform (DFT) of a 3D array $\mathbf{X}\in\mathbb{C}^{L\times M\times N}$, is defined as the 3D array $\widehat{\mathbf{X}}\in\mathbb{C}^{L\times M\times N}$ which is given by 
\begin{equation}\label{3DFT}
\widehat{\mathbf{X}}(\ell,m,n):=\sum_{i=1}^L\sum_{j=1}^M\sum_{k=1}^N\mathbf{X}(i,j,k)e^{-2\pi\ii(i-1)(\ell-1)/L}e^{-2\pi\ii(j-1)(m-1)/M}e^{-2\pi\ii(k-1)(n-1)/N},
\end{equation}
for every $\mathbf{1}\le(\ell,m,n)\le(L,M,N)$.  In addition,  the inverse discrete Fourier transform (iDFT) of $\mathbf{X}\in\mathbb{C}^{L\times M\times N}$, is defined as the matrix $\check{\mathbf{X}}\in\mathbb{C}^{L\times M\times N}$ which is given by 
\begin{equation}\label{3iDFT}
\check{\mathbf{X}}(\ell,m,n):=\frac{1}{LMN}\sum_{i=1}^L\sum_{j=1}^M\sum_{k=1}^N\mathbf{X}(i,j,k)e^{2\pi\ii(i-1)(\ell-1)/L}e^{2\pi\ii(j-1)(m-1)/M}e^{2\pi\ii(k-1)(n-1)/N},
\end{equation}
for every $\mathbf{1}\le(\ell,m,n)\le(L,M,N)$. It is worth mentioning that the DFT can be computed efficiently using the Fast Fourier Transform (FFT) algorithm, see \cite{Av.Co.Do.El.Is, Fe.Ku.Po, FFTSOd} and references therein.

\section{\bf Finite Fourier Series on $\mathbb{Z}^2\backslash SE(2)$}

Throughout, we discuss a unified constructive approach to approximate functions on the right coset space $\mathbb{Z}^2\backslash SE(2)$ using Fourier series. The structure of the approximation technique utilizes both algebraic aspects of the group $SE(2)$ and analytic aspects of functions on the group $SE(2)$ by employing a fundamental domain.

Let $\mathbb{L}$ be the group of translational isometries of the  orthogonal lattice $\mathbb{Z}^2$ in $\mathbb{R}^2$. 
Let $\Omega:=[-\frac{1}{2},\frac{1}{2})^2\times SO(2)$.
Then $\Omega\subset SE(2)$ is a fundamental domain for the subgroup $\mathbb{L}$ in $SE(2)$. Since $SO(2)$ can be parametrized by $[0,2\pi)$, we may also use $\Omega$ to denote the manifold $[-\frac{1}{2},\frac{1}{2})^2\times[0,2\pi)$. In addition, the characteristic function of $\Omega$, is denoted by $E_\Omega$. 

We start by investigating construction of Fourier series on the right coset space $\mathbb{L}\backslash SE(2)$ using Fourier series on the fundamental domain $\Omega$.
\subsection{Structure of Fourier series}\label{SFS}
This part studies structure of Fourier series on $\mathbb{L}\backslash SE(2)$ using a transition structure and Fourier series for functions on $SE(2)$ which are supported in the fundamental domain $\Omega$.

Assume that $\mathbb{I}:=\mathbb{Z}^3$ and $\mathbf{k}=(k_1,k_2,k_3)^T\in\mathbb{I}$.
Let $\phi_{\mathbf{k}}:SE(2)\to\mathbb{C}$ be given by 
\begin{equation}\label{phik}
\phi_{\mathbf{k}}(\mathbf{x},\mathbf{R}_\theta):=e^{2\pi\ii (k_1x+k_2y)}e^{\ii k_3\theta},
\end{equation}
for every $(\mathbf{x},\mathbf{R}_\theta)\in SE(2)$ with $\mathbf{x}:=(x,y)\in\mathbb{R}^2$ and $\mathbf{R}_\theta\in SO(2)$ with $\theta\in[0,2\pi)$. 

In addition, we have 
\begin{equation}\label{pkxy}
\phi_{\mathbf{k}}(\mathbf{x}+\mathbf{y},\mathbf{R}_{\theta+\alpha})=\phi_{\mathbf{k}}(\mathbf{x},\mathbf{R}_\theta)\phi_\mathbf{k}(\mathbf{y},\mathbf{R}_\alpha),
\end{equation}
\begin{equation}\label{p-k}
\overline{\phi_\mathbf{k}(\mathbf{x},\mathbf{R}_\theta)}=\phi_{-\mathbf{k}}(\mathbf{x},\mathbf{R}_\theta),
\end{equation}
and 
\begin{equation}\label{pk+l}
\phi_{\mathbf{k}+\mathbf{l}}(\mathbf{x},\mathbf{R}_\theta)=\phi_{\mathbf{k}}(\mathbf{x},\mathbf{R}_\theta)\phi_{\mathbf{l}}(\mathbf{x},\mathbf{R}_\theta),
\end{equation}
for every $\mathbf{x},\mathbf{y}\in\mathbb{R}^2$ and $\theta,\alpha\in [0,2\pi)$, and $\mathbf{k},\mathbf{l}\in\mathbb{I}$. 

The space of square integrable functions on $SE(2)$ which are supported in $\Omega$ will be denoted by $L^2(\Omega,\D\omega)$ or $L^2(\Omega)$. Since $(E_\Omega\phi_\mathbf{k})_{\mathbf{k}\in\mathbb{I}}$ is an orthonormal basis (ONB) for the Hilbert function space $L^2(\Omega)$, we obtain 
\begin{equation}\label{MainExpO}
f=\sum_{\mathbf{k}\in\mathbb{I}}\widehat{f}[\mathbf{k}]\phi_\mathbf{k},
\end{equation}
for $f\in L^2(\Omega)$, where the series (\ref{MainExpO}) converges in the sense of $L^2(\Omega)$, which is interpreted as 
\[
\lim_{K\to\infty}\left\|f-S_K(f)\right\|_{L^2(\Omega)}^2=\lim_{K\to\infty}\int_{\Omega}\left|f(\omega)-S_K(f)(\omega)\right|^2\D\omega=0.
\]
where the Fourier partial sum $S_K(f):SE(2)\to\mathbb{C}$ is given by 
\begin{equation}\label{SKf}
S_K(f)(\omega):=\sum_{\|\mathbf{k}\|_\infty\le K}\widehat{f}[\mathbf{k}]\phi_\mathbf{k}(\omega),
\end{equation}
for every $\omega\in\Omega$, where $\widehat{f}[\mathbf{k}]$, the Fourier coefficient of $f$ at $\mathbf{k}\in\mathbb{I}$, is given by 
\begin{equation}\label{fk}
\widehat{f}[\mathbf{k}]=:\langle f,\phi_\mathbf{k}\rangle_{L^2(\Omega)}=\int_{\Omega}f(\omega)\overline{\phi_\mathbf{k}(\omega)}\D\omega.
\end{equation}

We then investigate structure of Fourier series on the right coset space $\mathbb{X}$ in terms of the Fourier series on fundamental domain $\Omega$ which discussed above. To this end, we begin with discussing a unified constructive characterization for integration on the right coset space $\mathbb{X}$ in terms of integration on the fundamental domain $\Omega$. 

\begin{lemma}\label{intO}
Let $\mu$ be the normalized $SE(2)$-invariant measure on the right coset space $\mathbb{X}$ according to (\ref{Weil}) and $\psi\in L^1(\mathbb{X},\mu)$. Then 
\[
\int_{\mathbb{X}}\psi(\mathbb{L}\mathbf{g})\dd\mu(\mathbb{L}\mathbf{g})=\int_{\Omega}\psi(\mathbb{L}\omega)\dd\omega.
\]
\end{lemma}
\begin{proof}
Since $\Omega$ is a fundamental domain of $\mathbb{L}$ in $SE(2)$,  we get 
$\widetilde{E_\Omega}(\mathbb{L}\mathbf{g})=1$,
for every $\mathbf{g}\in SE(2)$. Indeed, if $\mathbf{g}=(\mathbf{x},\mathbf{R})$ then 
\[
\widetilde{E_\Omega}(\mathbb{L}\mathbf{g})=\sum_{\bfell\in\mathbb{Z}^2}E_\Omega(\bfell+\mathbf{s},\mathbf{R})=E_\Omega(\mathbf{s},\mathbf{R})=1,
\]
where $\mathbf{n}\in\mathbb{Z}^2$ and $\mathbf{s}\in [\frac{-1}{2},\frac{1}{2})^2$ such that $\mathbf{g}=(\mathbf{n},\mathbf{I})\circ (\mathbf{s},\mathbf{R})$. So,  using Weil's formula (\ref{Weil}), we obtain 
\begin{align*}
\int_{\Omega}\psi(\mathbb{L}\omega)\dd\omega
&=\int_{SE(2)}E_{\Omega}(\mathbf{g})\psi(\mathbb{L}\mathbf{g})\dd\mathbf{g}
\\&=\int_{\mathbb{X}}\widetilde{E_\Omega}(\mathbb{L} g)\psi(\mathbb{L}\mathbf{g})\dd\mu(\mathbb{L}\mathbf{g})
=\int_{\mathbb{X}}\psi(\mathbb{L}\mathbf{g})\dd\mu(\mathbb{L}\mathbf{g}).\qedhere
\end{align*}
\end{proof}

Assume that $\mathbf{k}=(k_1,k_2,k_3)^T\in\mathbb{I}$.
Let $\varphi_{\mathbf{k}}:\mathbb{X}\to\mathbb{C}$ be given by 
\[
\varphi_{\mathbf{k}}(\mathbb{L}\mathbf{g}):=e^{2\pi\ii (k_1x+k_2y)}e^{\ii k_3\theta},
\]
for every $\mathbf{g}=(\mathbf{x},\mathbf{R}_\theta)\in SE(2)$, with  $\mathbf{x}=(x,y)^T$. Then using (\ref{TH}) we have $\widetilde{E_\Omega\phi_\mathbf{k}}=\varphi_\mathbf{k}.$

Next, we show that $(\varphi_\mathbf{k}:\mathbf{k}\in\mathbb{I})$ is an ONB for $L^2(\mathbb{X},\mu)$.

\begin{proposition}\label{Xmain}
{\it Suppose $\mu$ is the normalized $SE(2)$-invariant measure on $\mathbb{X}$ according to (\ref{Weil}). Then, $(\varphi_\mathbf{k})_{\mathbf{k}\in\mathbb{I}}$ is an ONB for the Hilbert space $L^2(\mathbb{X},\mu)$.
}
\end{proposition}
\begin{proof}
Suppose $\mathbf{k},\mathbf{l}\in\mathbb{I}$. Since $(E_\Omega\phi_\mathbf{k})_{\mathbf{k}\in\mathbb{I}}$ is an ONB for $L^2(\Omega,\D\omega)$, Lemma \ref{intO} implies that 
\begin{align*}
\langle\varphi_\mathbf{k},\varphi_\mathbf{l}\rangle_{L^2(\mathbb{X})}&=\int_{\mathbb{X}}\varphi_\mathbf{k}(\mathbb{L}\mathbf{g})\overline{\varphi_\mathbf{l}(\mathbb{L}\mathbf{g})}\D\mu(\mathbb{L}\mathbf{g})
\\&=\int_{\Omega}\varphi_\mathbf{k}(\mathbb{L}\omega)\overline{\varphi_\mathbf{l}(\mathbb{L}\omega)}\D\omega=\int_{\Omega}\phi_\mathbf{k}(\omega)\overline{\phi_\mathbf{l}(\omega)}\D\omega=\delta_{\mathbf{k},\mathbf{l}}.
\end{align*}
It can also be readily check that the family $(\varphi_\mathbf{k})_{\mathbf{k}\in\mathbb{I}}$ is complete in the Hilbert function space $L^2(\mathbb{X})$. 
\end{proof}

Invoking Proposition \ref{Xmain}, every $\psi\in L^2(\mathbb{X},\mu)$ satisfies the expression  $\psi=\sum_{\mathbf{k}\in\mathbb{I}}\widehat{\psi}\{\mathbf{k}\}\varphi_\mathbf{k}$,
where the series converges in the sense of $L^2(\mathbb{X},\mu)$, which is,
\[
\lim_{K\to\infty}\left\|\psi-S_K(\psi)\right\|_{L^2(\mathbb{X})}^2=\lim_{K\to\infty}\int_{\mathbb{X}}\left|\psi(\mathbb{L}\mathbf{g})-S_K(\psi)(\mathbb{L}\mathbf{g})\right|^2\D\mu(\mathbb{L}\mathbf{g})=0.
\]
where the Fourier partial sum $S_K(\psi):\mathbb{X}\to\mathbb{C}$ is given by 
\begin{equation}\label{psiK}
S_K(\psi):=\sum_{\|\mathbf{k}\|_\infty\le K}\widehat{\psi}\{\mathbf{k}\}\varphi_\mathbf{k},
\end{equation}
for every $K\in\mathbb{N}$, and $\widehat{\psi}\{\mathbf{k}\}$, the Fourier coefficient of $\psi$ at $\mathbf{k}\in\mathbb{I}$, is given by 
\begin{equation}\label{psik}
\widehat{\psi}\{\mathbf{k}\}=:\langle\psi,\varphi_\mathbf{k}\rangle_{L^2(\mathbb{X})}=\int_{\mathbb{X}}\psi(\mathbb{L}\mathbf{g})\overline{\varphi_\mathbf{k}(\mathbb{L}\mathbf{g})}\D\mu(\mathbb{L}\mathbf{g}),
\end{equation}

We conclude this section by discussing a unified approach to explicitly formulate Fourier coefficients of functions on $\mathbb{X}$ in terms of Fourier coefficients of functions supported in the fundamental domain $\Omega$.

Let $\psi\in L^2(\mathbb{X},\mu)$ and $\mathbf{k}\in\mathbb{I}$ are arbitrary. Then 
\begin{equation}\label{psikOalt}
\widehat{\psi}\{\mathbf{k}\}=\int_{\Omega}\psi(\mathbb{L}\omega)\overline{\phi_{\mathbf{k}}(\omega)}\D\omega.
\end{equation}
Indeed, using Lemma \ref{intO} we get 
\begin{align*}
\widehat{\psi}\{\mathbf{k}\}
&=\langle\psi,\varphi_\mathbf{k}\rangle_{L^2(\mathbb{X},\mu)}
=\int_{\mathbb{X}}\psi(\mathbb{L}\mathbf{g})\overline{\varphi_{\mathbf{k}}(\mathbb{L}\mathbf{g})}\D\mu(\mathbb{L}\mathbf{g})
\\&=\int_{\Omega}\psi(\mathbb{L}\omega)\overline{\varphi_{\mathbf{k}}(\mathbb{L}\omega)}\D\omega=\int_{\Omega}\psi(\mathbb{L}\omega)\overline{\phi_{\mathbf{k}}(\omega)}\D\omega.
\end{align*}
Since $L^2(\mathbb{X},\mu)\subset L^1(\mathbb{X},\mu)$, we get $\psi\in L^1(\mathbb{X},\mu)$. According to Proposition 2.48 of \cite{FollH}, assume that $f\in L^1(SE(2))$ such that $\widetilde{f}=\psi$.
Then formula (\ref{psikOalt}) reads as 
\begin{equation*}
\widehat{\psi}\{\mathbf{k}\}=\int_{\Omega}\psi(\mathbb{L}\omega)\overline{\phi_{\mathbf{k}}(\omega)}\D\omega=\int_{\Omega}\widetilde{f}(\mathbb{L}\omega)\overline{\phi_{\mathbf{k}}(\omega)}\D\omega.
\end{equation*}
In addition, if $f$ is supported in $\Omega$ then  
\begin{equation}\label{psikf}
\widehat{\psi}\{\mathbf{k}\}=\int_{\Omega}f(\omega)\overline{\phi_\mathbf{k}(\omega)}\D\omega=\widehat{f}[\mathbf{k}].
\end{equation}

\begin{remark}\label{psikpsikO}
Assume $\psi\in L^2(\mathbb{X},\mu)$ is arbitrary. Suppose that the function $\psi_\Omega:SE(2)\to\mathbb{C}$ is given by
\[
\psi_\Omega(\mathbf{g}):=E_\Omega(\mathbf{g})\psi(\mathbb{L}\mathbf{g})=\left\{\begin{array}{lll}\psi(\mathbb{L}\mathbf{g}) &\ {\rm if}\ \mathbf{g}\in\Omega\\
0 &\ {\rm if}\ \mathbf{g}\not\in\Omega
\end{array}\right..
\]
Then $\psi_\Omega$ is supported in $\Omega$ and $\widetilde{\psi_\Omega}=\psi$. Therefore, if  $\mathbf{k}\in\mathbb{I}$, formula (\ref{psikf}) implies that 
$\widehat{\psi}\{\mathbf{k}\}=\widehat{\psi_\Omega}[\mathbf{k}]$.
\end{remark}

\subsection{Finite Fourier series}\label{FFS}
This section introduces structure of a constructive numerical Fourier scheme on the coset space $\mathbb{X}$. In details, we discuss a unified numerical scheme to compute fast and  accurate approximation of the Fourier partial sums of the form $S_K(\psi)$ given by (\ref{psiK}) for functions $\psi\in L^2(\mathbb{X})$. The developed computational strategy employs the notion of finite Fourier coefficients, as numerical approximations of Fourier coefficients given by (\ref{psik}). Then using the notion of finite Fourier coefficients, we introduce finite Fourier series  as numerical approximations of Fourier partial
sums given by (\ref{psiK}).

To begin with, due to the construction of the Fourier series on the right coset space $\mathbb{Z}^2\backslash SE(2)$ which investigated in Section \ref{SFS}, we develop structure of a fast numerical scheme on the fundamental domain $\Omega$.  In this direction, we discuss a numerical integration scheme for computing accurate approximations of the Fourier coefficients of the form (\ref{fk}) using finite Fourier coefficients on 3D boxes which can be reformulated in closed form in terms of DFT and hence can be implemented by fast Fourier algorithms (also known as FFT), as will be discussed in Section \ref{NumFFS}. 

We start by introducing a unified constructive numerical scheme, called as finite Fourier series,  to approximate functions on $SE(2)$ which are supported in the fundamental domain $\Omega$. In details, we develop a numerical constructive algorithm to compute fast and accurate approximation of the Fourier partial sums of the form $S_K(f)$ given by (\ref{SKf}) when $f$ is supported in $\Omega$. 

First, we discuss a unified numerical scheme to compute accurate/fast approximation of Fourier coefficients given by (\ref{fk}). So, we begin with introducing a unified fundamental grid as the sampling grid. 

For every $\mathbf{L}:=(L_\x,L_\y,L_\rt)\in\mathbb{N}^3$, let $\mathbf{\Omega}_\mathbf{L}$ be the finite subset of the fundamental domain $\Omega$ given by 
\begin{equation}\label{OmL}
\mathbf{\Omega}_\mathbf{L}:=\left\{\bfomega_{i,j,l}:\mathbf{1}\le (i,j,l)\le\mathbf{L}\right\},
\end{equation}
where $\bfomega_{i,j,l}:=(x_i,y_j,\theta_l)$ with $x_i:=-\frac{1}{2}+\frac{(i-1)}{L_\mathrm{x}}$, $y_j:=-\frac{1}{2}+\frac{(j-1)}{L_\mathrm{y}}$ and $\theta_l:=\frac{2\pi(l-1)}{L_\mathrm{r}}$ for every $\mathbf{1}\le (i,j,l)\le\mathbf{L}$. In this case, the grid $\mathbf{\Omega}_\mathbf{L}$ is of size $L_\x\times L_\y\times L_\rt$. We may also call $\mathbf{\Omega}_\mathbf{L}$ as the fundamental grid associated to the vector $\mathbf{L}$, or just the $\mathbf{L}$-sampling grid.

Suppose $f:SE(2)\to\mathbb{C}$ is a function supported in $\Omega$ and $\mathbf{k}:=(k_1,k_2,k_3)^T\in\mathbb{I}$. Let $\mathbf{L}:=(L_\mathrm{x},L_\mathrm{y},L_\mathrm{r})^T\in\mathbb{N}^3$. The finite Fourier coefficient (FFC) of $f$ at $\mathbf{k}$ associated to the vector $\mathbf{L}$ is given by 
\begin{equation}\label{fkLvMain}
\widehat{f}[\mathbf{k};\mathbf{L}]:=\frac{1}{|\mathbf{\Omega}_\mathbf{L}|}\sum_{\bfomega\in\Omega_\mathbf{L}}f(\bfomega)\overline{\phi_\mathbf{k}(\bfomega)},
\end{equation}
where $\mathbf{\Omega}_\mathbf{L}$ is the fundamental grid given by (\ref{OmL}).

Next, we discuss absolute error bound for the finite Fourier coefficients as 
an accurate and computable approximation of the Fourier coefficients. 

\begin{theorem}\label{fkLmain}
{\it Let $f\in\mathcal{C}^1(SE(2))$ be a function supported in $\Omega^\circ$.
Suppose $\mathbf{k}\in\mathbb{Z}^3$ and $\mathbf{K}\in\mathbb{N}^3$. Then 
\[
\left|\widehat{f}[\mathbf{k}]-\widehat{f}[\mathbf{k};2\mathbf{K}+1]\right|\le\frac{32\|\nabla f\|_\infty}{\min(\mathbf{K})},\hspace{1cm}\ {\rm if}\ |\mathbf{k}|\le\mathbf{K}.
\]
In particular, if $K\in\mathbb{N}$ then
\[
\left|\widehat{f}[\mathbf{k}]-\widehat{f}[\mathbf{k};2K+1]\right|\le\frac{32\|\nabla f\|_\infty}{K},\hspace{1cm}\ {\rm if}\ \|\mathbf{k}\|_\infty\le K.
\]
}\end{theorem}
\begin{proof}
Assume $|\mathbf{k}|\le\mathbf{K}$. Let $U$ be the restriction of $f$ to $\Omega$. Then $U$ is periodic. So, Theorem \ref{UkL} implies  
\[
\left|\widehat{f}[\mathbf{k}]-\widehat{f}[\mathbf{k};\mathbf{L}]\right|=\left|\widehat{U}[\mathbf{k}]-\widehat{U}[\mathbf{k};\mathbf{L}]\right|\le\frac{32\|\nabla U\|_\infty}{\min(\mathbf{K})}=\frac{32\|\nabla f\|_\infty}{\min(\mathbf{K})}.\qedhere
\]
\end{proof}

\begin{remark}
Theorem \ref{fkLmain} reads as the following accuracy result. 
Suppose $f\in \mathcal{C}^1(SE(2))$ is supported in the fundamental domain .  Let $\mathbf{k}\in\mathbb{Z}^3$ and $\epsilon>0$.  Assume $K:=\lceil \max\{\epsilon^{-1}32\|\nabla f\|_\infty,\|\mathbf{k}\|_\infty\}\rceil$ and $L:=2K+1$. Then $\widehat{f}[\mathbf{k};L]$ approximates the Fourier coefficient $\widehat{f}[\mathbf{k}]$ with the absolute error less than $\epsilon$. 
\end{remark}

\begin{corollary}\label{fkOK}
{\it Let $f\in\mathcal{C}^1(SE(2))$ be a function supported in $\Omega^\circ$, $\mathbf{k}\in\mathbb{Z}^3$ and $K\in\mathbb{N}$. Then 
\[
\left|\widehat{f}[\mathbf{k}]-\widehat{f}[\mathbf{k};2K+1]\right|=\mathcal{O}\left(\frac{1}{K}\right),\hspace{1cm}\ {\rm if}\ \|\mathbf{k}\|_\infty\le K.
\]
}\end{corollary}

Next, using the notion of finite Fourier coefficients introduced in (\ref{fkLvMain}), we present the notion of finite Fourier series as computable approximation for Fourier partial sums of the form (\ref{SKf}). 

Assume $\mathbf{K}\in\mathbb{N}^3$ and $f:SE(2)\to\mathbb{C}$ is a function supported in $\Omega$. The finite Fourier series of $f$ associated to $\mathbf{K}$, denoted by $S_{\mathbf{K}}[f]:SE(2)\to\mathbb{C}$, is given by 
\begin{equation}\label{SKvf}
S_\mathbf{K}[f](\omega):=\sum_{|\mathbf{k}|\le\mathbf{K}}\widehat{f}[\mathbf{k};2\mathbf{K}+1]\phi_{\mathbf{k}}(\omega),\hspace{1cm}{\rm for}\ \omega\in\Omega.
\end{equation}

In particular, if  $K\in\mathbb{N}$, the finite Fourier series of order $K$ of $f$, reads as 
\begin{equation}
S_K[f](\omega)=\sum_{\|\mathbf{k}\|_\infty\le K}\widehat{f}[\mathbf{k};2K+1]\phi_{\mathbf{k}}(\omega),\hspace{1cm}{\rm for}\ \omega\in\Omega.
\end{equation}

Next, we study structure of finite Fourier series on the right coset space $\mathbb{Z}^2\backslash SE(2)$. We utilize the theoretical approach discussed in Section \ref{SFS} to develop the numerical scheme of finite Fourier series on the right coset space $\mathbb{X}=\mathbb{L}\backslash SE(2)$, where $\mathbb{L}$ is the group of translational isometries of the orthogonal lattice $\mathbb{Z}^2$. In addition, $\mu$ is the $SE(2)$-invariant measure on $\mathbb{X}$ satisfying (\ref{Weil}).

Suppose that $\mathbf{L}:=(L_\x,L_\y,L_\rt)\in\mathbb{N}^3$. For a function $\psi:\mathbb{X}\to\mathbb{C}$ and $\mathbf{k}\in\mathbb{Z}^3$, let 
\begin{equation}\label{psikO}
\widehat{\psi}\{\mathbf{k};\mathbf{L}\}:=\frac{1}{|\mathbf{\Omega}_\mathbf{L}|}\sum_{\bfomega\in\Omega_\mathbf{L}}\psi(\mathbb{L}\bfomega)\overline{\varphi_\mathbf{k}(\mathbb{L}\bfomega)},
\end{equation}
where $\mathbf{\Omega}_\mathbf{L}$ is the fundamental grid given by (\ref{OmL}).
In detail, the finite Fourier coefficient of $\psi:\mathbb{X}\to\mathbb{C}$ at $\mathbf{k}$ using the sampling vector $\mathbf{L}:=(L_\mathrm{x},L_\mathrm{y}, L_\mathrm{r})$ is  
\begin{equation}\label{ssikLv}
\widehat{\psi}\{\mathbf{k};\mathbf{L}\}=\frac{1}{L_\mathrm{x}L_\mathrm{y}L_\mathrm{r}}\sum_{i=1}^{L_\mathrm{x}}\sum_{j=1}^{L_\mathrm{y}}\sum_{l=1}^{L_\mathrm{r}}\psi(\mathbb{L}\mathbf{g}(x_i,y_j,\theta_l))\overline{\phi_\mathbf{k}(x_i,y_j,\theta_l)},
\end{equation}
where $x_i:=-\frac{1}{2}+\frac{(i-1)}{L_\mathrm{x}}$, $y_j:=-\frac{1}{2}+\frac{(j-1)}{L_\mathrm{y}}$ and $\theta_l:=\frac{2\pi(l-1)}{L_\mathrm{r}}$ for every $\mathbf{1}\le (i,j,l)\le\mathbf{L}$.

Next proposition discusses absolute error of the finite Fourier coefficient $\widehat{\psi}\{\mathbf{k};\mathbf{L}\}$.
\begin{proposition}
{\it Suppose $\psi\in\mathcal{C}(\mathbb{X})$ such that $\psi_\Omega\in\mathcal{C}^1(SE(2))$ is supported in $\Omega^\circ$. Let $\mathbf{k}\in\mathbb{I}$ and $\mathbf{K}\in\mathbb{N}^3$. Then 
\[
\left|\widehat{\psi}\{\mathbf{k}\}-\widehat{\psi}\{\mathbf{k};2\mathbf{K}+1\}\right|\le\mathcal{O}\left(\frac{1}{\min(\mathbf{K})}\right),\hspace{1cm}\ {\rm if}\  |\mathbf{k}|\le\mathbf{K}.
\]
}\end{proposition}
\begin{proof}
Let $ |\mathbf{k}|\le\mathbf{K}$. 
Since $\psi_\Omega\in\mathcal{C}^1(SE(2))$ is supported in $\Omega^\circ$, using Remark \ref{psikpsikO} and Theorem \ref{fkLmain}, we obtain 
\[
\left|\widehat{\psi}\{\mathbf{k}\}-\widehat{\psi}\{\mathbf{k};2\mathbf{K}+1\}\right|=\left|\widehat{\psi_\Omega}[\mathbf{k}]-\widehat{\psi_\Omega}[\mathbf{k};2\mathbf{K}+1]\right|\le\frac{32\|\nabla\psi_\Omega\|_\infty}{\min(\mathbf{K})}.\qedhere
\]
\end{proof}
Let $\mathbf{K}\in\mathbb{N}^3$ and $\psi:\mathbb{X}\to\mathbb{C}$ is a function. The finite Fourier series of $\psi$ associated to $\mathbf{K}$, denoted by $S_{\mathbf{K}}[\psi]:\mathbb{X}\to\mathbb{C}$, is the function given by 
\begin{equation}\label{SKvpsi}
S_\mathbf{K}[\psi](\mathbb{L}\mathbf{g}):=\sum_{|\mathbf{k}|\le\mathbf{K}}\widehat{\psi}\{\mathbf{k};2\mathbf{K}+1\}\varphi_{\mathbf{k}}(\mathbb{L}\mathbf{g}),\hspace{1cm}{\rm for}\ \mathbf{g}\in SE(2).
\end{equation}
In particular, if $K\in\mathbb{N}$, the finite Fourier series $S_K[\psi]$ reads as 
\begin{equation}
S_K[\psi](\mathbb{L}\mathbf{g}):=\sum_{\|\mathbf{k}\|_\infty\le K}\widehat{\psi}\{\mathbf{k};2K+1\}\varphi_{\mathbf{k}}(\mathbb{L}\mathbf{g}),\hspace{1cm}{\rm for}\ \mathbf{g}\in SE(2).
\end{equation}

\section{\bf Numerics of Finite Fourier Series}\label{NumFFS}
This section discusses a unified strategy in terms of discrete Fourier transforms for numerical implementation of finite Fourier coefficients given by (\ref{fkLvMain}) using fast Fourier
algorithms in MATLAB. It is worth mentioning that by applying the approach discussed in Remark \ref{psikpsikO}, functions on the right coset space $\mathbb{X}$ can be realized by functions on $SE(2)$ which are supported in the fundamental domain $\Omega$. Therefore, the matrix forms and numerical experiments will be  discussed for functions supported in $\Omega$. 

\subsection{Matrix forms}\label{MatSKf}
This part discusses matrix form of the finite transform and the inverse finite transform. Let $f:SE(2)\to\mathbb{C}$ be a function supported in $\Omega$.  Assume $\mathbf{k}:=(k_1,k_2,k_3)^T\in\mathbb{I}$ and $\mathbf{L}:=(L_\mathrm{x},L_\mathrm{y},L_\mathrm{r})^T\in\mathbb{N}^3$.
Then 
\begin{align}
\widehat{f}[\mathbf{k};\mathbf{L}]
&=\frac{1}{L_\mathrm{x}L_\mathrm{y}L_\mathrm{r}}\sum_{i=1}^{L_\mathrm{x}}\sum_{j=1}^{L_\mathrm{y}}\sum_{l=1}^{L_\mathrm{r}}f(x_i,y_j,\theta_l)e^{-2\pi\ii(k_1x_i+k_2y_j)}e^{-\ii k_3\theta_l}\nonumber
\\&=\frac{(-1)^{k_1+k_2}}{L_\mathrm{x}L_\mathrm{y}L_\mathrm{r}}\sum_{i=1}^{L_\mathrm{x}}\sum_{j=1}^{L_\mathrm{y}}\sum_{l=1}^{L_\mathrm{r}}f(x_i,y_j,\theta_l)e^{-2\pi\ii k_1(i-1)/{L_\mathrm{x}}}e^{-2\pi\ii k_2(j-1))/{L_\mathrm{y}}}e^{-2\pi\ii k_3(l-1)/{L_\mathrm{r}}}\nonumber
\\&=\frac{(-1)^{k_1+k_2}}{L_\mathrm{x}L_\mathrm{y}L_\mathrm{r}}\sum_{i=1}^{L_\mathrm{x}}\sum_{j=1}^{L_\mathrm{y}}\sum_{l=1}^{L_\mathrm{r}}f(x_i,y_j,\theta_l)e^{-2\pi\ii\tau_{L_\mathrm{x}}(k_1)(i-1)/L_\mathrm{x}}e^{-2\pi\ii \tau_{L_\mathrm{y}}(k_2)(j-1))/L_\mathrm{y}}e^{-2\pi\ii\tau_{L_\mathrm{r}}(k_3)(l-1)/L_\mathrm{r}}\nonumber
\\&=\frac{(-1)^{k_1+k_2}}{L_\mathrm{x}L_\mathrm{y}L_\mathrm{r}}\widehat{\mathbf{F}}(\tau_{L_\mathrm{x}}(k_1)+1,\tau_{L_\mathrm{y}}(k_2)+1,\tau_{L_\mathrm{r}}(k_3)+1),\label{FkLv}
\end{align}
where $\mathbf{F}\in\mathbb{C}^{\mathbf{L}}$ is given by $\mathbf{F}(i,j,l):=f(x_i,y_j,\theta_l)$, 
for $1\le(i,j,l)\le\mathbf{L}$, and $\widehat{\mathbf{F}}$ is the DFT of $\mathbf{F}$ given by (\ref{3DFT}). In particular, we have 
\begin{equation}\label{fkLvM}
\widehat{f}[\mathbf{k};L]=\frac{(-1)^{k_1+k_2}}{L^3}\widehat{\mathbf{F}}(\tau_L(k_1)+1,\tau_L(k_2)+1,\tau_L(k_3)+1).
\end{equation}
The matrix form (\ref{FkLv}) has practical advantage in numerical experiments. The matrix form (\ref{FkLv}) reformulated multivariate summations in the right hand-side of (\ref{fkLvMain}) into discrete Fourier transform (DFT) of the sampled values of the function $f$ which can be then performed using fast Fourier algorithms. 

Next, we introduce Algorithm \ref{fkLAlg} using regular sampling on the fundamental domain $\Omega$ to approximate the Fourier coefficients of the form $\widehat{f}[\mathbf{k}]$ for function $f:SE(2)\to\mathbb{C}$ supported in the fundamental domain $\Omega$ with respect to a given absolute error.
\begin{algorithm}[H]
\caption{Finding $\epsilon$-approximation of $\widehat{f}[\mathbf{k}]$ using regular square sampling of $\Omega$ and FFT} 
\begin{algorithmic}[1]
\State{\bf input data} Given function $f:SE(2)\to\mathbb{C}$ supported in $\Omega$,  error $\epsilon>0$ and $\mathbf{k}:=(k_1,k_2,k_3)^T\in\mathbb{I}$
\State {\bf output result} $\widehat{f}[\mathbf{k};L]$ with the absolute error $\le\epsilon$\\ 
Put $\beta_\epsilon:=\max\{\epsilon^{-1}32\|\nabla f\|_\infty,\|\mathbf{k}\|_\infty\}$,  find $K\in\mathbb{N}$ such that $K\ge\beta_\epsilon$ and let  $L:=2K+1$\\
Generate the fundamental grid $(x_i,x_j,\theta_l)$ with $x_i:=\frac{-1}{2}+\frac{(i-1)}{L}$,  $\theta_l:=\frac{2\pi(l-1)}{L}$, and $1\le i,j,l\le L$\\
Generate sampled values $\mathbf{F}:=(f(x_i,x_j,\theta_l))_{1\le i,j,l\le L}$ of the 	function $f:SE(2)\to\mathbb{C}$.
\State Compute $\widehat{f}[\mathbf{k};L]$ according to (\ref{fkLvM}) using FFT.
\end{algorithmic} 
\label{fkLAlg}
\end{algorithm}
Next, we present matrix form of the finite Fourier series defined by (\ref{SKvf}) for functions supported in the fundamental domain $\Omega$.

Assume that $\mathbf{K}:=(K_\mathrm{x},K_\mathrm{y},K_\rt)\in\mathbb{N}^3$. Suppose $\mathbf{L}:=2\mathbf{K}+1$ and $\mathbf{L}=(L_\x,L_\y,L_\rt)$. Let $f:SE(2)\to\mathbb{C}$ be a function supported in $\Omega$ and $(x,y,\theta)\in \Omega$. Then  
\begin{align*}
&S_\mathbf{K}[f](x,y,\theta)
=\sum_{|\mathbf{k}|\le\mathbf{K}}\widehat{f}[\mathbf{k};\mathbf{L}]\phi_{\mathbf{k}}(x,y,\theta)
=\sum_{|k_1|\le K_\x}\sum_{|k_2|\le K_\y}\sum_{|k_3|\le K_\rt}\widehat{f}[\mathbf{k};\mathbf{L}]\phi_{\mathbf{k}}(x,y,\theta)
\\&=\frac{1}{L_\x L_\y L_\rt}\sum_{k_1=-K_\mathrm{x}}^{K_\mathrm{x}}\sum_{k_2=-K_\mathrm{y}}^{K_\mathrm{y}}\sum_{k_3=-K_\mathrm{r}}^{K_\mathrm{r}}(-1)^{k_1+k_2}\widehat{\mathbf{F}}(\tau_{L_\x}(k_1)+1,\tau_{L_\y}(k_2)+1,\tau_{L_\rt}(k_3)+1)e^{2\pi\ii (k_1x+k_2y)}e^{\ii k_3\theta}.
\end{align*}
In some practical experiments, visualization of the approximation function $S_\mathbf{K}[f]$ versus the function $f$ on a grid finer than the sampling grid, used for computing the coefficients $\widehat{f}[\mathbf{k};2\mathbf{K}+1]$, is required. Next, we present matrix form for evaluation of $S_\mathbf{K}[f]$ on configuration $\mathbf{N}$-grid with $\mathbf{N}\in\mathbb{N}^3$ and $\mathbf{N}\ge\mathbf{L}$. 

Let $\mathbf{N}:=(N_\x,N_\y,N_\rt)\in\mathbb{N}^3$ with $\mathbf{N}\ge\mathbf{L}$. 
Suppose that $x_n':=\frac{-1}{2}+\frac{(n-1)}{N_\x}$, $y_m':=\frac{-1}{2}+\frac{(m-1)}{N_\y}$, and $\theta_l'=\frac{2\pi(l-1)}{N_\rt}$ for every $1\le n\le N_\x$, $1\le m\le N_\y$, and $1\le l\le N_\rt$. Then using Proposition \ref{LtoN1D}, we achieve 
\begin{equation}\label{MatSKfonG}
S_{\mathbf{K}}[f](x_n',y_m',\theta_l')=\frac{N_\x N_\y N_\rt}{L_\x L_\y L_\rt}\mathrm{iDFT}({\mathbf{Q}}_f)(n,m,l),
\end{equation}
where $\mathrm{iDFT}(\mathbf{Q}_f)$ is given by (\ref{3iDFT}) and the 3D array $\mathbf{Q}_f\in\mathbb{C}^{\mathbf{N}}$ is given by 
\begin{equation}\label{Qf}
\mathbf{Q}_{f}(n_\x,n_\y,n_\rt):=\left\{\begin{array}{lllllll}
\widehat{\mathbf{F}}(n_\x,n_\y,n_\rt) & {\rm if}\ \ \ (n_\x,n_\y,n_\rt)\in\mathbb{I}_\x\times \mathbb{I}_\y\times\mathbb{I}_\rt\\
\widehat{\mathbf{F}}(n_\x,n_\y,L_\rt+n_\rt-N_\rt) & {\rm if}\ \ \ (n_\x,n_\y,n_\rt)\in\mathbb{I}_\x\times \mathbb{I}_\y\times\mathbb{J}_\rt\\
\widehat{\mathbf{F}}(n_\x,L_\y+n_\y-N_\y,L_\rt+n_\rt-N_\rt) & {\rm if}\ \ \ (n_\x,n_\y,n_\rt)\in\mathbb{I}_\x\times \mathbb{J}_\y\times\mathbb{J}_\rt\\
\widehat{\mathbf{F}}(n_\x,L_\y+n_\y-N_\y,n_\rt) & {\rm if}\ \ \ (n_\x,n_\y,n_\rt)\in\mathbb{I}_\x\times \mathbb{J}_\y\times\mathbb{I}_\rt\\
\widehat{\mathbf{F}}(L_\x+n_\x-N_\x,L_\y+n_\y-N_\y,L_\rt+n_\rt-N_\rt) & {\rm if}\ \ \ (n_\x,n_\y,n_\rt)\in\mathbb{J}_\x\times \mathbb{J}_\y\times\mathbb{J}_\rt\\
\widehat{\mathbf{F}}(L_\x+n_\x-N_\x,n_\y,L_\rt+n_\rt-N_\rt) & {\rm if}\ \ \ (n_\x,n_\y,n_\rt)\in\mathbb{J}_\x\times \mathbb{I}_\y\times\mathbb{J}_\rt\\
\widehat{\mathbf{F}}(L_\x+n_\x-N_\x,L_\y+n_\y-N_\y,n_\rt) & {\rm if}\ \ \ (n_\x,n_\y,n_\rt)\in\mathbb{J}_\x\times \mathbb{J}_\y\times\mathbb{I}_\rt\\
\widehat{\mathbf{F}}(L_\x+n_\x-N_\x,n_\y,n_\rt) & {\rm if}\ \ \ (n_\x,n_\y,n_\rt)\in\mathbb{J}_\x\times \mathbb{I}_\y\times\mathbb{I}_\rt\\
0 & {\rm otherwise}
\end{array}
\right.,
\end{equation}
with $\mathbb{I}_\x:=1:K_\x+1$, $\mathbb{I}_\y:=1:K_\y+1$, $\mathbb{I}_\rt:=1:K_\rt+1$,
$\mathbb{J}_\x:=N_\x-K_\x+1:N_\x$, $\mathbb{J}_\y:=N_\y-K_\y+1:N_\y$ and $\mathbb{J}_\rt:=N_\rt-K_\rt+1:N_\rt.$ 
In particular, if $\mathbf{N}=\mathbf{L}$ then $S_{\mathbf{K}}[f]=\mathrm{iDFT}(\widehat{\mathbf{F}}).$

The matrix form (\ref{MatSKfonG}) is also attractive from computational perspectives.  The matrix form (\ref{MatSKfonG}) reads the multivariate summations in the right hand-side of  (\ref{SKvf}) into inverse discrete Fourier transform (iDFT) of zero-padded version of the DFT of sampled values of $f$ which can be performed using fast Fourier algorithms. 

Then we discuss Algorithm \ref{SKLAlg} using regular sampling on the fundamental domain $\Omega$ to approximate the Finite Fourier series of the form $S_{\mathbf{K}}(f)$ for function $f:SE(2)\to\mathbb{C}$ supported in the fundamental domain $\Omega$ with respect to a given visualization grid associated to the size vector $\mathbf{N}\in\mathbb{N}^3$.
\begin{algorithm}[H]
\caption{Computing approximation of the $S_\mathbf{K}(f)$ on $\mathbf{N}$-grid using FFT} 
\begin{algorithmic}[1]
\State{\bf input data} 

Given $f\in L^2(\Omega)$,  approximation order $\mathbf{K}\in\mathbb{N}^3$, 
and configuration size $\mathbf{N}\in\mathbb{N}^3$

\State {\bf output result} 

Values of the finite Fourier series $S_\mathbf{K}[f]$ on the configuration $\mathbf{N}$-grid\\ 
Put $\mathbf{L}:=(L_\x,L_\y,L_\rt)^T=2\mathbf{K}+1$\\
Generate the fundamental sampling $\mathbf{L}$-grid $\{(x_i,y_j,\theta_l):\mathbf{1}\le(i,j,l)\le\mathbf{L}\}$ according to (\ref{OmL}).\\
Generate the 3D array $\mathbf{F}\in\mathbb{C}^\mathbf{L}$ using $\mathbf{F}(i,j,l):=f(x_i,y_j,\theta_l)$ every $\mathbf{1}\le(i,j,l)\le\mathbf{L}$\\
Compute the 3D array $\widehat{\mathbf{F}}\in\mathbb{C}^{\mathbf{L}}$ according to (\ref{3DFT}) using FFT\\
Generate the 3D array $\mathbf{Q}_f\in\mathbb{C}^{\mathbf{N}}$ according to (\ref{Qf}).
\State Compute $S_{\mathbf{K}}[f]\in\mathbb{C}^\mathbf{N}$ according to (\ref{MatSKfonG}) using FFT.
\end{algorithmic} 
\label{SKLAlg}
\end{algorithm}

\subsection{Numerical experiments} 
This part is dedicated to illustrate  some numerical experiments related to finite Fourier series on the fundamental domain $\Omega$, using the discussed matrix approach in \ref{MatSKf}. We here implement some experiments in MATLAB for different class of functions on $SE(2)$ which are supported (absolutely/approximately) on the fundamental domain $\Omega$

To start with, we consider
functions which are continuously differentiable on $SE(2)$ and supported in the fundamental domain $\Omega$. In addition, functions approximately supported in fundamental domain will be investigated as well. In this direction, to also verify stability of the finite Fourier coefficient (FFC) numerical scheme for functions which are approximately supported in the fundamental domain $\Omega$, we will run error test experiments to verify the order of the absolute error  in Corollary \ref{fkOK}. 

A large class of continuously differentiable functions on $SE(2)$, which are absolutely supported in the fundamental domain $\Omega$, can be constructed via polar functions supported in circles. The canonical suggestion for such scaling of polar functions is  scaling of Bessel functions with respect to their zeros. 

\newpage
For $m,\ell\in\mathbb{Z}$ and $n\in\mathbb{N}$,  let the generalized polar harmonics $\Psi_{m,n}^\ell:SE(2)\to\mathbb{C}$ be given by
\begin{equation}
\Psi_{m,n}^\ell(r,\phi,\theta):=\left\{\begin{array}{ll}
\vspace{0.1cm}
J_m(2z_{m,n}r)e^{\ii m\phi}e^{\ii\ell\theta} & {\rm if}\ r\le \frac{1}{2}\\
0 & {\rm if}\ r>\frac{1}{2}
\end{array}
\right.,
\end{equation}
where $J_m$ is the $m$th-order Bessel function of the
first kind and $z_{m,n}$ is the $n$th positive zeros of $J_m(x)$.
 
Then $\Psi_{m,n}^\ell$ is a smooth function which is supported in $\mathbb{B}_{1/2}\times[0,2\pi)$. So it is also supported in the fundamental domain $\Omega$ and vanishing on the boundary of $\Omega$.

\begin{example}
Let $m=0$, $n=3$, and $\ell=0$. Suppose that $\mathbf{K}:=(25,26,40)$.

\begin{figure}[H]
\centering
\includegraphics[keepaspectratio=true,width=\textwidth, height=0.3\textheight]{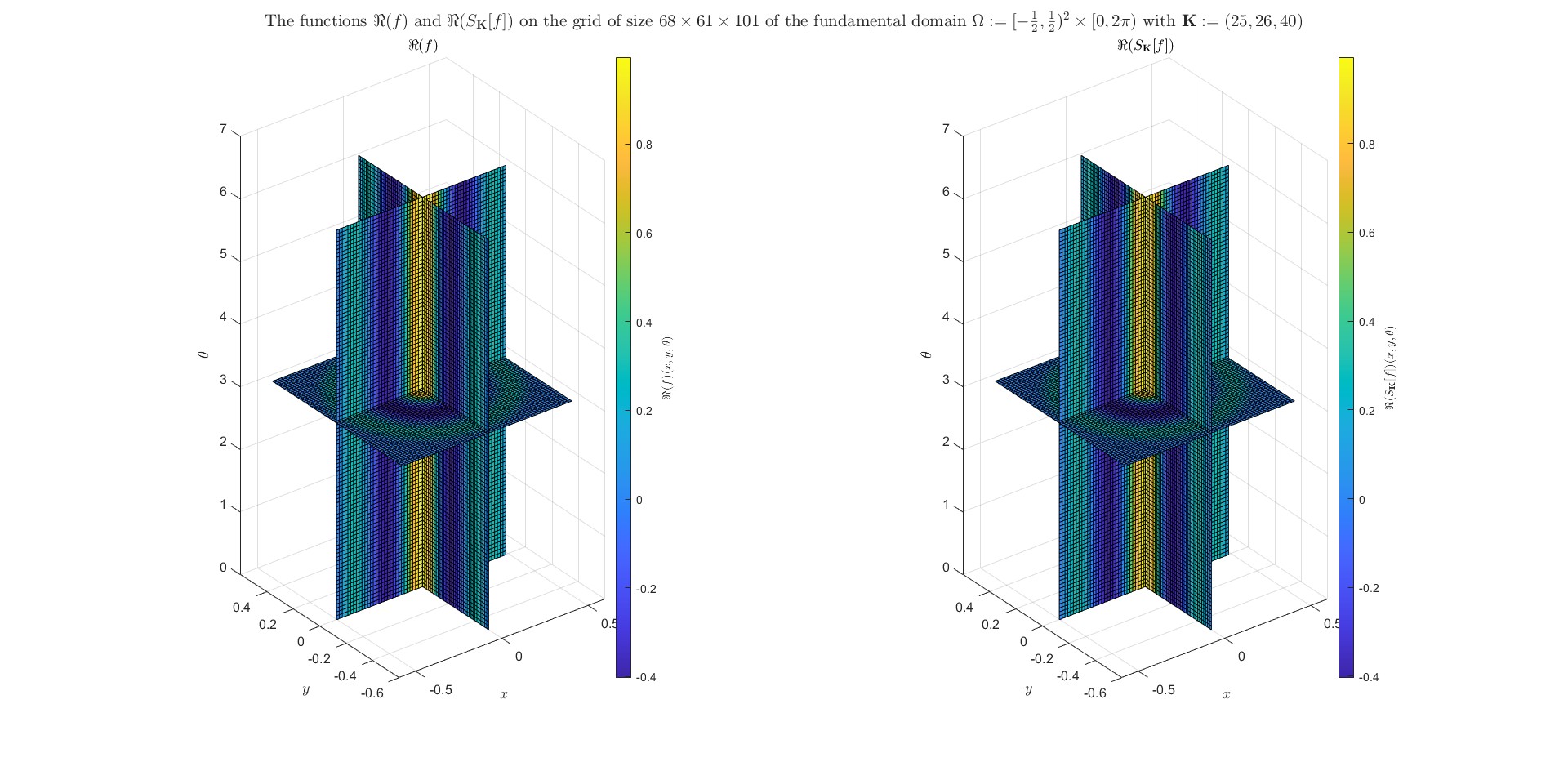}
\caption{The slice plot of $\Re(f)$ and $\Re(S_{\mathbf{K}}[f])$ on the uniform grid of size $68\times 61\times 101$ of the fundamental domain $\Omega$ at slices $[0,0,\pi]$,  where $f:=\Psi_{m,n}^\ell$.}
\label{fig:SlicP030Re}
\end{figure}
\begin{figure}[H]
\centering
\includegraphics[keepaspectratio=true,width=0.9\textwidth, height=0.3\textheight]{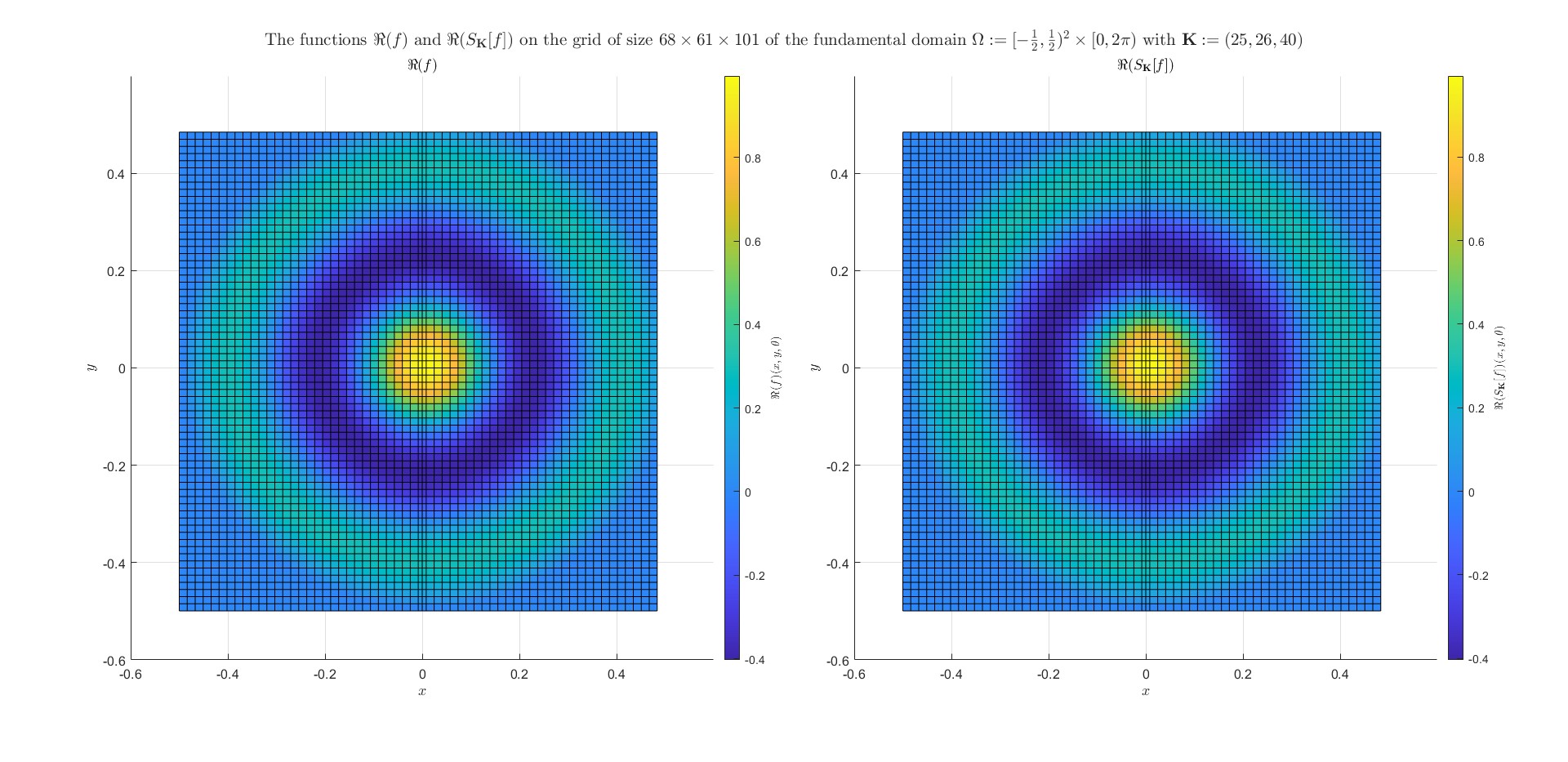}
\caption{The contour plot of $\Re(f)$ and $\Re(S_{\mathbf{K}}[f])$ on the uniform grid of size $68\times 61\times 101$ of the fundamental domain $\Omega$ at $\theta=\pi$, where $f:=\Psi_{m,n}^\ell$.}
\label{fig:ContP030Re}
\end{figure}
\end{example}

\newpage

Then we consider functions which are approximately supported in the fundamental domain $\Omega$.  A class of smooth functions which are approximately zero on the boundary and outside of the fundamental domain $\Omega$ can be considered as deformed 3D Gaussian and multidimentional Gaussian on $SE(2)$. 
    
{\bf 3D Deformed Gaussians.} Suppose that $\mathbf{H}\in M_2(\mathbb{R})$ is a positive-definite matrix. Assume that $s>0$ and $\nu\in[0,2\pi)$ are given.  Let $f_{\mathbf{H}}^{\nu,s}:SE(2)\to\mathbb{C}$ be the function given by 
\begin{equation}\label{fHsnu}
f_{\mathbf{H}}^{\nu,s}(\mathbf{x},\mathbf{R}_\theta):=e^{-\mathbf{x}^T\mathbf{H}\mathbf{x}}e^{-\frac{(\theta-\nu)^2}{s}},
\end{equation}
for every $(\mathbf{x},\mathbf{R}_\theta)\in SE(2)$.

\begin{example} 
Let $\mathbf{H}:=\mathrm{diag}(0.04,0.1)^{-1}$, $s:=0.4$, and $\nu:=\pi$.
Then $f_{s,\mathbf{H}}$ given by (\ref{fHsnu}) is approximately supported in the fundamental domain $\Omega$.  

To begin with, we consider experiment for analysis of structural behavior of the absolute error of approximating Fourier coefficients given by (\ref{fk}) using the finite Fourier coefficients (FFCs) of the form (\ref{fkLvMain}). In details, we visualize the absolute error values given by 
\[
\left|\widehat{f_{\mathbf{H}}^{\nu,s}}[\mathbf{k}]-\widehat{f_{\mathbf{H}}^{\nu,s}}[\mathbf{k};2K+1]\right|,
\]
for a given  $\mathbf{k}\in\mathbb{Z}^3$. 

Assume that $\mathbf{k}:=(1,-2,3)^T\in\mathbb{Z}^3$ is fixed. We here implemented an experiment which computed the absolute error $|\widehat{f_{\mathbf{H}}^{\nu,s}}[\mathbf{k}]-\widehat{f_{\mathbf{H}}^{\nu,s}}[\mathbf{k};2K+1]|$ for all integer values $K=1:30$. 
\vspace{1cm}
\begin{figure}[H]
\centering
\includegraphics[keepaspectratio=true,width=18cm, height=15cm]{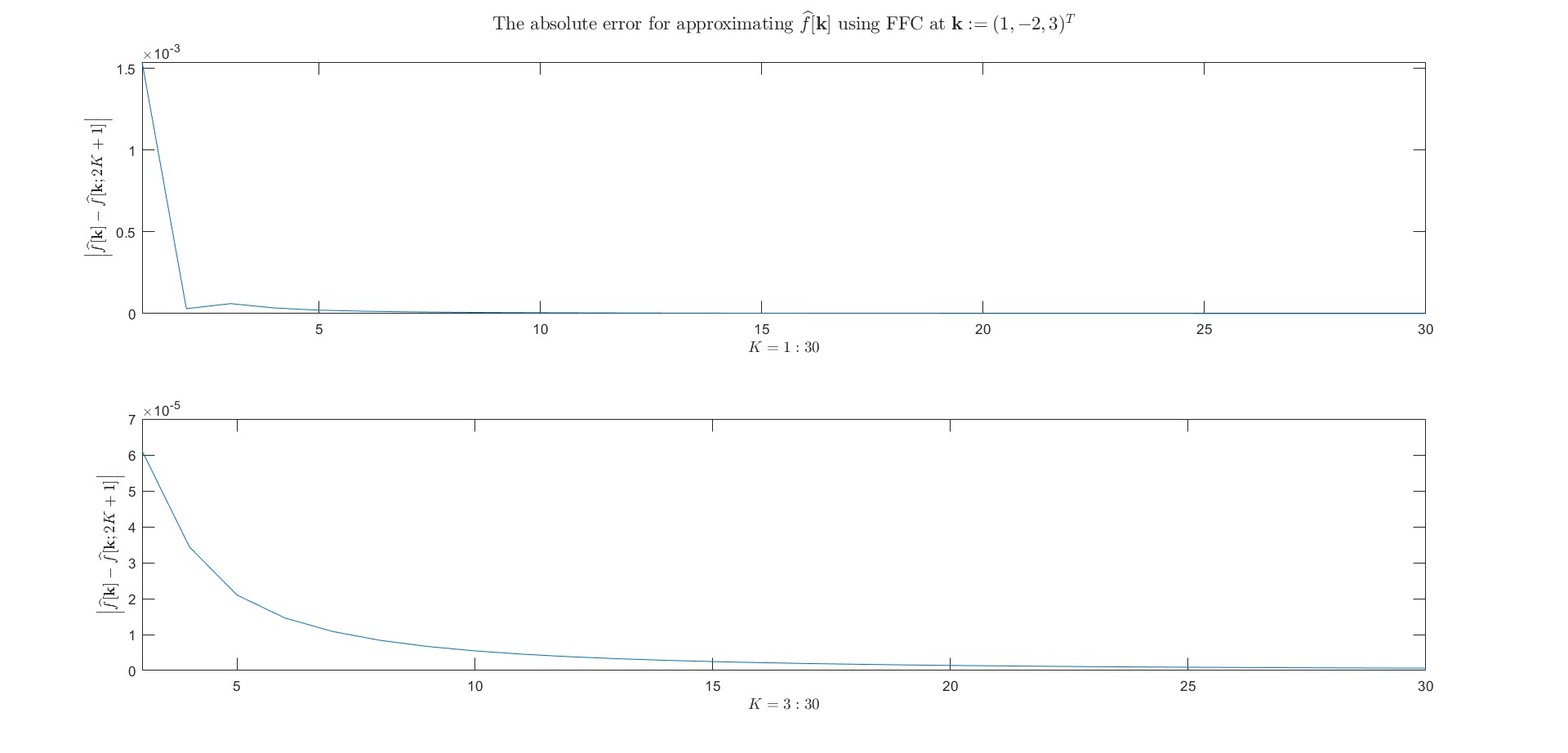}
\caption{The absolute error approximation for the Fourier coefficient $\widehat{f_{\mathbf{H}}^{\nu,s}}[\mathbf{k}]$ with $\mathbf{k}:=(1,-2,3)^T\in\mathbb{Z}^3$ using finite Fourier coefficients (FFCs) of the form $\widehat{f_{\mathbf{H}}^{\nu,s}}[\mathbf{k};2K+1]$ given by (\ref{fkLvM}).  The (upper) plot displays the generic behavior of the absolute error values $|\widehat{f_{\mathbf{H}}^{\nu,s}}[\mathbf{k}]-\widehat{f_{\mathbf{H}}^{\nu,s}}[\mathbf{k};2K+1]|$ for all integer values $1\le K\le 30$. As Corollary   \ref{fkOK} guarantees, the absolute error is $\mathcal{O}(1/K)$ if $\|\mathbf{k}\|_\infty\le K$. Since $\|\mathbf{k}\|_\infty=3$, to have better visualization of the structural behavior of the values $|\widehat{f_{\mathbf{H}}^{\nu,s}}[\mathbf{k}]-\widehat{f_{\mathbf{H}}^{\nu,s}}[\mathbf{k};2K+1]|$, we then plotted the absolute error values with respect to $K=3:30$. }
\label{fig:ErrFFCfsH}
\end{figure}

\newpage

\begin{figure}[H]
\centering
\includegraphics[keepaspectratio=true,width=18cm, height=15cm]{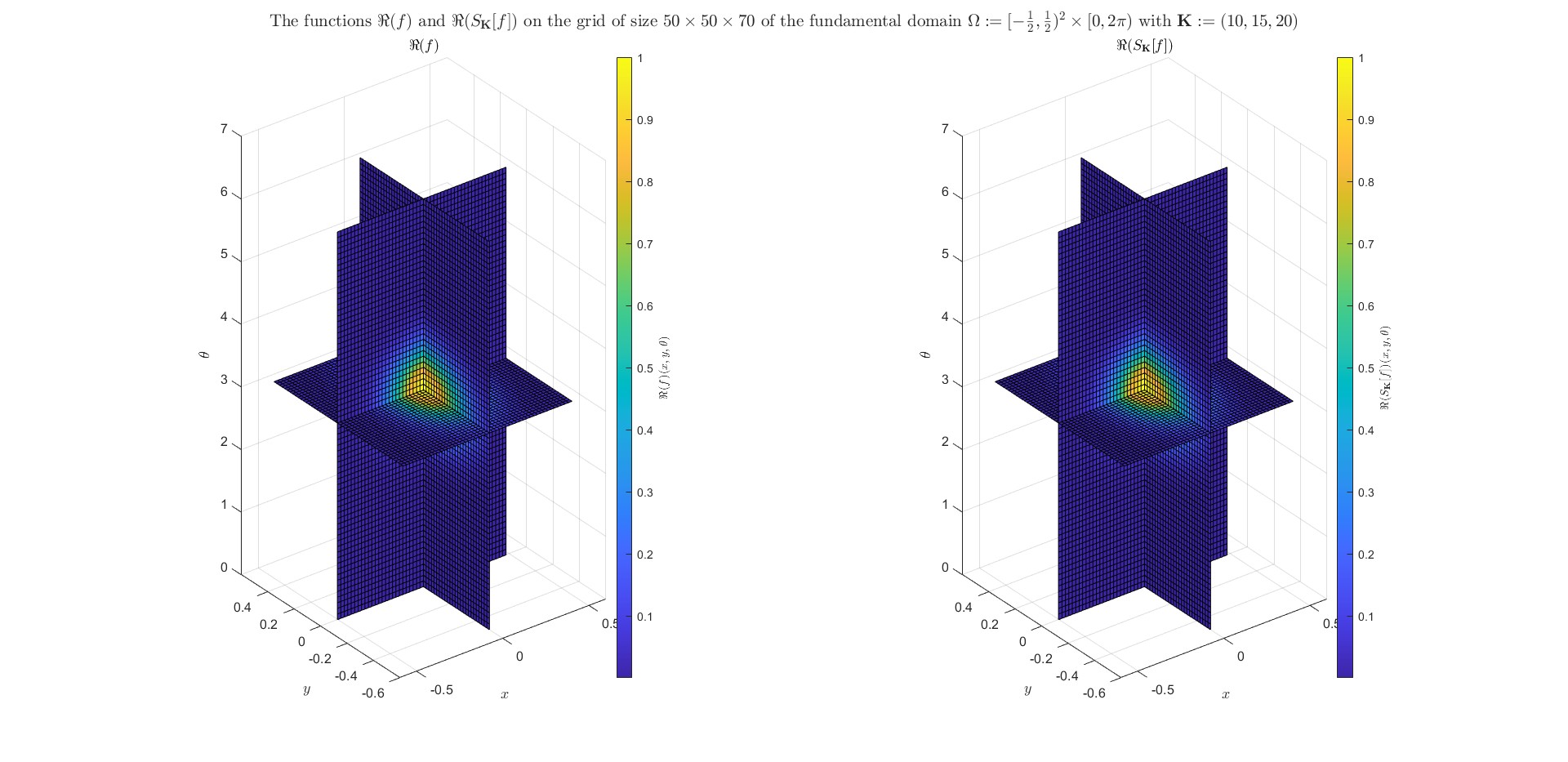}
\caption{The slice plot of $f_{\mathbf{H}}^{\nu,s}$ and $\Re(S_{\mathbf{K}}[f_{\mathbf{H}}^{\nu,s}])$ on the uniform grid of size $50\times 50\times 70$ of the fundamental domain $\Omega$ at slices $[0,0,\pi]$, where $\mathbf{H}:=\mathrm{diag}(0.04,0.1)^{-1}$, $s:=0.4$, $\nu:=\pi$, and $\mathbf{K}:=(10,15,20)$.}
\label{fig:Slic3DG}
\end{figure}

\begin{figure}[H]
\centering
\includegraphics[keepaspectratio=true,width=18cm, height=15cm]{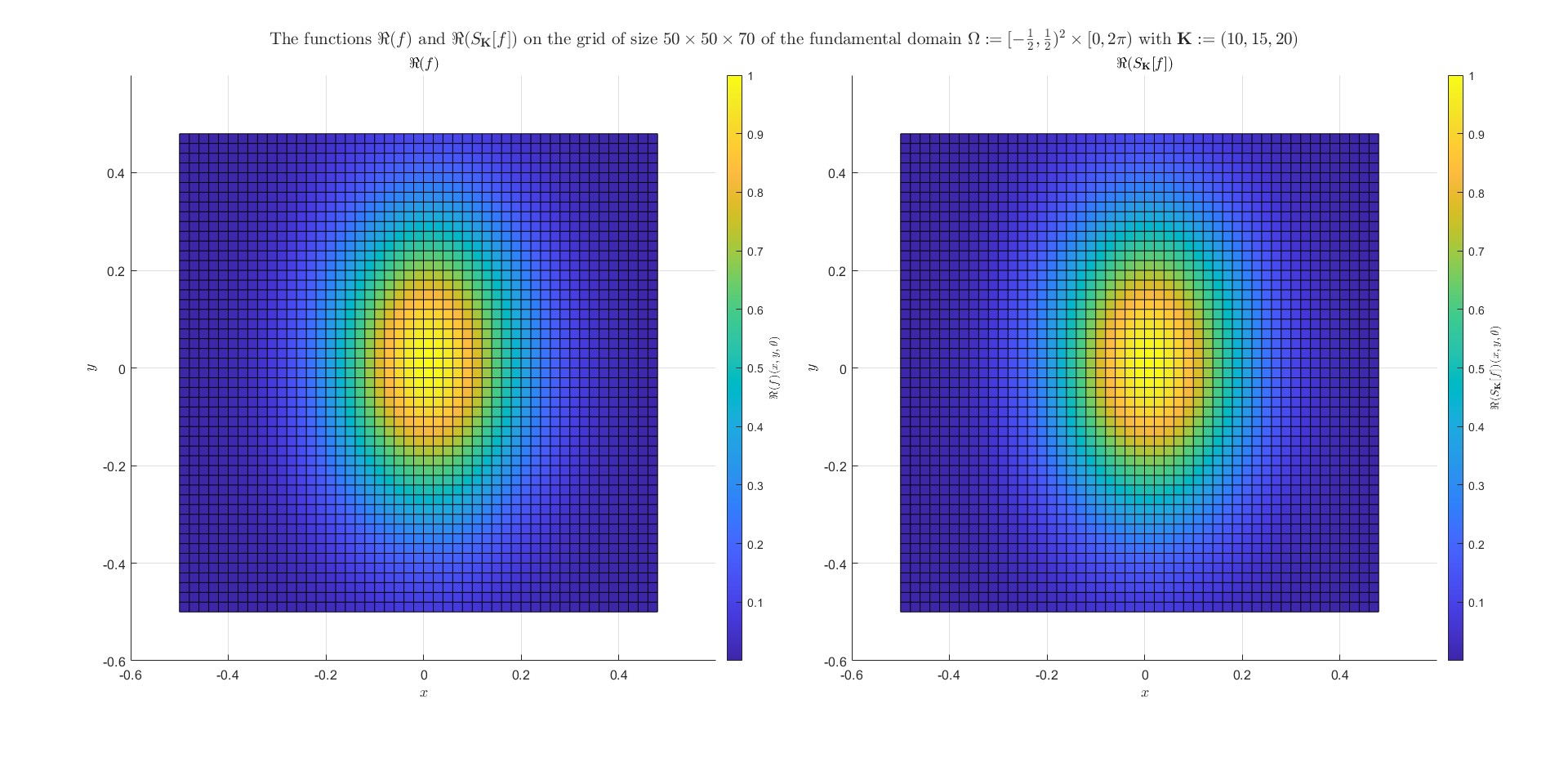}
\caption{The contour plot of $f_{\mathbf{H}}^{\nu,s}$ and $\Re(S_{\mathbf{K}}[f_{\mathbf{H}}^{\nu,s}])$ on the uniform grid of size $50\times 50\times 70$ of the fundamental domain $\Omega$ at $\theta=\pi$, where $\mathbf{H}:=\mathrm{diag}(0.04,0.1)^{-1}$, $s:=0.4$, $\nu:=\pi$, and $\mathbf{K}:=(10,15,20)$.}
\label{fig:Cont3DG}
\end{figure}
\end{example}

\newpage
{\bf Multidimentional Gaussian on $SE(2)$.} Suppose that $\mathbf{\Sigma}\in M_3(\mathbb{R})$ is positive-definite and $\bfbeta\in SE(2)$. Then  $f_{\bfbeta,\bf{\Sigma}}:SE(2)\to\mathbb{R}$ given by 
\begin{equation}\label{fbS}
f_{\bfbeta,\bf{\Sigma}}(\mathbf{g})=\exp\left(\frac{-{\log(\bfbeta^{-1}\circ\mathbf{g})^\lor}^T\mathbf{\Sigma}^{-1}\log(\bfbeta^{-1}\circ\mathbf{g})^\lor}{2}\right),
\end{equation}
is called the multidimensional Gaussian with mean $\bfbeta$ and the covariance $\mathbf{\Sigma}$ on $SE(2)$, where $\log:SE(2)\to\mathfrak{se}(2)$ is the matrix logarithm map, if elements of $SE(2)$ are considered as matrices of the form (\ref{gxyt}) and $^\lor:\mathfrak{se}(2)\to\mathbb{R}^3$ is the identification map.

\begin{example}
Let $\bfbeta:=(\mathbf{0},\mathbf{R}_\pi)$ and $\mathbf{\Sigma}:=\sigma^2\mathbf{I}_3$ with $\sigma^2:=0.05$. Then $f_{\bfbeta,\bf{\Sigma}}$ given by (\ref{fbS}) is approximately supported in the fundamental domain $\Omega$.  

To start with, we consider experiment for analysis of structural behavior of the absolute error of approximating Fourier coefficients given by (\ref{fk}) using the finite Fourier coefficients (FFCs) of the form (\ref{fkLvMain}). In details, we visualize the absolute error values given by 
\[
\left|\widehat{f_{\bfbeta,\bf{\Sigma}}}[\mathbf{k}]-\widehat{f_{\bfbeta,\bf{\Sigma}}}[\mathbf{k};2K+1]\right|,
\]
for a given  $\mathbf{k}\in\mathbb{Z}^3$. 

Assume that $\mathbf{k}:=(-1,2,-4)^T\in\mathbb{Z}^3$ is fixed. We here implemented an experiment which computed the absolute error $\left|\widehat{f_{\bfbeta,\bf{\Sigma}}}[\mathbf{k}]-\widehat{f_{\bfbeta,\bf{\Sigma}}}[\mathbf{k};2K+1]\right|$ for all integer values $K=1:30$. 
\begin{figure}[H]
\centering
\includegraphics[keepaspectratio=true,width=18cm, height=15cm]{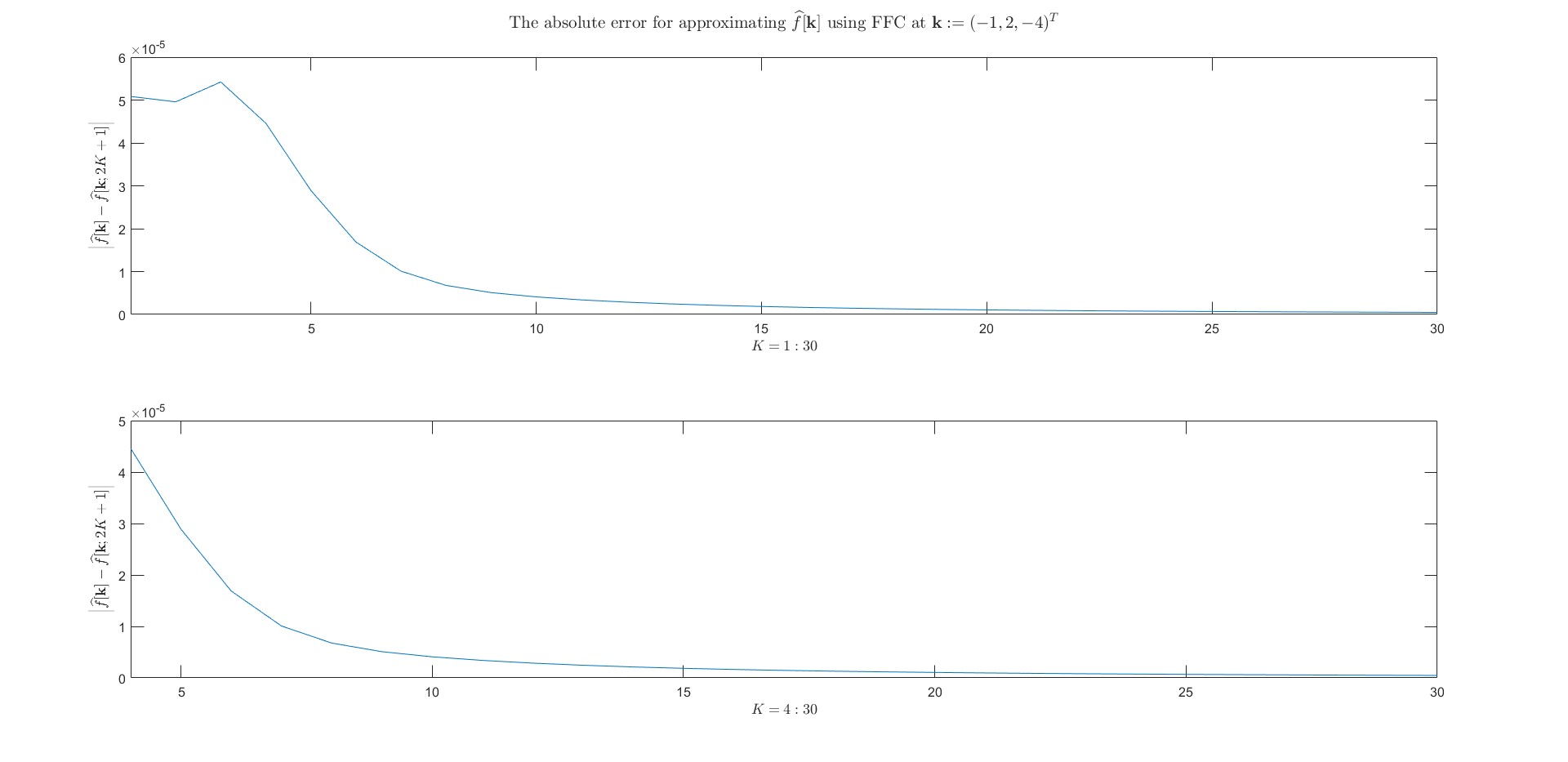}
\caption{The absolute error approximation for the Fourier coefficient $\widehat{f_{\bfbeta,\bf{\Sigma}}}[\mathbf{k}]$ with $\mathbf{k}:=(-1,2,-4)^T\in\mathbb{Z}^3$ using finite Fourier coefficients (FFCs) of the form $\widehat{f_{\bfbeta,\bf{\Sigma}}}[\mathbf{k};2K+1]$ given by (\ref{fkLvM}).  The (upper) plot displays the generic behavior of the absolute error values $|\widehat{f_{\bfbeta,\bf{\Sigma}}}[\mathbf{k}]-\widehat{f_{\bfbeta,\bf{\Sigma}}}[\mathbf{k};2K+1]|$ for all integer values $1\le K\le 30$. As Corollary \ref{fkOK} guarantees, the absolute error is $\mathcal{O}(1/K)$ if $\|\mathbf{k}\|_\infty\le K$. Since $\|\mathbf{k}\|_\infty=4$, to have better visualization of the structural behavior of the values $|\widehat{f_{\bfbeta,\bf{\Sigma}}}[\mathbf{k}]-\widehat{f_{\bfbeta,\bf{\Sigma}}}[\mathbf{k};2K+1]|$, we then plotted the absolute error values with respect to $K=4:30$. }
\label{fig:ErrFFCfbS}
\end{figure}

\newpage
\begin{figure}[H]
\centering
\includegraphics[keepaspectratio=true,width=18cm, height=15cm]{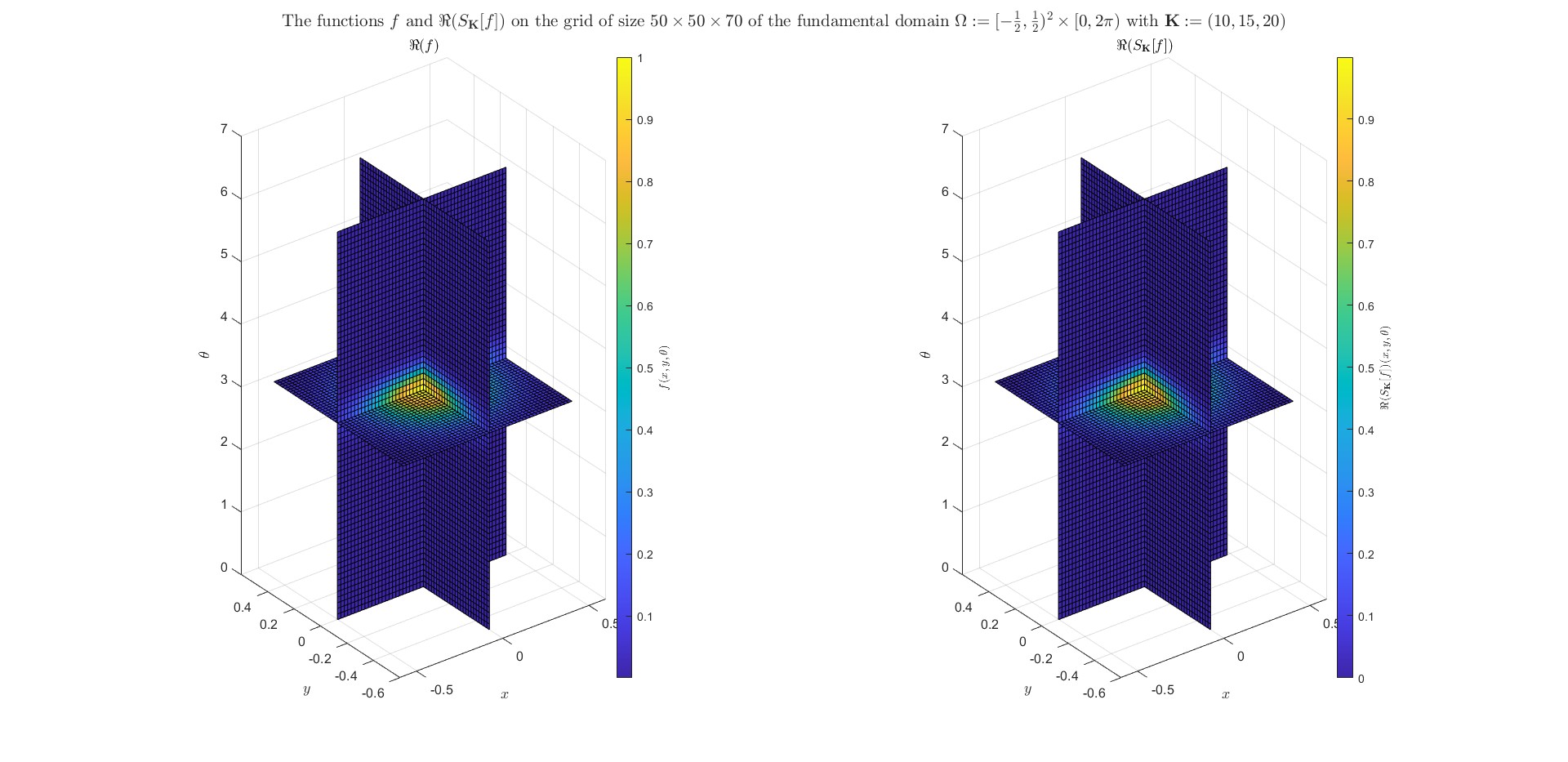}
\caption{The slice plot of $f_{\bfbeta,\bf{\Sigma}}$ and $\Re(S_{\mathbf{K}}[f_{\bfbeta,\bf{\Sigma}}])$ on the uniform grid of size $50\times 50\times 70$ of the fundamental domain $\Omega$ at slices $[0,0,\pi]$, where $\bfbeta:=(\mathbf{0},\mathbf{R}_\pi)$ and $\mathbf{\Sigma}:=\sigma^2\mathbf{I}_3$ with $\sigma^2:=0.05$.}
\label{fig:SlicSE2G}
\end{figure}

\begin{figure}[H]
\centering
\includegraphics[keepaspectratio=true,width=18cm, height=15cm]{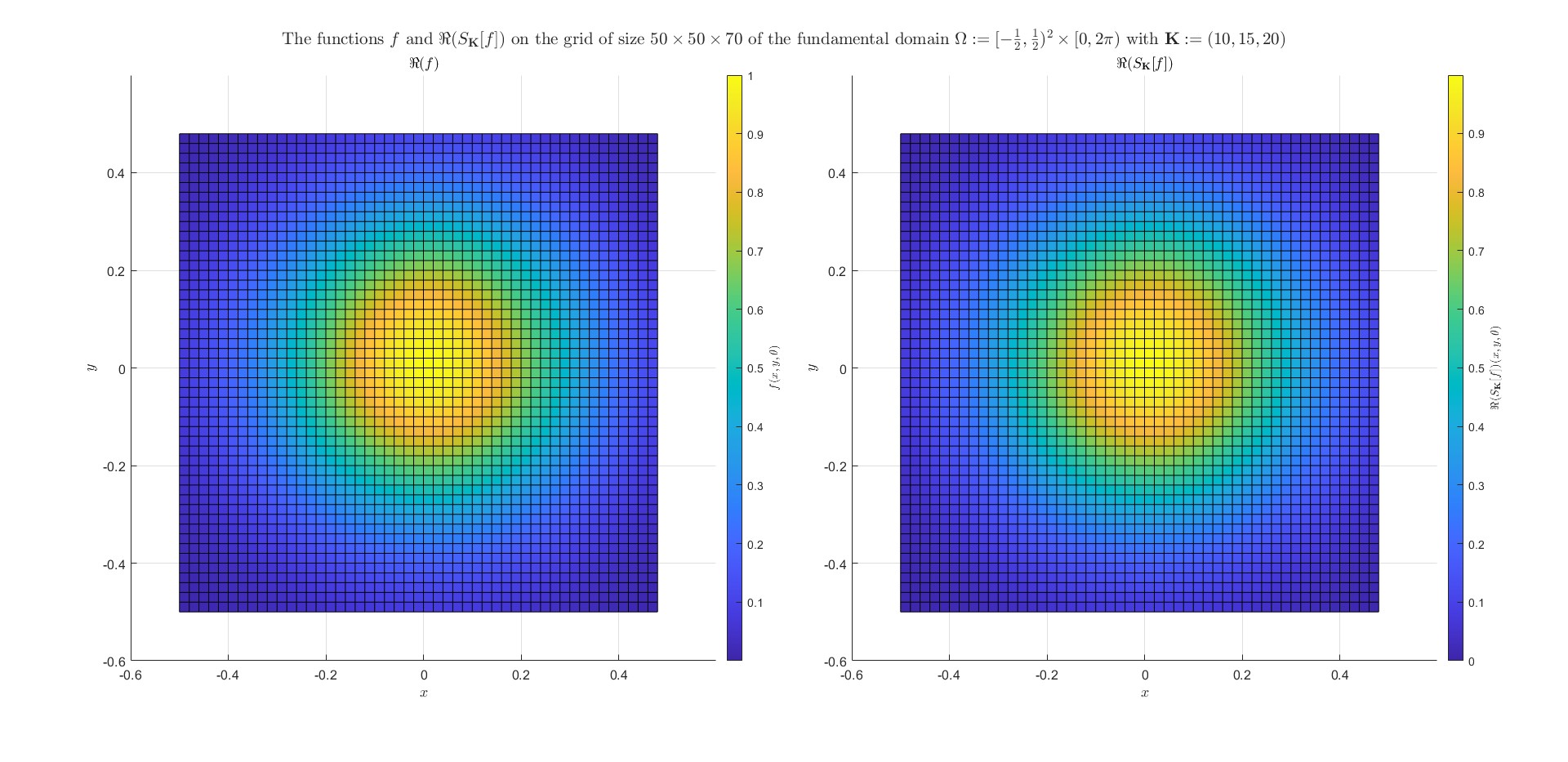}
\caption{The contour plot of $f_{\bfbeta,\bf{\Sigma}}$ and $\Re(S_{\mathbf{K}}[f_{\bfbeta,\bf{\Sigma}}])$ on the uniform grid of size $50\times 50\times 70$ of the fundamental domain $\Omega$ at $\theta=\pi$, where $\bfbeta:=(\mathbf{0},\mathbf{R}_\pi)$ and $\mathbf{\Sigma}:=\sigma^2\mathbf{I}_3$ with $\sigma^2:=0.05$.}
\label{fig:ContSE2G}
\end{figure}
\end{example}

\newpage
\section{\bf Convolutional Finite Fourier Series on $\mathbb{Z}^2\backslash SE(2)$}\label{FSCR}

This section discusses structure of convolutional finite Fourier series as a numerical approximation for convolution of functions on the right coset space $\mathbb{X}:=\mathbb{Z}^2\backslash SE(2)$ with functions on $SE(2)$ which are radial in translations. 
To begin with, we investigate structure of Fourier coefficients of $SE(2)$-convolutions with functions which are radial in translations and supported in the fundamental domain $\Omega$. Then we consider Fourier coefficients of convolution of functions on $SE(2)$ which are radial in translations with functions on $\mathbb{Z}^2\backslash SE(2)$. 

\subsection{Structure of Fourier coefficients of convolutions}\label{CNV-FS-O}
This part presents the $SE(2)$-convolution property of Fourier series of functions supported in $\Omega$. In details, we study the Fourier coefficients of $SE(2)$-convolutions supported in $\Omega$.

For $f_j\in L^1(SE(2))$ with $j\in\{1,2\}$, the $SE(2)$-convolution $f_1\star f_2\in L^1(SE(2))$ is given by 
\begin{equation}\label{fstarg}
(f_1\star f_2)(\mathbf{h})=\int_{SE(2)}f_1(\mathbf{g})f_2(\mathbf{g}^{-1}\circ\mathbf{h})\D\mathbf{g},\hspace{1cm}{\rm for}\ \ \mathbf{h}\in SE(2).
\end{equation} 
Suppose $\mathbf{h}=(\mathbf{x},\mathbf{R})$ and $\mathbf{g}=(\mathbf{y},\mathbf{S})$ where $\mathbf{x},\mathbf{y}\in\mathbb{R}^2$ and $\mathbf{R},\mathbf{S}\in SO(2)$. Then
\begin{align*}
\mathbf{g}^{-1}\circ\mathbf{h}&=(\mathbf{y},\mathbf{S})^{-1}\circ(\mathbf{x},\mathbf{R})
=(-\mathbf{S}^{-1}\mathbf{y},\mathbf{S}^{-1})\circ(\mathbf{x},\mathbf{R})
\\&=(-\mathbf{S}^{-1}\mathbf{y}+\mathbf{S}^{-1}\mathbf{x},\mathbf{S}^{-1}\mathbf{R})
=(\mathbf{S}^{-1}(\mathbf{x}-\mathbf{y}),\mathbf{S}^{-1}\mathbf{R}).
\end{align*}
Therefore,  (\ref{fstarg}) reduces to 
\begin{equation}\label{f*g}
(f_1\star f_2)(\mathbf{x},\mathbf{R})=\int_{SO(2)}\int_{\mathbb{R}^2}f_1(\mathbf{y},\mathbf{S})f_2(\mathbf{S}^{-1}(\mathbf{x}-\mathbf{y}),\mathbf{S}^{-1}\mathbf{R})\D\mathbf{y}\D\mathbf{S}.
\end{equation}
In addition, $\mathrm{supp}(f_1\star f_2)
\subseteq\mathrm{supp}(f_1)\circ\mathrm{supp}(f_2)$,
where $\circ$ denotes the $SE(2)$ group operation given in (\ref{SE2law}).
So, if $f_1,f_2$ have compact supports then $f_1\star f_2$ is compactly supported as well. In particular, if $f_j\in L^1(SE(2))$ ($j\in\{1,2\}$) are functions supported in $\mathbb{D}_{a}$ (resp. $\mathbb{D}_b$). Then $f_1\star f_2$ is supported in $\mathbb{D}_{a+b}$, see Proposition 3.1 of \cite{AGHF.GSC.JAT}.

Next, we prove $SE(2)$-convolution property for functions of the form $f\star \rho$ where $f$ is arbitrary and $\rho$ is radial in translations. \footnote{A function $\rho:SE(2)\to\mathbb{C}$ is called radial in translations if $\rho(\mathbf{x},\mathbf{R})=\rho(\mathbf{Sx},\mathbf{R}),$ 
for every $\mathbf{R},\mathbf{S}\in SO(2)$ and $\mathbf{x}\in\mathbb{R}^2$.}

\begin{theorem}\label{f*gP}
{\it Let $f,\rho\in L^1(SE(2))$ be functions supported in $\mathbb{D}_{1/4}$ with $\rho$ radial in translations. 
\begin{enumerate}
\item If $(\mathbf{x},\mathbf{R}_\theta)\in SE(2)$ then 
\begin{equation}\label{f*gRadialMain}
f\star\rho(\mathbf{x},\mathbf{R}_\theta)=\frac{1}{2\pi}\int_{0}^{2\pi}\int_{\mathbb{R}^2}f(\mathbf{y},\mathbf{R}_\alpha)\rho(\mathbf{x}-\mathbf{y},\mathbf{R}_{\theta-\alpha})\D\mathbf{y}\D\alpha.
\end{equation}
\item If $\mathbf{k}\in\mathbb{I}$ then $\widehat{f\star\rho}[\mathbf{k}]=\widehat{f}[\mathbf{k}]\widehat{\rho}[\mathbf{k}]$.
\end{enumerate}
}\end{theorem}
\begin{proof}
Suppose $f,\rho\in L^1(SE(2))$ are functions supported in $\mathbb{D}_{1/4}$. Then $f\star \rho$ is supported in $\mathbb{D}_{1/2}$. So, $f\star \rho$ is supported in $\Omega$. 
(1) Since $\rho$ is radial in translations, we get  
\begin{equation}\label{gRadial}
\rho(\mathbf{S}^{-1}(\mathbf{x}-\mathbf{y}),\mathbf{S}^{-1}\mathbf{R})=\rho(\mathbf{x}-\mathbf{y},\mathbf{S}^{-1}\mathbf{R}).
\end{equation}
Therefore, applying (\ref{gRadial}) in (\ref{f*g}), we achieve  
\begin{equation}\label{f*gRadial}
f\star \rho(\mathbf{x},\mathbf{R})=\int_{SO(2)}\int_{\mathbb{R}^2}f(\mathbf{y},\mathbf{S})\rho(\mathbf{x}-\mathbf{y},\mathbf{S}^{-1}\mathbf{R})\D\mathbf{y}\D\mathbf{S}
\end{equation}
which implies the following explicit form 
\begin{equation}
f\star \rho(\mathbf{x},\mathbf{R}_\theta)=\frac{1}{2\pi}\int_{0}^{2\pi}\int_{\mathbb{R}^2}f(\mathbf{y},\mathbf{R}_\alpha)\rho(\mathbf{x}-\mathbf{y},\mathbf{R}_{\theta-\alpha})\D\mathbf{y}\D\alpha
\end{equation}
for every $(\mathbf{x},\mathbf{R}_\theta)\in SE(2)$ with $\mathbf{x}\in\mathbb{R}^2$ and $\theta\in [0,2\pi)$.

(2) Applying (\ref{pkxy}) and (\ref{f*gRadialMain}) we get 
\begin{align*}
\widehat{f\star \rho}[\mathbf{k}]&=\langle f\star \rho,\phi_{\mathbf{k}}\rangle_{L^2(\Omega)}
=\int_{\Omega}f\star \rho(\omega)\overline{\phi_{\mathbf{k}}(\omega)}\D\omega
=\frac{1}{2\pi}\int_{0}^{2\pi}\int_{\mathbb{R}^2}f\star \rho(\mathbf{x},\mathbf{R}_\theta)\overline{\phi_{\mathbf{k}}(\mathbf{x},\mathbf{R}_\theta)}\D\mathbf{x}\D\theta
\\&=\frac{1}{4\pi^2}\int_{0}^{2\pi}\int_{\mathbb{R}^2}\left(\int_{0}^{2\pi}\int_{\mathbb{R}^2}f(\mathbf{y},\mathbf{R}_\alpha)\rho(\mathbf{x}-\mathbf{y},\mathbf{R}_{\theta-\alpha})\D\mathbf{y}\D\alpha\right)\overline{\phi_{\mathbf{k}}(\mathbf{x},\mathbf{R}_\theta)}\D\mathbf{x}\D\theta
\\&=\frac{1}{4\pi^2}\int_{0}^{2\pi}\int_{\mathbb{R}^2}f(\mathbf{y},\mathbf{R}_\alpha)\left(\int_{0}^{2\pi}\int_{\mathbb{R}^2}\rho(\mathbf{x}-\mathbf{y},\mathbf{R}_{\theta-\alpha})\overline{\phi_{\mathbf{k}}(\mathbf{x},\mathbf{R}_\theta)}\D\mathbf{x}\D\theta\right)\D\mathbf{y}\D\alpha
\\&=\frac{1}{4\pi^2}\int_{0}^{2\pi}\int_{\mathbb{R}^2}f(\mathbf{y},\mathbf{R}_\alpha)\left(\int_{0}^{2\pi}\int_{\mathbb{R}^2}\rho(\mathbf{x},\mathbf{R}_{\theta})\overline{\phi_{\mathbf{k}}(\mathbf{x}+\mathbf{y},\mathbf{R}_{\theta+\alpha})}\D\mathbf{x}\D\theta\right)\D\mathbf{y}\D\alpha
\\&=\frac{1}{4\pi^2}\int_{0}^{2\pi}\int_{\mathbb{R}^2}f(\mathbf{y},\mathbf{R}_\alpha)\overline{\phi_{\mathbf{k}}(\mathbf{y},\mathbf{R}_{\alpha})}\left(\int_{0}^{2\pi}\int_{\mathbb{R}^2}\rho(\mathbf{x},\mathbf{R}_{\theta})\overline{\phi_{\mathbf{k}}(\mathbf{x},\mathbf{R}_{\theta})}\D\mathbf{x}\D\theta\right)\D\mathbf{y}\D\alpha
\\&=\frac{1}{4\pi^2}\left(\int_{0}^{2\pi}\int_{\mathbb{R}^2}f(\mathbf{y},\mathbf{R}_\alpha)\overline{\phi_{\mathbf{k}}(\mathbf{y},\mathbf{R}_{\alpha})}\D\mathbf{y}\D\alpha\right)\left(\int_{0}^{2\pi}\int_{\mathbb{R}^2}\rho(\mathbf{x},\mathbf{R}_{\theta})\overline{\phi_{\mathbf{k}}(\mathbf{x},\mathbf{R}_{\theta})}\D\mathbf{x}\D\theta\right)
\\&=\langle f,\phi_\mathbf{k}\rangle_{L^2(\Omega)}\langle \rho,\phi_\mathbf{k}\rangle_{L^2(\Omega)}=\widehat{f}[\mathbf{k}]\widehat{\rho}[\mathbf{k}].\qedhere
\end{align*}
\end{proof}

Then, we discuss structure of Fourier series for convolution of functions on $SE(2)$ which are radial in translations with functions on $\mathbb{X}$.

The unified theory for convolution of functions on coset (homogeneous) spaces of compact subgroups of locally compact groups studied in \cite{AGHF.BMMSS, AGHF.IJM, AGHF.BBMSS}.  The later theory has been applied for the case of compact subgroups of special Euclidean groups in \cite{Duits.PDE-G-CNN} and for the case of compact groups in \cite{Kondor}. This unified theory can be canonically reformulated for convolution integrals on right coset space of compact subgroups in $SE(2)$ which is not the case for $\mathbb{L}\backslash SE(2)$. In the case of coset spaces of arbitrary discrete subgroups of $SE(2)$, the unified structure of convolution module action of $SE(2)$ developed in \cite{AGHF.GSC.PAMQ}. 

Let $f\in L^1(SE(2))$ and $\psi\in L^1(\mathbb{X},\mu)$.  The convolution of $f$ with $\psi$ is defined as the function $\psi\oslash f:\mathbb{X}\to\mathbb{C}$ given by  
\begin{equation}\label{oslash}
(\psi\oslash f)(\mathbb{L}\mathbf{g}):=\int_{SE(2)} \psi(\mathbb{L}\mathbf{h})f(\mathbf{h}^{-1}\circ\mathbf{g})\dd\mathbf{h},
\end{equation}
for $\mathbf{g}\in SE(2)$.
Since $SE(2)$ is a unimodular group, we get 
\[
(\psi\oslash f)(\mathbb{L}\mathbf{g})=\int_{SE(2)} \psi(\mathbb{L}\mathbf{g}\circ\mathbf{h})f(\mathbf{h}^{-1})\dd\mathbf{h},
\]
for $f\in L^1(SE(2))$,  $\psi\in L^1(\mathbb{X},\mu)$,  and $\mathbf{g}\in SE(2)$. It is also shown that if $f,g\in L^1(SE(2))$ then  
\begin{equation}\label{Tf*g}
\widetilde{(f\star g)}(\mathbb{L}\mathbf{g})=(\widetilde{f}\oslash g)(\mathbb{L}\mathbf{g}),
\end{equation}  
for every $\mathbf{g}\in SE(2)$, see Theorem 4.2 of \cite{AGHF.GSC.PAMQ}.

\begin{proposition}
{\it Let $f,\rho\in L^1(SE(2))$ be functions supported in $\mathbb{D}_{1/4}$ with $\rho$ radial in translations. Suppose that $\psi:=\widetilde{f}$ and $\mathbf{k}\in\mathbb{I}$. Then 
\[
\widehat{\psi\oslash \rho}\{\mathbf{k}\}=\widehat{\psi}\{\mathbf{k}\}\widehat{\rho}[\mathbf{k}].
\]
}\end{proposition}
\begin{proof}
Applying (\ref{psikf}), Theorem \ref{f*gP}(2) and (\ref{Tf*g}), we achieve 
\begin{align*}
\widehat{\psi\oslash \rho}\{\mathbf{k}\}
&=\widehat{\widetilde{f}\oslash \rho}\{\mathbf{k}\}=\widehat{\widetilde{f\star \rho}}\{\mathbf{k}\}
\\&=\widehat{f\star \rho}[\mathbf{k}]=\widehat{f}[\mathbf{k}]\widehat{\rho}[\mathbf{k}]=\widehat{\psi}\{\mathbf{k}\}\widehat{\rho}[\mathbf{k}].\qedhere
\end{align*}
\end{proof}

\subsection{Convolutional finite Fourier series}\label{CNV-FFS}
We here develop a unified notion of convolutional finite Fourier series, which can be considered as a fast and accurate computable approximation for convolutions of functions on $\mathbb{Z}^2\backslash SE(2)$ with functions on $SE(2)$ which are radial in translations. 

To begin with, we consider structure of the unified numerical scheme for convolutions supported in the fundamental domain $\Omega$. 

Let $f,\rho:SE(2)\to\mathbb{C}$ be integrable functions supported in $\mathbb{D}_{1/4}$ with $\rho$ radial in translations.  Then $f\star \rho$ is supported in $\Omega$ .  Assume $\mathbf{K}\in\mathbb{N}^3$ and $\mathbf{L}:=2\mathbf{K}+1$.  Then finite Fourier series of order $\mathbf{K}\in\mathbb{N}^3$ of the function $f\star \rho$,  given by (\ref{SKf}), reads as 
\[
S_\mathbf{K}[f\star \rho](\omega)=\sum_{|\mathbf{k}|\le\mathbf{K}}\widehat{f\star\rho}[\mathbf{k};\mathbf{L}]\phi_{\mathbf{k}}(\omega),
\]
for every $\omega\in\Omega$, where 
\begin{equation}\label{f*gkLv}
\widehat{f\star \rho}[\mathbf{k};\mathbf{L}]=\frac{1}{L_\mathrm{x}L_\mathrm{y}L_\mathrm{r}}\sum_{i=1}^{L_\mathrm{x}}\sum_{j=1}^{L_\mathrm{y}}\sum_{l=1}^{L_\mathrm{r}}f\star \rho(x_i,y_j,\theta_l)\overline{\phi_\mathbf{k}(x_i,y_j,\theta_l)},
\end{equation}
with $x_i:=-\frac{1}{2}+\frac{(i-1)}{L_\mathrm{x}}$, $y_j:=-\frac{1}{2}+\frac{(j-1)}{L_\mathrm{y}}$ and $\theta_l:=\frac{2\pi(l-1)}{L_\mathrm{r}}$ for every $\mathbf{1}\le (i,j,l)\le\mathbf{L}=(L_\x,L_\y,L_\rt)$.

The finite Fourier coefficient (\ref{f*gkLv}) requires values of the $SE(2)$-convolution integral $f\star \rho$ on the uniform sampling grid of size $L_\x\times L_\y\times L_\rt$.  This canonically imposes several practical disadvantages/inflexibilities to the finite Fourier series numerical algorithm (\ref{SKvf}) when applied to convolutions.  To begin with, it makes the algorithms computationally expensive.  In addition,  the numerical scheme is not applicable to problems which the ultimate goal is approximating values of the $SE(2)$-convolutions. 

To improve this, we develop a unified finite Fourier series for convolution of functions $f,\rho$ which are supported in $\mathbb{D}_{1/4}$ and $\rho$ is radial in translations, called {\it convolutional finite Fourier series} denoted by $S_\mathbf{K}[f,\rho]$,  as an accurate and numerical approximation of $S_\mathbf{K}[f\star \rho]$ which can be computed independent of the values of the $SE(2)$-convolution integral itself.  The guiding idea is to approximate the finite Fourier transform $
\widehat{f\star \rho}[\mathbf{k};L]$
with the multiplication of finite Fourier transforms 
$\widehat{f}[\mathbf{k};L]\widehat{\rho}[\mathbf{k};L]$.

Initially,  we prove that multiplication of finite Fourier transform of functions can be considered as an accurate approximation for finite Fourier transform of their convolutions. 

\begin{theorem}\label{fgkL=fkLgkL}
Let $f,\rho\in\mathcal{C}^1(SE(2))$ be functions supported in $\mathbb{D}_{1/4}^\circ$ with $\rho$ radial in translations. Suppose $\mathbf{K}\in\mathbb{N}^3$ and $\mathbf{L}:=2\mathbf{K}+1$. Then 
\[
\left|\widehat{f\star \rho}[\mathbf{k};\mathbf{L}]-\widehat{f}[\mathbf{k};\mathbf{L}]\widehat{\rho}[\mathbf{k};\mathbf{L}]\right|\le \mathcal{O}\left(\frac{1}{\min(\mathbf{K})}\right)\hspace{1cm}{\rm if}\ |\mathbf{k}|\le\mathbf{\mathbf{K}}.
\]
\end{theorem}
\begin{proof}
Assume that $\mathbf{k}\in\mathbb{I}$ and $|\mathbf{k}|\le\mathbf{K}$. Using Theorem \ref{fkLmain} for $f$ and $\rho$, we achieve 
\[
\left|\widehat{f}[\mathbf{k}]-\widehat{f}[\mathbf{k};\mathbf{L}]\right|\le\mathcal{O}\left(\frac{1}{\min(\mathbf{K})}\right),
\hspace{1cm}
\left|\widehat{\rho}[\mathbf{k}]-\widehat{\rho}[\mathbf{k};\mathbf{L}]\right|\le\mathcal{O}\left(\frac{1}{\min(\mathbf{K})}\right),
\]
implying that  
\begin{equation}\label{fkgk-fkLgkL}
\left|\widehat{f}[\mathbf{k}]\widehat{\rho}[\mathbf{k}]-\widehat{f}[\mathbf{k};\mathbf{L}]\widehat{\rho}[\mathbf{k};\mathbf{L}]\right|\le\mathcal{O}\left(\frac{1}{\min(\mathbf{K})}\right).
\end{equation}
In addition, applying Theorem \ref{fkLmain} for $f\star\rho$, we get 
\begin{equation}\label{fgk-fgkL}
\left|\widehat{f\star \rho}[\mathbf{k}]-\widehat{f\star \rho}[\mathbf{k};\mathbf{L}]\right|\le\mathcal{O}\left(\frac{1}{\min(\mathbf{K})}\right).
\end{equation}
Then applying Theorem \ref{f*gP}(2), (\ref{fkgk-fkLgkL}) and (\ref{fgk-fgkL}), we achieve 
\begin{align*}
\left|\widehat{f\star \rho}[\mathbf{k};\mathbf{L}]-\widehat{f}[\mathbf{k};\mathbf{L}]\widehat{\rho}[\mathbf{k};\mathbf{L}]\right|
&\le\left|\widehat{f\star \rho}[\mathbf{k};\mathbf{L}]-\widehat{f\star \rho}[\mathbf{k}]\right|+\left|\widehat{f\star \rho}[\mathbf{k}]-\widehat{f}[\mathbf{k};\mathbf{L}]\widehat{\rho}[\mathbf{k};\mathbf{L}]\right|
\\&\le \mathcal{O}\left(\frac{1}{\min(\mathbf{K})}\right)+\left|\widehat{f}[\mathbf{k}]\widehat{\rho}[\mathbf{k}]-\widehat{f}[\mathbf{k};\mathbf{L}]\widehat{\rho}[\mathbf{k};\mathbf{L}]\right|\le 2\mathcal{O}\left(\frac{1}{\min(\mathbf{K})}\right).\qedhere
\end{align*}
\end{proof}

\begin{corollary}
{\it Let $f,\rho\in\mathcal{C}^1(SE(2))$ be functions supported in $\mathbb{D}_{1/4}^\circ$ with $\rho$ radial in translations. Suppose $K\in\mathbb{N}$. Then 
\[
\left|\widehat{f\star \rho}[\mathbf{k};2K+1]-\widehat{f}[\mathbf{k};2K+1]\widehat{\rho}[\mathbf{k};2K+1]\right|\le \mathcal{O}\left(\frac{1}{K}\right)\hspace{1cm}{\rm if}\ \|\mathbf{k}\|_\infty\le K.
\]
}\end{corollary}
Suppose $\mathbf{K}\in\mathbb{N}^3$ and $\mathbf{L}:=2\mathbf{K}+1$. Assume that $f,\rho:SE(2)\to\mathbb{C}$ are functions supported in $\mathbb{D}_{1/4}$ with $\rho$ radial in translations. The convolutional finite Fourier series, denoted by $S_\mathbf{K}[f,\rho]:SE(2)\to\mathbb{C}$, is the function given by 
\begin{equation}\label{SKfg}
S_\mathbf{K}[f,\rho](\omega):=\sum_{|\mathbf{k}|\le\mathbf{K}}\widehat{f}[\mathbf{k};\mathbf{L}]\widehat{\rho}[\mathbf{k};\mathbf{L}]\phi_{\mathbf{k}}(\omega),\hspace{0.5cm}{\rm for}\ \omega\in\Omega.
\end{equation}

\begin{remark}\label{SKfgnote}
Let $f,\rho\in\mathcal{C}^1(SE(2))$ be supported in $\mathbb{D}_{1/4}^\circ$ with $\rho$ radial in translations. If $K\in\mathbb{N}$ such that $S_K(f\star \rho)$ is a good approximation of $f\star \rho$ then $S_K[f,\rho]$ is a numerical approximation of $S_K(f\star \rho)$.  
\end{remark}

We finish this section by discussing structure of convolutional finite Fourier series on the right coset space $\mathbb{X}=\mathbb{L}\backslash SE(2)$. 

Next proposition discusses order of approximation for the finite Fourier coefficient $\widehat{\psi\oslash g}\{\mathbf{k};\mathbf{L}\}$.
\begin{proposition}
{\it Suppose $\psi\in\mathcal{C}(\mathbb{X})$ with $\psi_\Omega\in\mathcal{C}^1(SE(2))$ is  supported in $\mathbb{D}_{1/4}^\circ$. Let $\rho\in\mathcal{C}^1(SE(2))$ be a function supported in $\mathbb{D}_{1/4}$ and radial in translations. Assume $\mathbf{k}\in\mathbb{I}$ and $\mathbf{K}\in\mathbb{N}^3$. Then 
\[
\left|\widehat{\psi\oslash \rho}\{\mathbf{k};2\mathbf{K}+1\}-\widehat{\psi}\{\mathbf{k};2\mathbf{K}+1\}\widehat{\rho}[\mathbf{k};2\mathbf{K}+1]\right|\le\mathcal{O}\left(\frac{1}{\min(\mathbf{K})}\right),\hspace{1cm}\ {\rm if}\  |\mathbf{k}|\le\mathbf{K}.
\]
}\end{proposition}
\begin{proof}
Let $\mathbf{L}:=2\mathbf{K}+1$ and $|\mathbf{k}|\le\mathbf{K}$. Since $\widetilde{\psi_\Omega}=\psi$, we get $\psi\oslash \rho=\widetilde{\psi_\Omega\star \rho}$. So, Theorem \ref{fgkL=fkLgkL} implies 
\[
\left|\widehat{\psi\oslash\rho}\{\mathbf{k};\mathbf{L}\}-\widehat{\psi}\{\mathbf{k};\mathbf{L}\}\widehat{\rho}[\mathbf{k};\mathbf{L}]\right|=\left|\widehat{\psi_\Omega\star\rho}[\mathbf{k};\mathbf{L}]-\widehat{\psi_\Omega}[\mathbf{k};\mathbf{L}]\widehat{\rho}[\mathbf{k};\mathbf{L}]\right|\le\mathcal{O}\left(\frac{1}{\min(\mathbf{K})}\right).\qedhere
\]
\end{proof}
Let $\mathbf{K}:=(K_\x,K_\y,K_\rt)\in\mathbb{N}^3$ and $\mathbf{L}:=2\mathbf{K}+1$. Suppose $\psi:\mathbb{X}\to\mathbb{C}$ and $\rho:SE(2)\to\mathbb{C}$ are functions with $\rho$ radial in translations. The finite convolutional Fourier series of $\psi\oslash \rho$ of order $\mathbf{K}$, denoted by $S_{\mathbf{K}}[\psi,\rho]:\mathbb{X}\to\mathbb{C}$, is given by 
\begin{equation}\label{SKvpsig}
S_\mathbf{K}[\psi,\rho](\mathbb{L}\mathbf{g}):=\sum_{|\mathbf{k}|\le\mathbf{K}}\widehat{\psi}\{\mathbf{k};\mathbf{L}\}\widehat{\rho}[\mathbf{k};\mathbf{L}]\varphi_{\mathbf{k}}(\mathbb{L}\mathbf{g}),\hspace{1cm}{\rm for}\hspace{0.2cm}\mathbf{g}\in SE(2).
\end{equation}

\section{\bf Numerics of Convolutional Finite Fourier Series}
This section discusses numerical aspects of convolutional finite Fourier series.
To begin with, we develop matrix form of convolutional finite Fourier series. Then some implementations of numerical convolutions in MATLAB will be discussed.   

\subsection{Matrix forms}\label{MatSKfg}
This section presents the matrix form of the convolutional finite Fourier series 
which developed in Section \ref{CNV-FFS}. 

Let $f,\rho:SE(2)\to\mathbb{C}$ be functions supported in $\mathbb{D}_{1/4}$ such that $\rho$ is radial in translations. Assume $\mathbf{K}:=(K_\mathrm{x},K_\mathrm{y},K_\mathrm{t})\in\mathbb{N}^3$ and $\mathbf{L}:=2\mathbf{K}+1=(L_\mathrm{x},L_\mathrm{y},L_\mathrm{r})$. Let $\mathbf{F},\mathbf{P}\in\mathbb{C}^{\mathbf{L}}$ are given by sampling of $f$ (resp. $\rho$) on the grid $\mathbf{\Omega}_\mathbf{L}$.
Suppose $(x,y,\theta)\in \Omega$. Then using (\ref{FkLv}) we get  
\begin{equation}\label{fgalt}
S_\mathbf{K}[f,\rho](x,y,\theta)=\frac{1}{(L_\mathrm{x}L
_\mathrm{y}L_\mathrm{r})^2}\sum_{k_1=-K_\x}^{K_\x}\sum_{k_2=-K_\y}^{K_\y}\sum_{k_3=-K_\rt}^{K_\rt}\mathbf{H}(\tau_{L_\x}(k_1)+1,\tau_{L_\y}(k_2)+1,\tau_{L_\rt}(k_3)+1)e^{2\pi\ii (k_1x+k_2y)}e^{\ii k_3\theta},
\end{equation}
where $\mathbf{H}:=\widehat{\mathbf{F}}\odot\widehat{\mathbf{P}}$, $\widehat{\mathbf{F}}$ (resp. $\widehat{\mathbf{P}}$) 
is the DFT of $\mathbf{F}$ (resp. $\mathbf{P}$) given by (\ref{3DFT}). In details, (\ref{SKfg}) implies  
\begin{align*}
&S_\mathbf{K}[f,\rho](x,y,\theta)=\sum_{|\mathbf{k}|\le\mathbf{K}}\widehat{f}[\mathbf{k};\mathbf{L}]\widehat{\rho}[\mathbf{k};\mathbf{L}]\phi_{\mathbf{k}}(x,y,\theta)
\\&=\frac{1}{L_\mathrm{x}L
_\mathrm{y}L_\mathrm{r}}\sum_{k_1=-K_\x}^{K_\x}\sum_{k_2=-K_\y}^{K_\y}\sum_{k_3=-K_\rt}^{K_\rt}(-1)^{k_1+k_2}\widehat{\mathbf{F}}(\tau_{L_\x}(k_1)+1,\tau_{L_\y}(k_2)+1,\tau_{L_\rt}(k_3)+1)\widehat{\rho}[\mathbf{k};\mathbf{L}]e^{2\pi\ii (k_1x+k_2y)}e^{\ii k_3\theta}
\\&=\frac{1}{(L_\mathrm{x}L
_\mathrm{y}L_\mathrm{r})^2}\sum_{k_1=-K_\x}^{K_\x}\sum_{k_2=-K_\y}^{K_\y}\sum_{k_3=-K_\rt}^{K_\rt}\widehat{\mathbf{F}}\odot\widehat{\mathbf{P}}(\tau_{L_\x}(k_1)+1,\tau_{L_\y}(k_2)+1,\tau_{L_\rt}(k_3)+1)e^{2\pi\ii (k_1x+k_2y)}e^{\ii k_3\theta}.
\end{align*}
Let $\mathbf{N}:=(N_\x,N_\y,N_\rt)\in\mathbb{N}^3$ with $\mathbf{N}\ge\mathbf{L}$. 
Suppose that $x_n':=\frac{-1}{2}+\frac{(n-1)}{N_\x}$, $y_m':=\frac{-1}{2}+\frac{(m-1)}{N_\y}$, and $\theta_l'=\frac{2\pi(l-1)}{N_\rt}$ for every $1\le n\le N_\x$, $1\le m\le N_\y$, and $1\le l\le N_\rt$. Then 
\begin{equation}\label{MatSKfgonG}
S_{\mathbf{K}}[f,\rho](x_n',y_m',\theta_l')=\frac{N_\x N_\y N_\rt}{(L_\x L_\y L_\rt)^2}\mathrm{iDFT}(\mathbf{W}\odot\mathbf{Q}_{f}\odot\mathbf{Q}_\rho)(n,m,l),
\end{equation}
where iDFT is given by (\ref{3iDFT}), $\mathbf{Q}_{f},\mathbf{Q}_\rho\in\mathbb{C}^{\mathbf{N}}$ are given by (\ref{Qf}), and $\mathbf{W}\in\mathbb{C}^{\mathbf{N}}$ is given by 
\begin{equation}\label{W}
\mathbf{W}(n_\x,n_\y,n_\rt):=\left\{\begin{array}{lllllll}
(-1)^{n_\x+n_\y} & {\rm if}\ \ \ (n_\x,n_\y,n_\rt)\in\mathbb{I}_\x\times \mathbb{I}_\y\times\mathbb{I}_\rt\\
(-1)^{n_\x+n_\y} & {\rm if}\ \ \ (n_\x,n_\y,n_\rt)\in\mathbb{I}_\x\times \mathbb{I}_\y\times\mathbb{J}_\rt\\
(-1)^{n_\x+n_\y-N_\y} & {\rm if}\ \ \ (n_\x,n_\y,n_\rt)\in\mathbb{I}_\x\times \mathbb{J}_\y\times\mathbb{J}_\rt\\
(-1)^{n_\x+n_\y-N_\y} & {\rm if}\ \ \ (n_\x,n_\y,n_\rt)\in\mathbb{I}_\x\times \mathbb{J}_\y\times\mathbb{I}_\rt\\
(-1)^{n_\x+n_\y-N_\x-N_\y} & {\rm if}\ \ \ (n_\x,n_\y,n_\rt)\in\mathbb{J}_\x\times \mathbb{J}_\y\times\mathbb{J}_\rt\\
(-1)^{n_\x+n_\y-N_\x} & {\rm if}\ \ \ (n_\x,n_\y,n_\rt)\in\mathbb{J}_\x\times \mathbb{I}_\y\times\mathbb{J}_\rt\\
(-1)^{n_\x+n_\y-N_\x-N_\y} & {\rm if}\ \ \ (n_\x,n_\y,n_\rt)\in\mathbb{J}_\x\times \mathbb{J}_\y\times\mathbb{I}_\rt\\
(-1)^{n_\x+n_\y-N_\x} & {\rm if}\ \ \ (n_\x,n_\y,n_\rt)\in\mathbb{J}_\x\times \mathbb{I}_\y\times\mathbb{I}_\rt\\
0 & {\rm otherwise}
\end{array}
\right..
\end{equation}
Indeed, applying (\ref{fgalt}), Corollary \ref{LtoN1D-1} for $N_\x,N_\y$ and 
 Corollary \ref{LtoN1D+1} for $N_\rt$, we achieve 
\begin{align*}
&(L_\mathrm{x}L
_\mathrm{y}L_\mathrm{r})^2S_\mathbf{K}[f,\rho](x_n',y_m',\theta_l')
\\&=\sum_{k_1=-K_\x}^{K_\x}\sum_{k_2=-K_\y}^{K_\y}\sum_{k_3=-K_\rt}^{K_\rt}(-1)^{k_1+k_2}\mathbf{H}(\tau_{L_\x}(k_1)+1,\tau_{L_\y}(k_2)+1,\tau_{L_\rt}(k_3)+1)e^{\frac{2\pi\ii k_1(n-1)}{N_\x}}e^{\frac{2\pi\ii k_2(m-1)}{N_\y}}e^{\frac{2\pi\ii k_3(l-1)}{N_\rt}}
\\&=\sum_{n_\x=1}^{N_\x}\sum_{n_\y=1}^{N_\y}\sum_{n_\rt=1}^{N_\rt}\mathbf{W}(n_\x,n_\y,n_\rt)\mathbf{Q}_f(n_\x,n_\y,n_\rt)\mathbf{Q}_\rho(n_\x,n_\y,n_\rt)e^{\frac{2\pi\ii(n_\x-1)(n-1)}{N_\x}}e^{\frac{2\pi\ii (n_\y-1)(m-1)}{N_\y}}e^{\frac{2\pi\ii(n_\rt-1)(l-1)}{N_\rt}}
\\&=N_\x N_\y N_\rt\mathrm{iDFT}(\mathbf{W}\odot\mathbf{Q}_f\odot\mathbf{Q}_\rho)(n,m,l),
\end{align*}
implying (\ref{MatSKfgonG}).
In particular, if $\mathbf{N}=\mathbf{L}$ then (\ref{MatSKfgonG}) reduces to the following closed matrix form 
\[
S_\mathbf{K}[f,\rho]=\frac{1}{L_\x L_\y L_\rt}\mathrm{iDFT}(\mathbf{W}\odot\widehat{\mathbf{F}}\odot\widehat{\mathbf{P}}).
\]
Then we discuss Algorithm \ref{SKfgAlg} using regular sampling on $\Omega$ to approximate the finite convolutional Fourier series $S_{\mathbf{K}}[f,\rho]$ for functions $f,\rho:SE(2)\to\mathbb{C}$ supported in $\mathbb{D}_{1/2}$, with $\rho$ radial in translations.
\begin{algorithm}[H]
\caption{Computing numerical approximation of the $S_\mathbf{K}(f\star \rho)$ on $\mathbf{N}$-grid using FFT} 
\begin{algorithmic}[1]
\State{\bf input data}

Given functions $f,\rho:SE(2)\to\mathbb{C}$ supported in $\mathbb{D}_{1/2}$ with $\rho$ radial in translations,

 approximation order $\mathbf{K}\in\mathbb{N}^3$, 
 
configuration size $\mathbf{N}\in\mathbb{N}^3$ such that $2\mathbf{K}+1\le\mathbf{N}$. 

\State {\bf output result} 

Values of the convolutional finite Fourier series $S_\mathbf{K}[f,\rho]$ on configuration $\mathbf{N}$-grid \\ 
Put $\mathbf{L}:=(L_\x,L_\y,L_\rt)^T=2\mathbf{K}+1$\\
Generate the fundamental sampling grid $\mathbf{\Omega}_\mathbf{L}:=\{(x_i,y_j,\theta_l):\mathbf{1}\le(i,j,l)\le\mathbf{L}\}$ according to (\ref{OmL}).\\
Generate $\mathbf{F},\mathbf{P}\in\mathbb{C}^\mathbf{L}$ by sampling of $f,\rho$ on $\mathbf{\Omega}_\mathbf{L}$.\\
Compute the DFT arrays $\widehat{\mathbf{F}},\widehat{\mathbf{P}}\in\mathbb{C}^{\mathbf{L}}$ according to (\ref{3DFT}) using FFT\\
Generate the arrays $\mathbf{Q}_{f},\mathbf{Q}_{\rho}\in\mathbb{C}^{\mathbf{N}}$ associated to the configuration $\mathbf{N}$-grid according to (\ref{Qf}).\\
Load/Generate the weight array $\mathbf{W}\in\mathbb{C}^{\mathbf{N}}$ according to (\ref{W}).\\
Compute the Hadamard product  $\mathbf{W}\odot \mathbf{Q}_f\odot\mathbf{Q}_\rho$.
\State Compute $S_{\mathbf{K}}[f,\rho]$ according to (\ref{MatSKfgonG}) using FFT  \Comment{$S_{\mathbf{K}}[f,\rho]$ is approximation of $f\star\rho$}
\end{algorithmic} 
\label{SKfgAlg}
\end{algorithm}

\subsection{Numerical Convolutions}\label{NE.CNV}
We then continue by some numerical experiments in MATLAB for approximating $SE(2)$-convolutions with functions which are radial in translations.

This part is dedicated to illustrate  some numerical experiments related to convolutional finite Fourier series on $\Omega$, using the discussed matrix approach in Section \ref{MatSKfg}. We here implement some experiments in MATLAB for convolution of functions on $SE(2)$ which are approximately supported in $\Omega$ with functions which are radial in translations. 

\begin{example}
Let $\mathbf{H}:=\mathrm{diag}(0.03,0.01)^{-1}$, $\mathbf{U}:=\kappa^{-1}\mathbf{I}_2$, with $\kappa:=0.02$. Suppose $s:=0.01$, and $\nu:=\frac{\pi}{2}$. Let $f:=f_{\mathbf{H}}^{\nu,s}$ and $\rho:=f_{\mathbf{U}}^{\nu,s}$, where $f_{\mathbf{H}}^{\nu,s},f_{\mathbf{U}}^{\nu,s}$ are the 3D deformed Gaussians given by (\ref{fHsnu}). 
Then $f,\rho$ are approximately supported in $\mathbb{D}_{1/4}\subset \Omega$.
So, by applying  Proposition 3.1 of \cite{AGHF.GSC.JAT}, we conclude that the convolution $f\star \rho$ is approximately supported in $\mathbb{D}_{1/2}\subset\Omega$. 

The naive computational strategy to approximate the value of $(f\star \rho)(\mathbf{h})$, can be implemented by approximating the integral of $f\rho_\mathbf{h}$ on $\Omega$ using numerical integration, where  $\rho_{\bf h}(\mathbf{g}):=\rho(\mathbf{g}^{-1}\circ\mathbf{h})$ for $\mathbf{h},\mathbf{g}\in SE(2)$. However, the function $\rho$ is radial in translations. 
Because, we have 
\[
\rho(\mathbf{x},\mathbf{R}_\theta)=e^{-\kappa^{-1}\mathbf{x}^T\mathbf{I}_2\mathbf{x}}e^{-\frac{(\theta-\nu)^2}{s'}}=e^{-\frac{\|\mathbf{x}\|^2_2}{\kappa}}e^{-\frac{(\theta-\nu)^2}{s'}},
\]
for every $(\mathbf{x},\mathbf{R}_\theta)\in SE(2)$. This implies that  
\[
\rho(\mathbf{Sx},\mathbf{R}_\theta)=e^{-\frac{\|\mathbf{Sx}\|^2_2}{\kappa}}e^{-\frac{(\theta-\nu)^2}{s'}}=e^{-\frac{\|\mathbf{x}\|^2_2}{\kappa}}e^{-\frac{(\theta-\nu)^2}{s'}}=\rho(\mathbf{x},\mathbf{R}_\theta),
\]
for every $(\mathbf{x},\mathbf{R}_\theta)\in SE(2)$ and $\mathbf{S}\in SO(2)$.
So, the convolution $f\star \rho$ can be approximated using the convolutional finite Fourier  series $S_\mathbf{K}[f,\rho]$, if $\|\mathbf{K}\|_{\infty}$ is sufficiently large.

The following numerical experiment compares approximations of $f\star \rho$ using $T_{\mathbf{L}}(f\rho_{(\cdot)})$ versus $\Re(S_{\mathbf{K}}[f,\rho])$ on $\Omega$, where $\mathbf{L}:=2\mathbf{K}+1$ and $T_\mathbf{L}(\cdot)$ is the trapezoidal numerical integration on the sampling grid $\mathbf{\Omega}_\mathbf{L}$. 
\begin{figure}[H]
\centering
\advance\leftskip-2.7cm
\includegraphics[keepaspectratio=true,width=1.3\textwidth, height=\textheight]{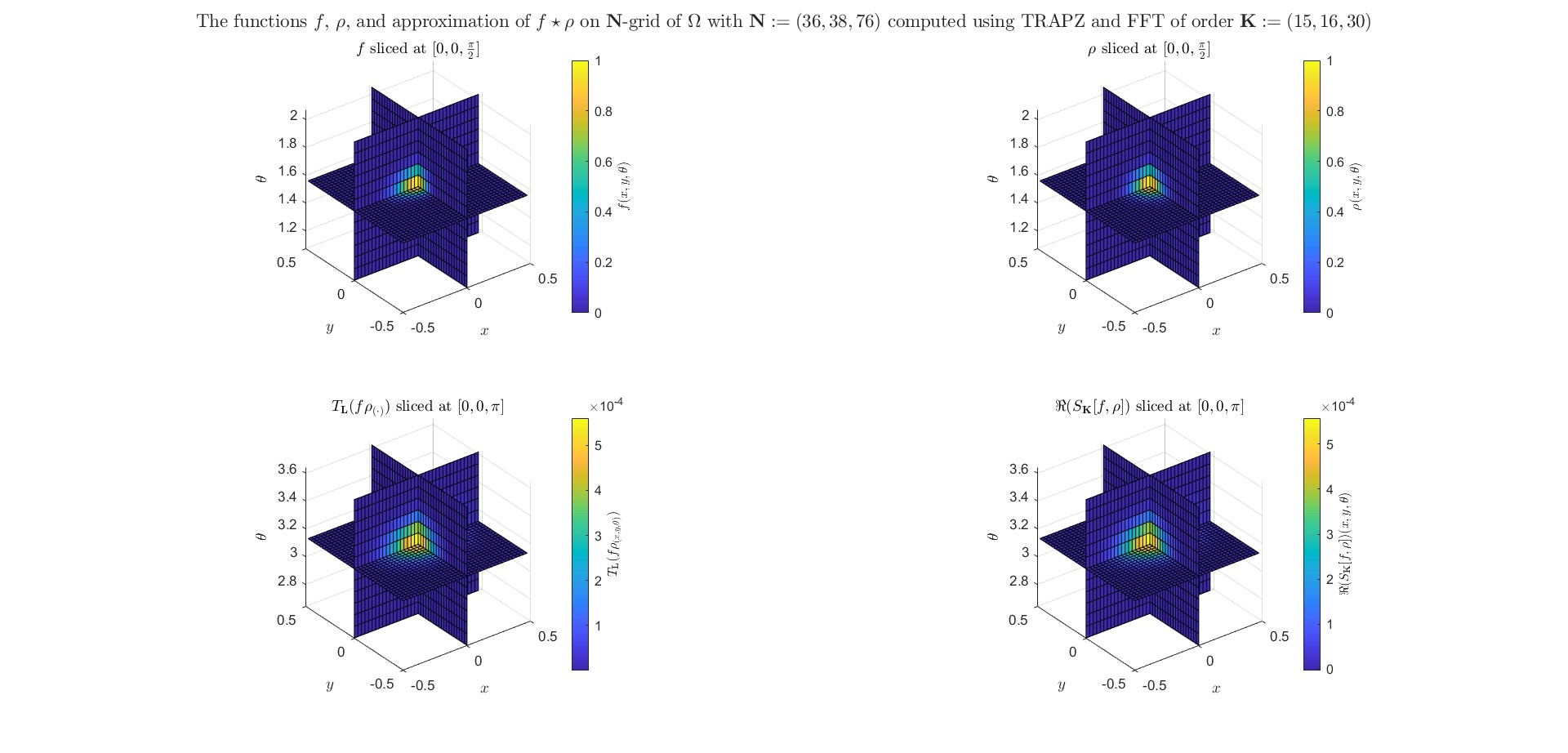}
\caption{The slice plots of $f$, $\rho$,  $T_{\mathbf{L}}(f\rho_{(\cdot)})$, and $\Re(S_{\mathbf{K}}[f,\rho])$ on the uniform grid of size $36\times 38\times 76$ of $\Omega$, with $\mathbf{K}:=(15,16,30)$ and $\mathbf{L}:=2\mathbf{K}+1$.  The functions  $f$, $\rho$ are sliced at $[0,0,\frac{\pi}{2}]$ and the functions $T_{\mathbf{L}}(f\rho_{(\cdot)})$, $\Re(S_{\mathbf{K}}[f,\rho])$ are sliced at $[0,0,\pi]$.}
\label{fig:SliccnvSE2Gdiag}
\end{figure}
\begin{figure}[H]
\centering
\advance\leftskip-2.7cm
\includegraphics[keepaspectratio=true,width=1.3\textwidth, height=\textheight]{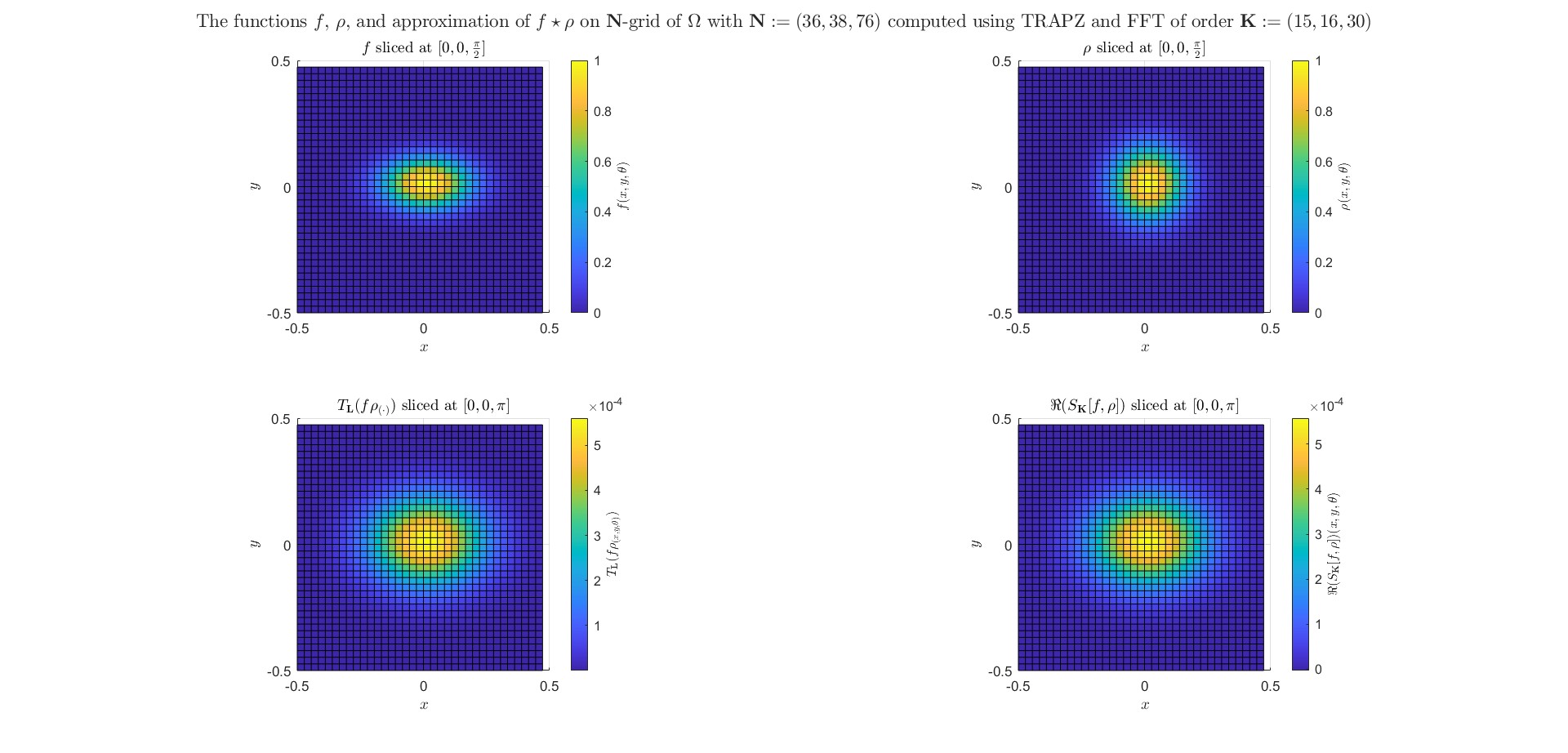}
\caption{The contour plots of $f$, $\rho$, $T_{\mathbf{L}}(f\rho_{(\cdot)})$, and $\Re(S_{\mathbf{K}}[f,\rho])$ on the uniform grid of size $36\times 38\times 76$ of $\Omega$, with $\mathbf{K}:=(15,16,30)$. The functions  $f$, $\rho$ are sliced $[0,0,\frac{\pi}{2}]$ and the functions $T_{\mathbf{L}}(f\rho_{(\cdot)})$, $\Re(S_{\mathbf{K}}[f,\rho])$ are sliced at $[0,0,\pi]$.}
\label{fig:ContcnvSE2Gdiag}
\end{figure}
The actual runtime for computing  $T_{\mathbf{L}}(f\rho_{(\cdot)})$ on the uniform grid of size $36\times 38\times 76$ of $\Omega$ was $\SI{6.8638e+03
}{\second}$ and for computing $\Re(S_{\mathbf{K}}[f,\rho])$ on the uniform grid of size $36\times 38\times 76$ of $\Omega$ was $\SI{0.0849}{\second}$.
\end{example}


\subsection{Numerical Multiple Convolutions}
We then conclude this section by some numerical experiments in MATLAB for approximating multiple $SE(2)$-convolutions with functions which are radial in translations.

Let $\rho\in L^1(SE(2))$ be a function which is radial in translations and $q\in\mathbb{N}$. Suppose $\rho^{(q)}:SE(2)\to\mathbb{C}$ is given by
\begin{equation}\label{rq}
\rho^{(q+1)}:=\rho^{(q)}\star \rho,
\end{equation}
where $\star$ is the $SE(2)$-convolution, and $\rho^{(1)}:=\rho$.

In some applications \cite{PI5}, approximation of convolutions of the following form appears
\begin{equation}\label{Vq}
V_q(f,\rho):=f\star \rho^{(q)},
\end{equation}
where $q\in\mathbb{N}$ is given, $f$ is arbitrary, and $\rho$ is the multidimensional Gaussian on $SE(2)$ given by 
 \begin{equation}\label{gbS0}
\rho(\mathbf{x},\mathbf{R}):=\exp\left(\frac{-\|{\log(\mathbf{x},\mathbf{R})^\lor}\|_2^2}{2\sigma^2}\right),
\end{equation}
for every $(\mathbf{x},\mathbf{R})\in SE(2)$ and $\sigma\in\mathbb{R}$.

Assume that $f,\rho:SE(2)\to\mathbb{C}$ are (approximately) supported in $\mathbb{D}_{1/2(q+1)}$. Then $f\star \rho^{(q)}$ is supported in $\mathbb{D}_{1/2}\subset\Omega$. Suppose the goal is computing efficient approximation of $f\star \rho^{(q)}$ for a given $q\in\mathbb{N}$. Since $\rho$ is radial in translations, $\rho^{(q)}$ is radial in translations. Applying Remark \ref{SKfgnote}, the convolutional finite Fourier series $S_\mathbf{K}[f,\rho^{(q)}]$ given by
\begin{equation}\label{SKfgq}
S_\mathbf{K}[f,\rho^{(q)}](\omega)=\sum_{|\mathbf{k}|\le\mathbf{K}}\widehat{f}[\mathbf{k};2\mathbf{K}+1]\widehat{\rho^{(q)}}[\mathbf{k};2\mathbf{K}+1]\phi_{\mathbf{k}}(\omega),\hspace{0.5cm}{\rm for}\ \omega\in\Omega,
\end{equation}
is an efficient numerical approximation of $f\star \rho^{(q)}$, if $\|\mathbf{K}\|_\infty$ is sufficiency large. 

However, the later requires values of $\rho^{(q)}$ on the sampling grid  which makes $S_\mathbf{K}[f,\rho^{(q)}]$ computationally expensive, if $q>1$. 
Since $\rho$ is radial in translations, Theorem \ref{f*gP}(2) implies  
\[
\widehat{\rho^{(q)}}[\mathbf{k}]=\widehat{\rho}[\mathbf{k}]^q,
\]
for every $\mathbf{k}\in\mathbb{I}$.  So, by Theorem \ref{fgkL=fkLgkL}, $\widehat{\rho^{(q)}}[\mathbf{k};2\mathbf{K}+1]$ can be approximated by  $\widehat{\rho}[\mathbf{k};2\mathbf{K}+1]^q$. 
Therefore, $S_\mathbf{K}^q[f,\rho]$ defined by 
\begin{equation}\label{SqKfg}
S_\mathbf{K}^q[f,\rho](\omega):=\sum_{|\mathbf{k}|\le\mathbf{K}}\widehat{f}[\mathbf{k};2\mathbf{K}+1]\widehat{\rho}[\mathbf{k};2\mathbf{K}+1]^q\phi_\mathbf{k}(\omega),\hspace{1cm}{\rm for}\hspace{0.2cm}\omega\in\Omega,
\end{equation}
can be considered as an accurate approximation of the convolution $f\star \rho^{(q)}$ if $\|\mathbf{K}\|_\infty$ is sufficiently large. 
\subsubsection{\bf Matrix forms}
Next, we develop a matrix form for $S_\mathbf{K}^q[f,\rho]$ given by (\ref{SqKfg}). Suppose $q\in\mathbb{N}$ and $f,\rho\in L^1(SE(2))$ are supported in $\mathbb{D}_{\frac{1}{2(q+1)}}$. Let $\mathbf{k}:=(k_1,k_2,k_3)^T\in\mathbb{I}$, $\mathbf{K}\in\mathbb{N}^3$ and $\mathbf{L}:=2\mathbf{K}+1$.
Then applying (\ref{FkLv}), we get 
\begin{equation}
\widehat{\rho}[\mathbf{k};\mathbf{L}]^q=\frac{((-1)^{k_1+k_2})^q}{(L_\mathrm{x}L_\mathrm{y}L_\mathrm{r})^q}\widehat{\mathbf{P}}(\tau_{L_\x}(k_1)+1,\tau_{L_\y}(k_2)+1,\tau_{L_\rt}(k_3)+1)^q.
\end{equation}
Let $\mathbf{N}:=(N_\x,N_\y,N_\rt)\in\mathbb{N}^3$ with $\mathbf{N}\ge\mathbf{L}$. 
Suppose $x_n':=\frac{-1}{2}+\frac{(n-1)}{N_\x}$, $y_m':=\frac{-1}{2}+\frac{(m-1)}{N_\y}$, and $\theta_l'=\frac{2\pi(l-1)}{N_\rt}$ for every $1\le n\le N_\x$, $1\le m\le N_\y$, and $1\le l\le N_\rt$. Then 
\begin{equation}\label{MatSqKfgG}
S_{\mathbf{K}}^q[f,\rho](x_n',y_m',\theta_l')=\frac{N_\x N_\y N_\rt}{(L_\x L_\y L_\rt)^{q+1}}\mathrm{iDFT}(\mathbf{W}^{\odot q}\odot\mathbf{Q}_{f}\odot\mathbf{Q}_\rho^{\odot q})(n,m,l),
\end{equation}
where iDFT is given by (\ref{3iDFT}), $\mathbf{W}\in\mathbb{C}^{\mathbf{N}}$ is given by (\ref{W}), $\mathbf{Q}_{f},\mathbf{Q}_\rho\in\mathbb{C}^{\mathbf{N}}$ are given by (\ref{Qf}), and $^{\odot q}$ is $q$-th order Hadamard power.

The Algorithm \ref{SKfgpAlg} discusses finite multiple convolutional Fourier series $S_{\mathbf{K}}^q[f,\rho]$ as an approximation of $f\star \rho^{(q)}$, where functions $f,\rho:SE(2)\to\mathbb{C}$ are supported in $\mathbb{D}_{1/2(q+1)}$ and $\rho$ is radial in translations. 
\begin{algorithm}[H]
\caption{Computing numerical approximation of the $S_\mathbf{K}(f\star \rho^{(q)})$ on $\mathbf{N}$-grid using FFT} 
\begin{algorithmic}[1]
\State{\bf input data} 

Given $q\in\mathbb{N}$, 

functions $f,\rho:SE(2)\to\mathbb{C}$ (approximately) supported in $\mathbb{D}_{1/2(q+1)}$ with $\rho$ radial in translations,

 approximation order $\mathbf{K}\in\mathbb{N}^3$, 
 
 configuration size $\mathbf{N}\in\mathbb{N}^3$.
\State {\bf output result} 

Approximated values of $S_\mathbf{K}^q[f,\rho]$ on configuration $\mathbf{N}$-grid\\ 
Put $\mathbf{L}:=(L_\x,L_\y,L_\rt)^T=2\mathbf{K}+1$ and $(N_\x,N_\y,N_\rt):=\mathbf{N}$\\
Generate the fundamental sampling grid $\mathbf{\Omega}_\mathbf{L}:=\{(x_i,y_j,\theta_l):\mathbf{1}\le (i,j,l)\le\mathbf{L}\}$\\
Generate $\mathbf{F},\mathbf{P}\in\mathbb{C}^\mathbf{L}$ using sampling of $f$ (resp. $\rho$) on $\mathbf{\Omega}_\mathbf{L}$\\
Compute $\widehat{\mathbf{F}},\widehat{\mathbf{P}}\in\mathbb{C}^{\mathbf{L}}$ according to (\ref{3DFT}) using FFT\\
Generate $\mathbf{Q}_{f},\mathbf{Q}_{\rho}\in\mathbb{C}^{\mathbf{N}}$ associated to the configuration $\mathbf{N}$-grid according to (\ref{Qf}).\\
Compute the  $q$-th order Hadamard power $\mathbf{Q}_\rho^{[q]}:=\mathbf{Q}_\rho^{\odot q}$.\\
Compute the Hadamard product  $\mathbf{H}_{f,\rho}^{[q]}:=\mathbf{Q}_f\odot\mathbf{Q}_\rho^{[q]}$.
\If{$q$ is odd}
\State Load/Generate $\mathbf{W}\in\mathbb{C}^{\mathbf{N}}$ and compute the Hadamard product  $\mathbf{W}\odot \mathbf{H}_{f,\rho}^{[q]}$.
\State Compute $S_{\mathbf{K}}^{q}[f,\rho]:=\frac{N_\x N_\y N_\rt}{(L_\x L_\y L_\rt)^{q+1}}\mathrm{iDFT}\left(\mathbf{W}\odot \mathbf{H}_{f,\rho}^{[q]}\right)$ using FFT 
\Else
\State Compute $S_{\mathbf{K}}^{q}[f,\rho]:=\frac{N_\x N_\y N_\rt}{(L_\x L_\y L_\rt)^{q+1}}\mathrm{iDFT}\left(\mathbf{H}_{f,\rho}^{[q]}\right)$ using FFT  
\EndIf \Comment{$S_{\mathbf{K}}^q[f,\rho]$ is approximation of $V_q(f,\rho)$}
\end{algorithmic} 
\label{SKfgpAlg}
\end{algorithm}
\newpage
\begin{example}
Let $\mathbf{H}:=\mathrm{diag}(0.025,0.0025)^{-1}$, $s:=0.0125$, and $\nu:=\frac{\pi}{4}$.
Suppose that $f=f_{\mathbf{H}}^{\nu,s}$ is the associated 3D deformed Gaussian given by (\ref{fHsnu}).  Assume that $\sigma^2:=0.005$ and $\rho$ be the associated multidimensional Gaussian given by (\ref{gbS0}) which is $\frac{\pi}{4}$ shifted in rotations direction. 
Then $f,\rho$ are approximately supported in $\mathbb{D}_{\frac{1}{8}}$. So, the $SE(2)$-convolution $f\star \rho^{(3)}$ is approximately supported in $\mathbb{D}_{\frac{1}{2}}$.
Therefore, the $SE(2)$-convolution $f\star \rho^{(3)}$ can be approximated using the convolutional finite Fourier  series $S_\mathbf{K}^3[f,\rho]$.
\begin{figure}[H]
\centering
\includegraphics[keepaspectratio=true,width=\textwidth, height=\textheight]{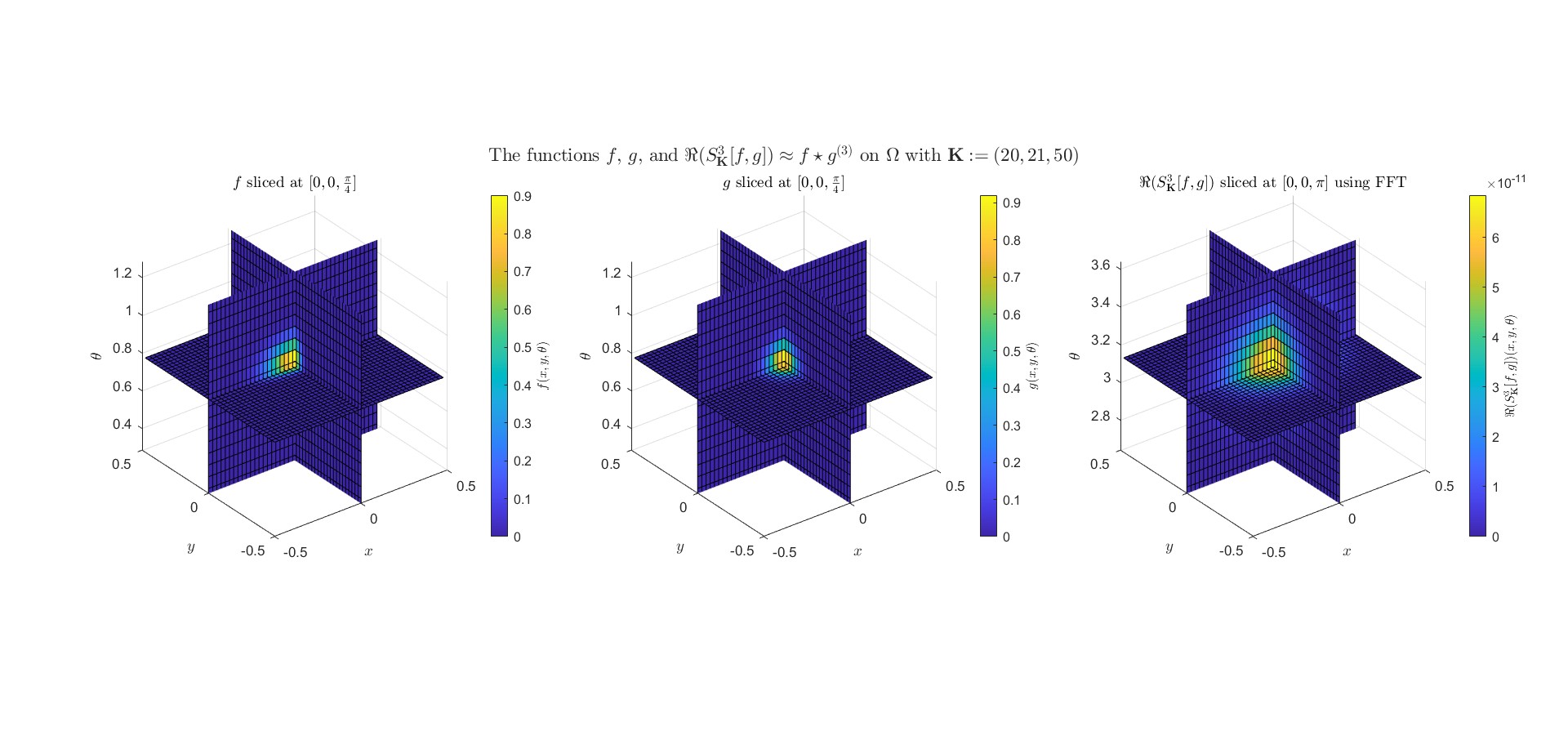}
\caption{The slice plot of $f:=f_{\mathbf{H}}^{\nu,s}$, $g=\rho$ and $\Re(S_{\mathbf{K}}^3[f,\rho])$ on the uniform grid of size $41\times 43\times 151$ of the fundamental domain $\Omega$, with $\mathbf{K}:=(20,21,50)$. The functions  $f$, $\rho$ are sliced $[0,0,\frac{\pi}{4}]$ and  $\Re(S_{\mathbf{K}}^3[f,\rho])$ sliced at $[0,0,\pi]$.}
\label{fig:SlicCNVfg3}
\end{figure}
\begin{figure}[H]
\centering
\includegraphics[keepaspectratio=true,width=\textwidth, height=\textheight]{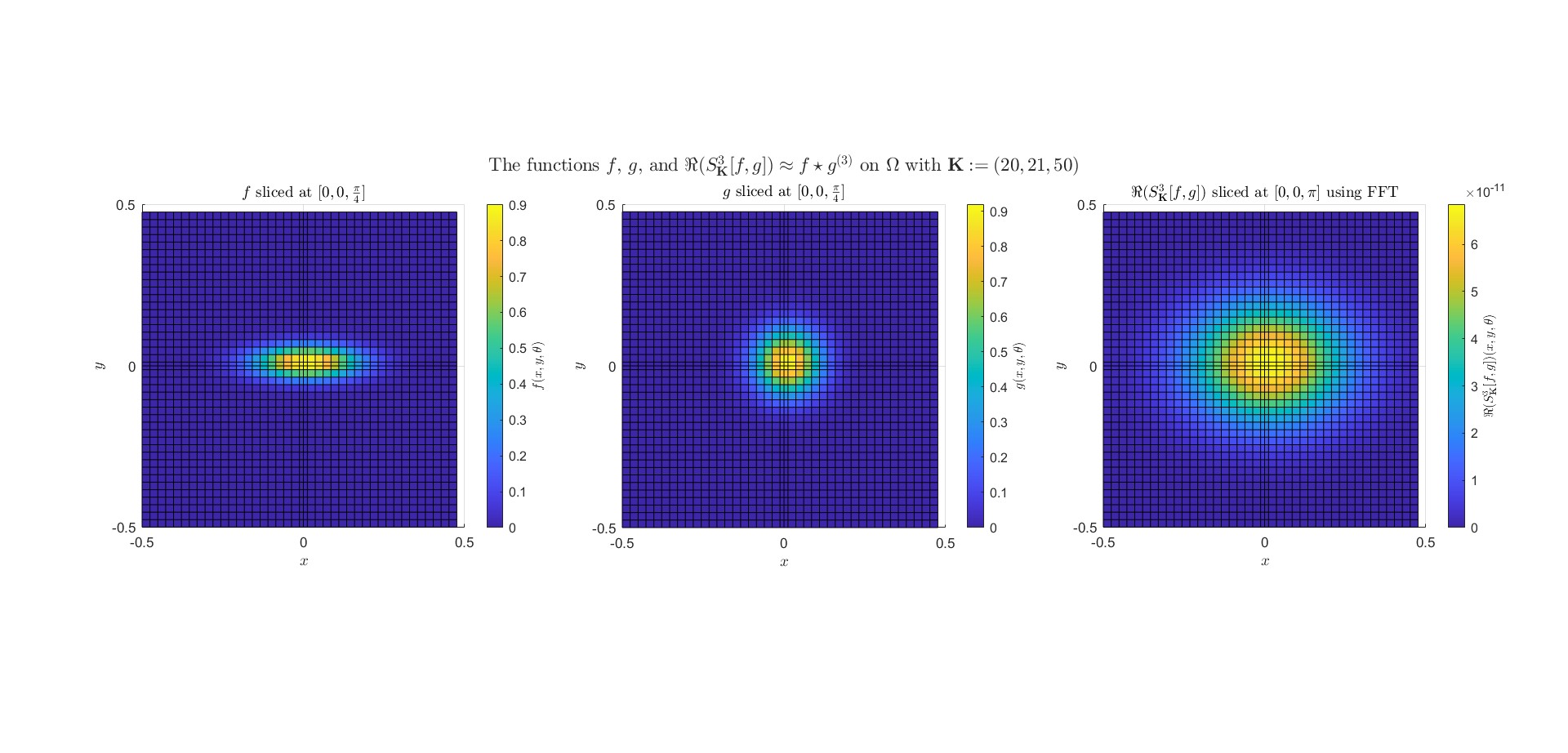}
\caption{The contour plot of $f$, $g=\rho$ and $\Re(S_{\mathbf{K}}^3[f,\rho])$ on the uniform grid of size $41\times 43\times 151$ of the fundamental domain $\Omega$, with $\mathbf{K}:=(20,21,50)$. The functions  $f$, $\rho$ are sliced $[0,0,\frac{\pi}{4}]$ and  $\Re(S_{\mathbf{K}}^3[f,\rho])$ sliced at $[0,0,\pi]$.}
\label{fig:ContCNVfg3}
\end{figure}

\end{example}

\newpage
In some practical applications in Robotics, if $q\in\mathbb{N}$ is given, the ultimate goal is not just computing efficient approximation of the convolution $f\star \rho^{(q)}$, but also  computing approximated values of the sequence of all convolutions $f\star \rho^{(p)}$ at each time $1\le p\le q$. 

The straightforward solution to this problem is computing each $SE(2)$-convolution $f\star \rho^{(p)}$ at each $1\le p\le q$, using the following algorithm.

\begin{algorithm}[H]
\caption{Naive computation of all approximations of $f\star \rho^{(p)}$ ($1\le p\le q$)} 
\begin{algorithmic}[1]
\State{\bf input data} Given $q\in\mathbb{N}$, $f,\rho:SE(2)\to\mathbb{C}$ supported in $\frac{1}{2(q+1)}$, with $\rho$ radial in translations, and approximation order $\mathbf{K}\in\mathbb{N}^3$.
\State {\bf output result} Approximation of $f\star \rho^{(p)}$ for every $1\le p\le q$.
\For{$p=1:q$}
\State Compute $V_p:=S_\mathbf{K}^p[f,\rho]$ using Algorithm \ref{SKfgpAlg}.
\State Return $V_p$. \Comment{$V_p$ is the approximation of $f\star \rho^{(p)}$.}
\EndFor
\end{algorithmic} 
\label{SKVqfAlg}
\end{algorithm}

The above algorithm is time consuming when $q$ is large, which is the case in practical experiments. Therefore, it is not computationally attractive for large $q$. To improve this, we can consider a recursive approach to compute approximation of all convolutions recursively to reach the final convolution $f\star \rho^{(q)}$. 

To this end, applying associativity of the  $SE(2)$-convolution, we have 
\begin{align*}
V_{p+1}(f,\rho)&=f\star \rho^{(p+1)}=f\star(\rho^{(p)}\star \rho)
\\&=(f\star \rho^{(p)})\star \rho=V_p(f,\rho)\star \rho\,,
\end{align*}
which implies that 
\begin{equation}\label{Vq1q}
V_{p+1}(f,\rho)=V_p(f,\rho)\star \rho.
\end{equation}
Suppose $V_p^\mathbf{K}[f,\rho]$ is an accurate numerical approximation of $V_p(f,\rho)$, for $\mathbf{K}\in\mathbb{N}^3$ with sufficiently large $\|\mathbf{K}\|_\infty$.
Since $\rho$ is radial in translations, using (\ref{Vq1q}), the convolutional finite Fourier  series $S_\mathbf{K}[V_p^\mathbf{K}[f,\rho],\rho]$ can be used as an efficient numerical approximation of $V_{p+1}(f,\rho)$. This approach requires approximated values of $V_{p}^\mathbf{K}[f,\rho]$ which can be constructed and computed recursively. Assume that $V_p^\mathbf{K}[f,\rho]$ is given recursively by sequence of numerical approximations 
\[
V_{p+1}^{\mathbf{K}}[f,\rho]:=S_\mathbf{K}[V_p^{\mathbf{K}}[f,\rho],\rho],
\]
with $V_1^\mathbf{K}[f,\rho]:=S_\mathbf{K}[f,\rho]$.

The Algorithm \ref{SKVqfgNaiveRecAlg} applies recursive approach to approximate all the convolutions of the form $V_p(f,\rho)$, given by (\ref{Vq}), using the convolutional finite Fourier series of convolutions for a function $f,\rho:SE(2)\to\mathbb{C}$ supported in $\mathbb{D}_{1/2(q+1)}$, with $\rho$ radial in translations.

\begin{algorithm}[H]
\caption{Naive recursive computation of all approximations of $V_p(f,g)$ ($1\le p\le q$)} 
\begin{algorithmic}[1]
\State{\bf input data} Given $q\in\mathbb{N}$, functions $f,\rho:SE(2)\to\mathbb{C}$ supported in $\frac{1}{2(q+1)}$ with $\rho$ radial in translations, and approximation order $\mathbf{K}\in\mathbb{N}^3$.
\State {\bf output result} Approximated values of $V_p(f,\rho)$ on the $\mathbf{L}$-sampling grid, for $1\le p\le q$.
\State Put $\mathbf{L}:=(L_\x,L_\y,L_\rt)=2\mathbf{K}+1$
\State Generate the $\mathbf{L}$-sampling grid $\mathbf{\Omega}_\mathbf{L}:=\{(x_i,y_j,\theta_l):\mathbf{1}\le(i,j,l)\le\mathbf{L}\}$.
\State Generate $V:=f$ on the $\mathbf{L}$-sampling grid, using values of $f$.
\For{$p=1:q$}
\State Compute $V_p:=S_\mathbf{K}[V,\rho]$ on the sampling grid via Algorithm \ref{SKfgAlg} using FFT.
\State Return $V_p$ with $1\le p\le q$. \Comment{$V_p$ is the approximated value of $V_p(f,\rho)$.}
\State Put  $V:=V_p$.
\EndFor
\end{algorithmic} 
\label{SKVqfgNaiveRecAlg}
\end{algorithm}

The following algorithm reduces actual running time of computing all approximations $V_p[f,\rho]$ by just considering the minimum required computations at each iteration steps of the loop.

\begin{algorithm}[H]
\caption{Fast recursive computation of all approximations of $V_p(f,\rho)$ ($1\le p\le q$)} 
\begin{algorithmic}[1]
\State{\bf input data} Given $q\in\mathbb{N}$, functions $f,\rho:SE(2)\to\mathbb{C}$ supported in $\frac{1}{2(q+1)}$ and $\rho$ radial in translations, and approximation order $\mathbf{K}\in\mathbb{N}^3$.
\State {\bf output result} Approximated values of $V_p(f,\rho)$ on the $\mathbf{L}$-sampling grid, for $1\le p\le q$.
\State Put $\mathbf{L}:=(L_\x,L_\y,L_\rt)=2\mathbf{K}+1$.
\State Load/Generate the weight array $\mathbf{W}\in\mathbb{C}^{\mathbf{L}}$.
\State Generate the $\mathbf{L}$-sampling grid $\mathbf{\Omega}_\mathbf{L}:=\{(x_i,y_j,\theta_l):\mathbf{1}\le(i,j,l)\le\mathbf{L}\}.$
\State Generate $\mathbf{F},\mathbf{P}\in\mathbb{C}^\mathbf{L}$ using $\mathbf{F}(i,j,l):=f(x_i,y_j,\theta_l)$ (resp. $\mathbf{P}(i,j,l):=\rho(x_i,y_j,\theta_l)$).
\State Compute $\widehat{\mathbf{F}},\widehat{\mathbf{P}}\in\mathbb{C}^\mathbf{L}$ using FFT.
\State Put $\mathbf{V}:=\widehat{\mathbf{F}}$.
\For{$p=1:q$}
\State Compute the Hadamard product $\mathbf{H}:=\mathbf{W}\odot\mathbf{V}\odot\widehat{\mathbf{P}}$.
\State Compute $\mathbf{V}_p\in\mathbb{C}^\mathbf{L}$ using FFT via   
$$\mathbf{V}_p:=\frac{1}{(L_\x L_\y L_\rt)^p}\mathrm{iDFT}(\mathbf{H}).$$
\State \Return $\mathbf{V}_p$. \Comment{$\mathbf{V}_p$ is the approximation of $V_p(f,\rho)$.}
\State Put  $\mathbf{V}:=\mathbf{H}$.
\EndFor
\end{algorithmic} 
\label{SKVqfgAlgFastLoop}
\end{algorithm}
\begin{example}
Let $\mathbf{H}:=\mathrm{diag}(0.0143,0.0014)^{-1}$, $s:=0.0071$, and $\nu:=\pi/7$. Suppose $f=f_{\mathbf{H}}^{\nu,s}$ is the associated 3D deformed Gaussian given by (\ref{fHsnu}).  Assume $\sigma^2:=0.0029$ and $\rho$ is the associated multidimensional Gaussian given by (\ref{gbS0}) which is $\pi/7$ shifted in rotations. 
Then $f,\rho$ are approximately supported in $\mathbb{D}_{1/14}$. So, $SE(2)$-convolutions $f\star \rho^{(p)}$ ($1\le p\le 6$) are approximately supported in $\mathbb{D}_{1/2}$. Therefore, $SE(2)$-convolutions $f\star \rho^{(p)}$ ($1\le p\le 6$) can be approximated using Algorithm \ref{SKVqfgAlgFastLoop}.
\begin{figure}[H]
\centering
\advance\leftskip-2cm
\includegraphics[keepaspectratio=true,width=1.2\textwidth, height=\textheight]{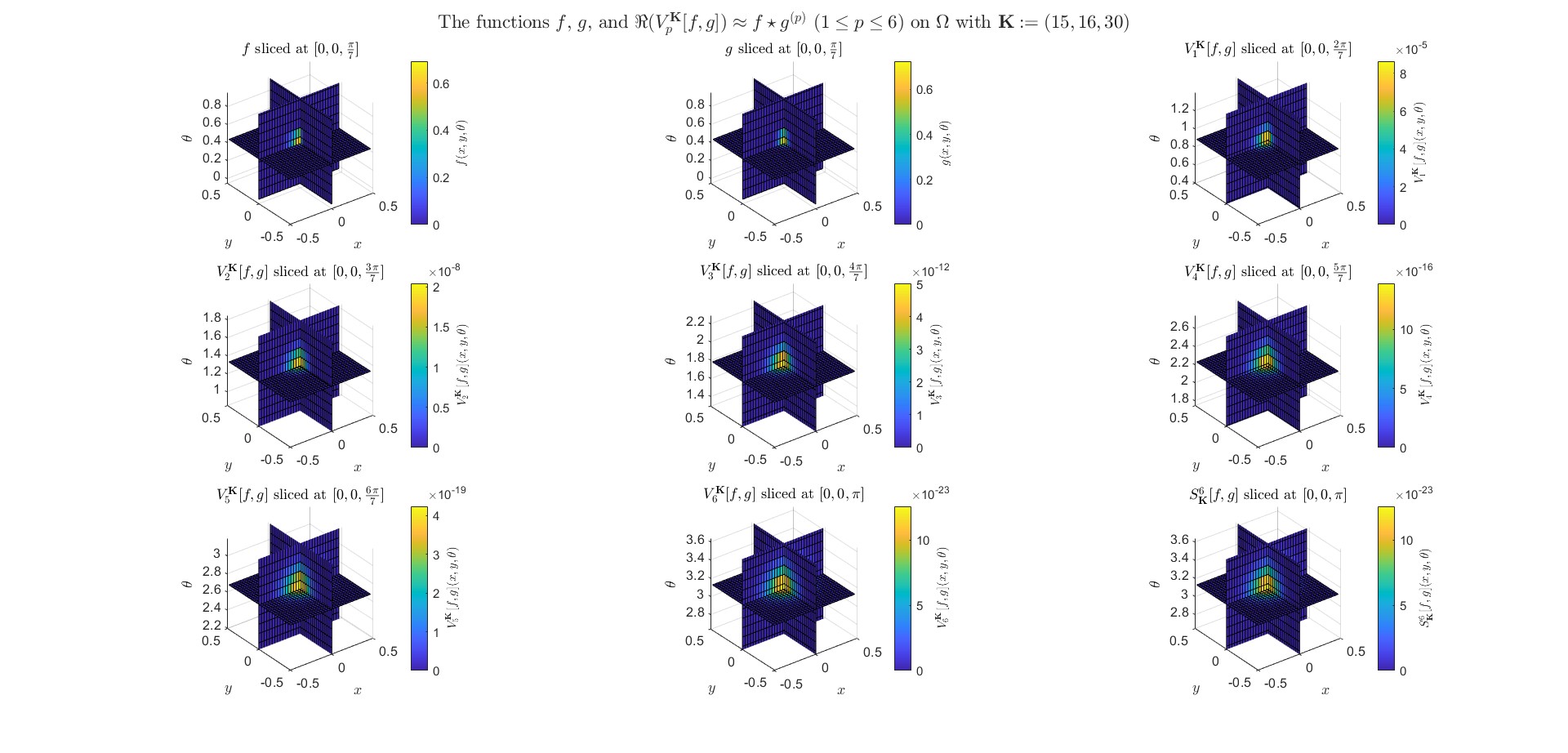}
\caption{The slice plot of $f$, $g=\rho$ and $\Re(V^{\mathbf{K}}_p[f,\rho])\approx f\star \rho^{(p)}$ on the uniform grid of size $31\times 33\times 61$ of the fundamental domain $\Omega$, with $\mathbf{K}:=(15,16,30)$. The functions  $f$, $\rho$ are sliced $[0,0,\frac{\pi}{7}]$ and  $\Re(V^{\mathbf{K}}_p[f,\rho])$ sliced at $[0,0,\frac{(p+1)\pi}{7}]$ for every $1\le p\le 6$.}
\label{fig:SlicCNVfg6}
\end{figure}
\begin{figure}[H]
\centering
\includegraphics[keepaspectratio=true,width=1.35\textwidth, height=1.5\textheight, angle=90,origin=c]{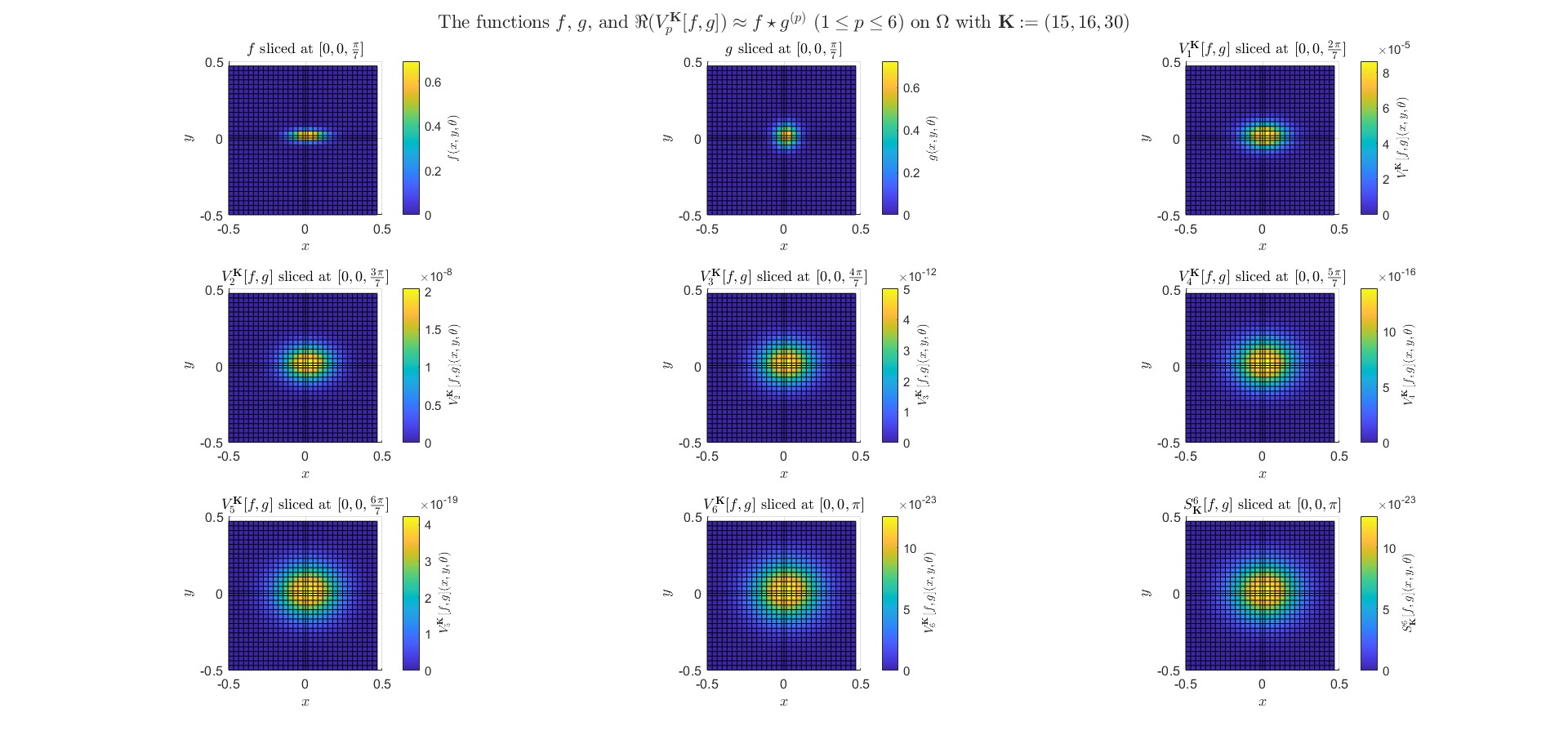}
\caption{The contour plot of $f$, $g=\rho$ and $\Re(V^{\mathbf{K}}_p[f,\rho])\approx f\star \rho^{(p)} (1\le p\le 6)$ on the uniform grid of size $31\times 33\times 61$ of $\Omega$, with $\mathbf{K}:=(15,16,30)$. The functions  $f$, $\rho$ are sliced $[0,0,\frac{\pi}{7}]$ and  $\Re(V^{\mathbf{K}}_p[f,\rho])$ sliced at $[0,0,\frac{(p+1)\pi}{7}]$ for every $1\le p\le 6$.}
\label{fig:ContCNVfg6}
\end{figure}
\end{example}

\newpage
\section{\bf Annex}
\subsection{Finite Fourier coefficients}
This part discusses finite Fourier coefficients as numerical approximation of Fourier coefficients on boxes. Let $a,b\in\mathbb{R}$ and $a<b$. For $k\in\mathbb{Z}$, let $e_k:[a,b]\to\mathbb{C}$ be given by 
\[
e_k(x):=\exp\left(\frac{2\pi\ii kx}{b-a}\right)=e^{\frac{2\pi\ii kx}{b-a}},\hspace{1cm}\ {\rm for}\ x\in[a,b].
\]
Then $(e_k)_{k\in\mathbb{Z}}$ is an orthonormal basis (ONB) for the Hilbert space $L^2[a,b]$, with the inner-product
\[
\langle U,V\rangle:=\frac{1}{b-a}\int_a^bU(x)\overline{V(x)}\D x,\hspace{1cm}\ {\rm for}\ U,V\in L^2[a,b]. 
\]
So, if $U\in L^2[a,b]$ then $U=\sum_{k\in\mathbb{Z}}\widehat{U}[k]e_k$ in $L^2$-sense, with the Fourier coefficient $\widehat{U}[k]$  given by 
\begin{equation}\label{uk}
\widehat{U}[k]:=\langle U,e_k\rangle=\frac{1}{b-a}\int_a^bU(x)\overline{e_k(x)}\D x,
\hspace{0.7cm}\ {\rm for}\ k\in\mathbb{Z}.
\end{equation}
Let $L\in\mathbb{N}$, and $x_i:=a+(i-1)\delta$ with $\delta:=\frac{b-a}{L}$.  
The finite Fourier coefficient of $U\in L^2[a,b]$ at $k\in\mathbb{Z}$ on the uniform grid of size $L$ is given by 
\begin{equation}\label{FFC}
\widehat{U}[k;L]:=\frac{1}{L}\sum_{i=1}^LU(x_i)\overline{e_k(x_i)}.
\end{equation}

Next we show that finite Fourier coefficient (\ref{FFC}) gives exact value of Fourier  coefficients for trigonometric polynomials. 
\begin{lemma}\label{TrigL}
Let $K\in\mathbb{N}$, and $L=2K+1$. Suppose $P$ is a trigonometric polynomial of degree $K$.  Then  
\begin{equation}\label{pHpHK}
\widehat{P}[k]=\widehat{P}[k;L],\hspace{1cm}{\rm for}\ \ |k|\le K. 
\end{equation}
\end{lemma}
\begin{proof}
Let $k\in\mathbb{Z}$ with $0\le |k|\le K$ be given. Then $P(x)=\sum_{l=-K}^K\widehat{P}[l]e_l(x)$.
In addition, for every $1\le i\le L$, we have $e_l(x_i)=\kappa^l e^{2\pi\ii l(i-1)/L}$, 
where $\kappa:=e^{\frac{2\pi\ii a}{b-a}}$. So, for $1\le i\le L$, we get 
\[
P(x_i)=\sum_{l=-K}^K\widehat{P}[l]e_l(x_i)=\sum_{\ell=-K}^K\widehat{P}[l]\kappa^l e^{2\pi\ii l(i-1)/L}.
\]
Assume $-K\le l\le K$.  Then $|l-k|\le 2K<L$.  Hence, 
$\sum_{i=1}^Le^{2\pi\ii(l-k)(i-1)/L}=L\delta_{l,k}$, implying that 
\begin{align*}
\widehat{P}[k;L]&=\frac{1}{L}\sum_{i=1}^LP(x_i)\kappa^{-k}e^{-2\pi\ii k(i-1)/L}
\\&=\frac{1}{L}\sum_{i=1}^L\left(\sum_{l=-K}^K\widehat{P}[l]\kappa^l e^{2\pi\ii l(i-1)/L}\right)\kappa^{-k}e^{-2\pi\ii k(i-1)/L}
\\&=\sum_{l=-K}^K\widehat{P}[l]\kappa^{l-k} \left(\frac{1}{L}\sum_{i=1}^Le^{2\pi\ii(l-k)(i-1)/L}\right)
=\sum_{l=-K}^K\widehat{P}[l]\kappa^{l-k} \delta_{l,k}=\widehat{P}[k].\qedhere
\end{align*}
\end{proof}

We then conclude the following bound of the absolute error for approximation of $\widehat{U}[k]$ using $\widehat{U}[k;L]$.
 \begin{proposition}\label{mainP}
{\it Let $a<b$, $U\in\mathcal{C}^1[a,b]$ with $U(a)=U(b)$, and $K\in\mathbb{N}$. Then  
\[
\Big|\widehat{U}[k]-\widehat{U}[k;2K+1]\Big|\le\frac{6(b-a)\|U'\|_\infty}{\pi K},\hspace{1cm}{\rm for}\ \ |k|\le K. 
\]
}\end{proposition}
\begin{proof}
Let $\alpha:[-\pi,\pi]\to[a,b]$ be given by $\alpha(t):=\frac{b-a}{2\pi}(t+\pi)+a$ for $t\in[-\pi,\pi]$. Suppose $V:[-\pi,\pi]\to\mathbb{C}$ is defined by $V(t):=U(\alpha(t)),$
for $t\in[-\pi,\pi]$. Then $V\in\mathcal{C}^1[-\pi,\pi]$, $V(-\pi)=V(\pi)$, and $\|V'\|_\infty=\frac{b-a}{2\pi}\|U'\|_\infty$. By Theorem 1.3 of \cite{Jack}, there exists a trigonometric polynomial $T$ of degree $K$ on $[-\pi,\pi]$ such that $\|T-V\|_\infty\le\frac{6\|V'\|_\infty}{K}$. Assume $P(x):=T(\alpha^{-1}(x))$ for $x\in [a,b]$. Then $P$ is a trigonometric polynomial of degree $K$ on $[a,b]$. So, we get 
\begin{equation}\label{JackAlt}
\|U-P\|_\infty=\|V-T\|_\infty\le\frac{6\|V'\|_\infty}{K}=\frac{3(b-a)\|U'\|_\infty}{\pi K}.
\end{equation}
Suppose $|k|\le K$. Using Lemma \ref{TrigL}, $\widehat{P}[k]=\widehat{P}[k;2K+1]$. So,  (\ref{JackAlt}) implies  
\begin{align*}
\left|\widehat{U}[k]-\widehat{U}[k;2K+1]\right|
&\le|\widehat{U}[k]-\widehat{P}[k]|+\left|\widehat{P}[k;2K+1]-\widehat{U}[k;2K+1]\right|
\\&\le\|U-P\|_\infty+\left|\frac{1}{2K+1}\sum_{i=1}^{2K+1}P(x_i)e_k(x_i)-\frac{1}{2K+1}\sum_{i=1}^{2K+1}U(x_i)e_k(x_i)\right|
\\&\le\|U-P\|_\infty+\frac{1}{2K+1}\sum_{i=1}^{2K+1}|P(x_i)-U(x_i)|
\\&\le\|U-P\|_\infty+\frac{1}{2K+1}\sum_{i=1}^{2K+1}\|P-U\|_\infty\le 2\|U-P\|_{\infty}\le\frac{6(b-a)\|U'\|_\infty}{\pi K}.\qedhere
\end{align*}
\end{proof}

\begin{remark}
The notion of finite Fourier coefficient appeared in a similar setting as finite Fourier transforms in \cite{FFC01} for the case of the interval $[0,1]$.
\end{remark}

Let $d>1$ and $\mathbf{a},\mathbf{b}\in\mathbb{R}^d$ with $\mathbf{a}<\mathbf{b}$. Assume $[\mathbf{a},\mathbf{b}]:=\prod_{\ell=1}^d[a_\ell,b_\ell]$  is the $d$-dimensional box, where $\mathbf{a}:=(a_1,\ldots,a_d)^T$ and $\mathbf{b}:=(b_1,\ldots,b_d)^T$. For $\mathbf{k}:=(k_1,\ldots,k_d)^T\in\mathbb{Z}^d$, let $\mathbf{e}_{\mathbf{k}}:[\mathbf{a},\mathbf{b}]\to\mathbb{C}$ be given by 
\vspace{-0.2cm}
\begin{equation}\label{ekxd}
\mathbf{e}_\mathbf{k}(\mathbf{x}):=\prod_{\ell=1}^de_{k_\ell}(x_\ell)=\prod_{\ell=1}^d\exp\left(\frac{2\pi\ii k_\ell x_\ell}{b_\ell-a_\ell}\right),\hspace{0.7cm}{\rm for}\ \mathbf{x}:=(x_1,\ldots,x_d)\in [\mathbf{a},\mathbf{b}].
\end{equation}
Then $(\mathbf{e}_{\mathbf{k}}:\mathbf{k}\in\mathbb{Z}^d)$ is an ONB for $L^2[\mathbf{a},\mathbf{b}]$ and the Fourier coefficient of $U\in L^1[\mathbf{a},\mathbf{b}]$ at $\mathbf{k}\in\mathbb{Z}^d$ is   
\begin{equation}\label{Uk}
\widehat{U}[\mathbf{k}]:=\frac{1}{\mathrm{vol}[\mathbf{a},\mathbf{b}]}\int_{[\mathbf{a},\mathbf{b}]}f(\mathbf{x})\overline{\mathbf{e}_\mathbf{k}(\mathbf{x})}\D\mathbf{x}.
\end{equation}

For $\mathbf{L}=(L_\ell)_{\ell=1}^d\in\mathbb{N}^d$, we define the associated uniform grid $\mathcal{G}_\mathbf{L}\subset [\mathbf{a},\mathbf{b}]$ as 
$\mathcal{G}_\mathbf{L}:=\{\mathbf{x}^\mathbf{i}:1\le\mathbf{i}\le \mathbf{L}\}$, 
where $\mathbf{x}^\mathbf{i}:=(x_\ell^{i_\ell})_{\ell}^d$ with $x_\ell^{i_\ell}:=a_\ell+(i_\ell-1)\delta_\ell$ with $\delta_\ell:=\frac{b_\ell-a_\ell}{L_\ell}$, if 
$\mathbf{i}:=(i_1,\ldots,i_d)$ and $\mathbf{1}\le\mathbf{i}\le\mathbf{L}$. 

Suppose $\mathbf{L}\in\mathbb{N}^d$, $\mathbf{k}\in\mathbb{Z}^d$, and $U:[\mathbf{a},\mathbf{b}]\to\mathbb{C}$ is a function. Define 
\begin{equation}\label{UkLd}
\widehat{U}[\mathbf{k};\mathbf{L}]:=\frac{1}{\det(\mathrm{diag}(\mathbf{L}))}\sum_{\mathbf{x}\in\mathcal{G}_\mathbf{L}}U(\mathbf{x})\overline{\mathbf{e}_\mathbf{k}(\mathbf{x})}.
\end{equation}
Next, we show that $\widehat{U}[\mathbf{k};\mathbf{L}]$ given by (\ref{UkLd}) can be considered as numerical approximation of $\widehat{U}[\mathbf{k}]$.
\begin{theorem}\label{UkL}
Let $\mathbf{K}\in\mathbb{N}^d$ and $\mathbf{k}\in\mathbb{Z}^d$. Assume $U\in\mathcal{C}^1[\mathbf{a},\mathbf{b}]$ is a periodic function. Then 
\[
\left|\widehat{U}[\mathbf{k}]-\widehat{U}[\mathbf{k};2\mathbf{K}+1]\right|\le\frac{6d\|\mathbf{a}-\mathbf{b}\|_\infty\|\nabla U\|_\infty}{\pi\min(\mathbf{K})},
\hspace{1cm} {\rm if}\hspace{.2cm}|\mathbf{k}|\le\mathbf{K}.
\] 
\end{theorem}
\begin{proof}
We the claim by induction on $d$. For $d=1$, Proposition \ref{mainP} implies the claim. Assume the claim holds for $\mathbb{R}^{d-1}$.
For $\mathbf{x}:=(x_1,\ldots,x_{d-1},x_d)^T\in\mathbb{R}^d$, let $\underline{\mathbf{x}}\in\mathbb{R}^{d-1}$ be $\underline{\mathbf{x}}:=(x_1,\ldots,x_{d-1})$.
Let $U\in\mathcal{C}^1[\mathbf{a},\mathbf{b}]$ be periodic, $|\mathbf{k}|\le\mathbf{K}$, and $\mathbf{I}:=\prod_{\ell=1}^{d-1}[a_\ell,b_\ell]$. For $x_d\in [a_d,b_d]$, let $U_{x_d}:\mathbf{I}\to\mathbb{C}$ be given by 
$U_{x_d}(\underline{\mathbf{x}}):=U(\mathbf{x}),$ and $u_{\underline{\mathbf{k}}}(x_d):=\widehat{U_{x_d}}[\underline{\mathbf{k}}]$. Then $U_{x_d}\in\mathcal{C}^1(\mathbf{I})$,  $u_{\underline{\mathbf{k}}}'(x_d)=\widehat{(\partial _d U)_{{x_d}}}[\underline{\mathbf{k}}]$, and $|u_{\underline{\mathbf{k}}}'(x_d)|\le\|\partial_dU\|_\infty$. So, $\widehat{U}[\mathbf{k}]=\widehat{u_{\underline{\mathbf{k}}}}[k_d]$, and $
\widehat{U}[\mathbf{k};\mathbf{L}]=\frac{1}{L_d}\sum_{i_d=1}^{L_d}\widehat{U_{x_{d}^{i}}}[\underline{\mathbf{k}};\underline{\mathbf{L}}]\overline{e_{k_d}(x_{d}^{i})}$,
for $\mathbf{k}:=(k_1,\ldots,k_d)^T\in\mathbb{Z}^d$.
Hence, we get 
\begin{align*}
\left|\widehat{u_{\underline{\mathbf{k}}}}[k_d;L_d]-\widehat{U}[\mathbf{k};\mathbf{L}]\right|
&=\left|\frac{1}{L_d}\sum_{i_d=1}^{L_d}u_{\underline{\mathbf{k}}}(x_d^{i})\overline{e_{k_d}(x_{d}^{i})}-\frac{1}{L_d}\sum_{i_d=1}^{L_d}\widehat{U_{x_{d}^{i}}}[\underline{\mathbf{k}};\underline{\mathbf{L}}]\overline{e_{k_d}(x_{d}^{i})}\right|
\\&=\left|\frac{1}{L_d}\sum_{i=1}^{L_d}\widehat{U_{x_d^{i}}}[\underline{\mathbf{k}}]\overline{e_{k_d}(x_{d}^{i})}-\frac{1}{L_d}\sum_{i=1}^{L_d}\widehat{U_{x_{d}^{i}}}[\underline{\mathbf{k}};\underline{\mathbf{L}}]\overline{e_{k_d}(x_{d}^{i})}\right|
\\&\le \frac{1}{L_d}\sum_{i=1}^{L_d}\left|\widehat{U_{x_d^{i}}}[\underline{\mathbf{k}}]-\widehat{U_{x_{d}^{i}}}[\underline{\mathbf{k}};\underline{\mathbf{L}}]\right|
\\&\le\frac{6(d-1)}{L_d}\sum_{i=1}^{L_d}\frac{\|\underline{\mathbf{a}}-\underline{\mathbf{b}}\|_\infty\left\|\nabla U_{x_d^i}\right\|_\infty}{\pi\min(\underline{\mathbf{K}})}\le\frac{6(d-1)\|\underline{\mathbf{a}}-\underline{\mathbf{b}}\|_\infty\|\nabla U\|_\infty}{\pi \min(\underline{\mathbf{K}})}.
\end{align*}
Applying Proposition \ref{mainP}, for function $u_{\underline{\mathbf{k}}}$, we obtain
\begin{align*}
\left|\widehat{u_{\underline{\mathbf{k}}}}[k_d]-\widehat{u_{\underline{\mathbf{k}}}}[k_d;L_d]\right|
&\le\frac{6(b_d-a_d)\|u_{\underline{\mathbf{k}}}'\|_\infty}{\pi K_d}\\&\le\frac{6(b_d-a_d)\|\partial_dU\|_\infty}{\pi K_d}\le\frac{6(b_d-a_d)\|\nabla U\|_\infty}{\pi K_d}.
\end{align*}
Then
\begin{align*}
\left|\widehat{U}[\mathbf{k}]-\widehat{U}[\mathbf{k};\mathbf{L}]\right|
&=\left|\widehat{u_{\underline{\mathbf{k}}}}[k_d]-\widehat{U}[\mathbf{k};\mathbf{L}]\right|
\\&\le \left|\widehat{u_{\underline{\mathbf{k}}}}[k_d]-\widehat{u_{\underline{\mathbf{k}}}}[k_d;L_d]\right|+\left|\widehat{u_{\underline{\mathbf{k}}}}[k_d;L_d]-\widehat{U}[\mathbf{k};\mathbf{L}]\right|
\\&\le \frac{6(b_d-a_d)\left\|\nabla U\right\|_\infty}{\pi K_d}+\frac{6(d-1)\|\underline{\mathbf{a}}-\underline{\mathbf{b}}\|_\infty\|\nabla U\|_\infty}{\pi \min(\underline{\mathbf{K}})}
\\&\le\frac{6\|\mathbf{a}-\mathbf{b}\|_\infty\left\|\nabla U\right\|_\infty}{\pi K_d}+\frac{6(d-1)\|\mathbf{a}-\mathbf{b}\|_\infty\|\nabla U\|_\infty}{\pi \min(\underline{\mathbf{K}})}
\\&\le\frac{6\|\mathbf{a}-\mathbf{b}\|_\infty\left\|\nabla U\right\|_\infty}{\pi\min(\mathbf{K})}+\frac{6(d-1)\|\mathbf{a}-\mathbf{b}\|_\infty\|\nabla U\|_\infty}{\pi \min(\mathbf{K})}
\le\frac{6d\|\mathbf{a}-\mathbf{b}\|_\infty\|\nabla U\|_\infty}{\pi\min(\mathbf{K})}.\qedhere
\end{align*}
\end{proof}

\begin{corollary}
{\it  Suppose $K\in\mathbb{N}$ and $\mathbf{k}\in\mathbb{Z}^d$. Assume $U\in\mathcal{C}^1[\mathbf{a},\mathbf{b}]$ is a periodic function. Then 
\[
\left|\widehat{U}[\mathbf{k}]-\widehat{U}[\mathbf{k};2K+1]\right|\le\frac{6d\|\mathbf{a}-\mathbf{b}\|_\infty\|\nabla U\|_\infty}{\pi K},
\hspace{1cm} {\rm if}\hspace{0.2cm}\|\mathbf{k}\|_\infty\le K.
\] 
}\end{corollary}

\subsection{Zero-padded DFT}
Let $K\in\mathbb{N}$ and $L:=2K+1$. Then 
\begin{equation}\label{tauKk}
\tau_L(k)=\left\{\begin{array}{ll}
k & {\rm if}\ \ \ \ \ \ 0\le k\le K\\
L+k & {\rm if}\ -K\le k\le -1
\end{array}
\right..
\end{equation} 

\begin{proposition}\label{LtoN1D}
{\it Let $K,N\in\mathbb{N}$ and $N\ge L:=2K+1$. Suppose 
$(c_k)_{k=-K}^K\subset\mathbb{C}$ and $\mathbf{x}\in\mathbb{C}^{L}$. Then 
\[
\sum_{k=-K}^Kc_k\mathbf{x}(\tau_L(k)+1)e^{2\pi\ii k(\ell-1)/N}=N\Check{\mathbf{z}}(\ell), \hspace{0.5cm}{\rm for}\hspace{0.2cm}1\le \ell\le N, 
\]
where $\mathbf{z}\in\mathbb{C}^N$ is given by 
\[
\mathbf{z}(n)=\left\{\begin{array}{lll}
c_{n-1}\mathbf{x}(n) & {\rm if}\ \ \ \ \ \ 1\le n\le K+1\\
0 & {\rm if}\ K+2\le n\le N-K\\
c_{n-1-N}\mathbf{x}(L+n-N) & {\rm if}\ \ \  N-K+1\le n\le N
\end{array}
\right..
\]
}\end{proposition}
\begin{proof}
Assume that $\mathbf{x}\in\mathbb{C}^L$. Let $1\le \ell\le N$ be given. Then using (\ref{tauKk}) we obtain 
\begin{align*}
\sum_{k=-K}^Kc_k\mathbf{x}(\tau_L(k)+1)e^{\frac{2\pi\ii k(\ell-1)}{N}}
&=\sum_{k=-K}^{-1}c_k\mathbf{x}(\tau_L(k)+1)e^{\frac{2\pi\ii k(\ell-1)}{N}}+\sum_{k=0}^Kc_k\mathbf{x}(\tau_L(k)+1)e^{\frac{2\pi\ii k(\ell-1)}{N}}
\\&=\sum_{k=-K}^{-1}c_k\mathbf{x}(L+k+1)e^{\frac{2\pi\ii k(\ell-1)}{N}}+\sum_{k=0}^Kc_k\mathbf{x}(k+1)e^{\frac{2\pi\ii k(\ell-1)}{N}}
\\&=\sum_{k=N-K+1}^{N}c_{k-1-N}\mathbf{x}(L+k-N)e^{\frac{2\pi\ii(k-1)(\ell-1)}{N}}+\sum_{k=1}^{K+1}c_{k-1}\mathbf{x}(k)e^{\frac{2\pi\ii(k-1)(\ell-1)}{N}}
\\&=\sum_{n=1}^N\mathbf{z}(n)e^{\frac{2\pi\ii(n-1)(\ell-1)}{N}}=N\Check{\mathbf{z}}(\ell).
\end{align*}
\end{proof}
\begin{corollary}\label{LtoN1D+1}
{\it Let $K,N\in\mathbb{N}$ and $N\ge L:=2K+1$. Suppose that 
$\mathbf{x}\in\mathbb{C}^{L}$ and $1\le \ell\le N$. Then 
\[
\sum_{k=-K}^K\mathbf{x}(\tau_L(k)+1)e^{2\pi\ii k(\ell-1)/N}=N\Check{\widetilde{\mathbf{x}}}(\ell),
\]
where $\widetilde{\mathbf{x}}\in\mathbb{C}^N$ is given by 
\[
\widetilde{\mathbf{x}}(n):=\left\{\begin{array}{lll}
\mathbf{x}(n) & {\rm if}\ \ \ \ \ \ 1\le n\le K+1\\
0 & {\rm if}\ K+2\le n\le N-K\\
\mathbf{x}(L+n-N) & {\rm if}\ \ \  N-K+1\le n\le N
\end{array}
\right..
\]
}\end{corollary}

\begin{corollary}\label{LtoN1D-1}
{\it Let $K,N\in\mathbb{N}$ and $N\ge L:=2K+1$. Suppose that 
$\mathbf{x}\in\mathbb{C}^{L}$ and $1\le \ell\le N$. Then 
\[
\sum_{k=-K}^K(-1)^k\mathbf{x}(\tau_L(k)+1)e^{2\pi\ii k(\ell-1)/N}=N\Check{\widetilde{\pm\mathbf{x}}}(\ell),
\]
where $\widetilde{\pm\mathbf{x}}\in\mathbb{C}^N$ is given by 
\[
\widetilde{\pm\mathbf{x}}(n):=\left\{\begin{array}{lll}
(-1)^{n-1}\mathbf{x}(n) & {\rm if}\ \ \ \ \ \ 1\le n\le K+1\\
0 & {\rm if}\ K+2\le n\le N-K\\
(-1)^{n-N-1}\mathbf{x}(L+n-N) & {\rm if}\ \ \  N-K+1\le n\le N
\end{array}
\right..
\]
}\end{corollary}

\vspace{0.5cm}

{\bf Concluding remarks.}
This paper discussed a fast and accurate numerical scheme for approximating functions on the right coset space $\mathbb{Z}^2\backslash SE(2)$ and convolutions of $SE(2)$ on the right coset space $\mathbb{Z}^2\backslash SE(2)$. 

In contrast to standard group-theoretic Fourier/Plancherel theory of the unimodular non-Abelian group $SE(2)$, where the spectrum of functions with compact support is neither compactly supported nor discrete, in this paper an alternative numerical approximation method is developed in which Fourier reconstruction formulas with discrete spectrum $\mathbb{Z}^3$ are constructed. The constructive method and the developed numerical scheme have several advantageous. To begin with, the Fourier reconstruction formula is by nature discrete. In addition, due to the structure of the numerical scheme it is efficient and can be implemented using fast Fourier algorithms (FFT).

The paper investigated compatibility of the fast numerical scheme for $SE(2)$-convolutions with functions which are radial in translations. In this direction, the convolutional finite Fourier series are developed to approximate $SE(2)$-convolution of arbitrary functions with functions which are radial in translations. The paper is concluded by discussing capability of the numerical scheme to develop fast algorithms for approximating multiple convolutions with functions with are radial in translations. \\

{\bf Acknowledgments.}
This research is supported by University of Delaware Startup grants. The findings and opinions expressed here are only those of the authors, and not of the funding agency.
\bibliographystyle{plain}
\bibliography{FastRadialCNV-SE2}

\begin{thebibliography}{10}

\bibitem{Av.Co.Do.El.Is}
A.~Averbuch, R.~R. Coifman, D.~L. Donoho, M.~Elad, and M.~Israeli.
\newblock Fast and accurate polar {F}ourier transform.
\newblock {\em Appl. Comput. Harmon. Anal.}, 21(2):145--167, 2006.

\bibitem{Ba.Gi}
D.~Barbieri and G.~Citti.
\newblock Reproducing kernel {H}ilbert spaces of {CR} functions for the
  {E}uclidean motion group.
\newblock {\em Anal. Appl. (Singap.)}, 13(3):331--346, 2015.

\bibitem{imme}
G.S. Chirikjian and I.~Ebert-Uphoff.
\newblock Numerical convolution on the euclidean group with applications to
  workspace generation.
\newblock {\em IEEE Transactions on Robotics and Automation}, 14(1):123--136,
  1998.

\bibitem{PI5}
G.S. Chirikjian and A.B. Kyatkin.
\newblock {\em Harmonic Analysis for Engineers and Applied Scientists: Updated
  and Expanded Edition}.
\newblock Courier Dover Publications, 2016.

\bibitem{GSC.JFAA.2000}
G.S. Chirikjian and Alexander~B. Kyatkin.
\newblock An operational calculus for the {E}uclidean motion group with
  applications in robotics and polymer science.
\newblock {\em J. Fourier Anal. Appl.}, 6(6):583--606, 2000.

\bibitem{citti+sarti}
G.~Citti and A.~Sarti.
\newblock A cortical based model of perceptual completion in the
  roto-translation space.
\newblock {\em Journal of Mathematical Imaging and Vision}, 24(3):307--326, May
  2006.

\bibitem{citti+sarti1}
Giovanna Citti and Alessandro Sarti.
\newblock {\em Models of the Visual Cortex in Lie Groups}, pages 1--55.
\newblock Springer Basel, Basel, 2015.

\bibitem{Amit}
Mihai Cucuringu, Yaron Lipman, and Amit Singer.
\newblock Sensor network localization by eigenvector synchronization over the
  euclidean group.
\newblock {\em ACM Trans. Sen. Netw.}, 8(3), August 2012.

\bibitem{Du.Entropy}
Remco Duits, Erik~J. Bekkers, and Alexey Mashtakov.
\newblock Fourier transform on the homogeneous space of 3d positions and
  orientations for exact solutions to linear pdes.
\newblock {\em Entropy}, 21(1), 2019.

\bibitem{duits1}
Remco Duits and Erik Franken.
\newblock Left-invariant parabolic evolutions on {${\rm SE}(2)$} and contour
  enhancement via invertible orientation scores {P}art {I}:linear
  left-invariant diffusion equations on {${\rm SE}(2)$}.
\newblock {\em Quart.Appl.Math.}, 68(2):255--292, 2010.

\bibitem{Du}
Remco Duits and Markus van Almsick.
\newblock The explicit solutions of linear left-invariant second order
  stochastic evolution equations on the 2{D} {E}uclidean motion group.
\newblock {\em Quart. Appl. Math.}, 66(1):27--67, 2008.

\bibitem{FFC01}
Charles~L. Epstein.
\newblock How well does the finite {F}ourier transform approximate the
  {F}ourier transform?
\newblock {\em Comm. Pure Appl. Math.}, 58(10):1421--1435, 2005.

\bibitem{Fe.Ku.Po}
Markus Fenn, Stefan Kunis, and Daniel Potts.
\newblock On the computation of the polar {FFT}.
\newblock {\em Appl. Comput. Harmon. Anal.}, 22(2):257--263, 2007.

\bibitem{FollH}
Gerald~B. Folland.
\newblock {\em A course in abstract harmonic analysis}.
\newblock Studies in Advanced Mathematics. CRC Press, Boca Raton, FL, 1995.

\bibitem{AGHF.BMMSS}
Arash Ghaani~Farashahi.
\newblock Convolution and involution on function spaces of homogeneous spaces.
\newblock {\em Bull. Malays. Math. Sci. Soc. (2)}, 36(4):1109--1122, 2013.

\bibitem{AGHF.IJM}
Arash Ghaani~Farashahi.
\newblock Abstract convolution function algebras over homogeneous spaces of
  compact groups.
\newblock {\em Illinois J. Math.}, 59(4):1025--1042, 2015.

\bibitem{AGHF.JAuMS}
Arash Ghaani~Farashahi.
\newblock Abstract harmonic analysis of relative convolutions over canonical
  homogeneous spaces of semidirect product groups.
\newblock {\em J. Aust. Math. Soc.}, 101(2):171--187, 2016.

\bibitem{AGHF.JKMS}
Arash Ghaani~Farashahi.
\newblock Abstract relative {F}ourier transforms over canonical homogeneous
  spaces of semi-direct product groups with abelian normal factor.
\newblock {\em J. Korean Math. Soc.}, 54(1):117--139, 2017.

\bibitem{AGHF.BBMSS}
Arash Ghaani~Farashahi.
\newblock Abstract {B}anach convolution function modules over coset spaces of
  compact subgroups in locally compact groups.
\newblock {\em Bull. Braz. Math. Soc. (N.S.)}, 53(2):357--377, 2022.

\bibitem{AGHF.GSC.PAMQ}
Arash Ghaani~Farashahi and Gregory Chirikjian.
\newblock Abstract noncommutative {F}ourier series on {$\Gamma\backslash
  SE(2)$}.
\newblock {\em Pure Appl. Math. Q.}, 18(1):71--100, 2022.

\bibitem{AGHF.GSC.JAT}
Arash Ghaani~Farashahi and Gregory~S. Chirikjian.
\newblock Fourier-{Z}ernike series of compactly supported convolutions on
  {$SE(2)$}.
\newblock {\em J. Approx. Theory}, 271:Paper No. 105621, 16, 2021.

\bibitem{HR2}
Edwin Hewitt and Kenneth~A. Ross.
\newblock {\em Abstract harmonic analysis. {V}ol. {II}: {S}tructure and
  analysis for compact groups. {A}nalysis on locally compact {A}belian groups}.
\newblock Die Grundlehren der mathematischen Wissenschaften, Band 152.
  Springer-Verlag, New York-Berlin, 1970.

\bibitem{HR1}
Edwin Hewitt and Kenneth~A. Ross.
\newblock {\em Abstract harmonic analysis. {V}ol. {I}}, volume 115 of {\em
  Grundlehren der Mathematischen Wissenschaften [Fundamental Principles of
  Mathematical Sciences]}.
\newblock Springer-Verlag, Berlin-New York, second edition, 1979.
\newblock Structure of topological groups, integration theory, group
  representations.

\bibitem{Kisil.Adv}
Vladimir~V. Kisil.
\newblock Relative convolutions. {I}. {P}roperties and applications.
\newblock {\em Adv. Math.}, 147(1):35--73, 1999.

\bibitem{Kisil.Book}
Vladimir~V. Kisil.
\newblock {\em Geometry of {M}\"{o}bius transformations}.
\newblock Imperial College Press, London, 2012.
\newblock Elliptic, parabolic and hyperbolic actions of
  ${{\rm{S}}L}_2$(${\Bbb{R}}$), With 1 DVD-ROM.

\bibitem{Kisil.JPA}
Vladimir~V. Kisil.
\newblock Operator covariant transform and local principle.
\newblock {\em J. Phys. A}, 45(24):244022, 10, 2012.

\bibitem{Kisil.BJMA}
Vladimir~V. Kisil.
\newblock Calculus of operators: covariant transform and relative convolutions.
\newblock {\em Banach J. Math. Anal}, 8(2):156--184, 2014.

\bibitem{Kondor}
Risi Kondor and Shubhendu Trivedi.
\newblock On the generalization of equivariance and convolution in neural
  networks to the action of compact groups.
\newblock In Jennifer Dy and Andreas Krause, editors, {\em Proceedings of the
  35th International Conference on Machine Learning}, volume~80 of {\em
  Proceedings of Machine Learning Research}, pages 2747--2755. PMLR, 10--15 Jul
  2018.

\bibitem{FFTSOd}
Peter~J. Kostelec and Daniel~N. Rockmore.
\newblock F{FT}s on the rotation group.
\newblock {\em J. Fourier Anal. Appl.}, 14(2):145--179, 2008.

\bibitem{Bana.Wolf.Chir}
Wolfe K.C. Mashner M.J. Chirikjian~G.S. Long, A.W.
\newblock The banana distribution is gaussian: A localization study with
  exponential coordinates.
\newblock In {\em Proc. Robotics: Science and Systems, Sydney, NSW, Australia,
  July 09 – July 13, 2012}.

\bibitem{DR.SIAM.JC.2007}
Cristopher Moore, Daniel Rockmore, Alexander Russell, and Leonard~J. Schulman.
\newblock The power of strong {F}ourier sampling: quantum algorithms for affine
  groups and hidden shifts.
\newblock {\em SIAM J. Comput.}, 37(3):938--958, 2007.

\bibitem{HR.JS}
Hans Reiter and Jan~D. Stegeman.
\newblock {\em Classical harmonic analysis and locally compact groups},
  volume~22 of {\em London Mathematical Society Monographs. New Series}.
\newblock The Clarendon Press, Oxford University Press, New York, second
  edition, 2000.

\bibitem{Jack}
Theodore~J. Rivlin.
\newblock {\em An introduction to the approximation of functions}.
\newblock Blaisdell Publishing Co. [Ginn and Co.], Waltham, Mass.-Toronto,
  Ont.-London, 1969.

\bibitem{DR.ACHA.1995}
Daniel Rockmore.
\newblock Fast {F}ourier transforms for wreath products.
\newblock {\em Appl. Comput. Harmon. Anal.}, 2(3):279--292, 1995.

\bibitem{xstal}
Bernard Shiffman, Shengnan Lyu, and Gregory~S. Chirikjian.
\newblock {Mathematical aspects of molecular replacement. V. Isolating feasible
  regions in motion spaces}.
\newblock {\em Acta Crystallographica Section A}, 76(2):145--162, Mar 2020.

\bibitem{Duits.PDE-G-CNN}
Bart M.~N. Smets, Jim Portegies, Erik~J. Bekkers, and Remco Duits.
\newblock Pde-based group equivariant convolutional neural networks.
\newblock {\em Journal of Mathematical Imaging and Vision}, 65(1):209--239, Jan
  2023.

\bibitem{YY.ITSF.2007}
Can~Evren Yarman and Birsen Yazici.
\newblock Euclidean motion group representations and the singular value
  decomposition of the {R}adon transform.
\newblock {\em Integral Transforms Spec. Funct.}, 18(1-2):59--76, 2007.

\bibitem{YY.IPI.2007}
Can~Evren Yarman and Birsen Yazici.
\newblock A new exact inversion method for exponential {R}adon transform using
  the harmonic analysis of the {E}uclidean motion group.
\newblock {\em Inverse Probl. Imaging}, 1(3):457--479, 2007.

\bibitem{duits}
Jiong Zhang, Behdad Dashtbozorg, Erik Bekkers, Josien P.~W. Pluim, Remco Duits,
  and Bart~M. ter Haar~Romeny.
\newblock Robust retinal vessel segmentation via locally adaptive derivative
  frames in orientation scores.
\newblock {\em IEEE Transactions on Medical Imaging}, 35(12):2631--2644, 2016.

\end{thebibliography}
\end{document}